\documentclass[twoside,11pt]{article}

\usepackage{amsmath}
\usepackage{amsthm}
\usepackage[ruled,vlined]{algorithm2e}
\usepackage{jmlr2e}
\usepackage{subcaption}
\usepackage{multirow}
\usepackage{array,multirow}
\newtheorem{condition}{Condition}
\newtheorem{example}{Example} 
\newtheorem{theorem}{Theorem}
\newtheorem{lemma}{Lemma} 
\newtheorem{proposition}{Proposition} 
\newtheorem{remark}{Remark}

\usepackage[dvipsnames]{xcolor}
\newcommand{\yc}[1]{{\color{black} #1}}
\usepackage{xcolor}

\newcommand\numberthis{\addtocounter{equation}{1}\tag{\theequation}}

\ShortHeadings{Inference for Noisy
Binary Matrix Completion}{Chen, Li, Ouyang and Xu}
\firstpageno{1}

\begin{document}
\title{Statistical Inference for
Noisy Incomplete \yc{Binary} Matrix}

\author{\name Yunxiao Chen \email y.chen186@lse.ac.uk  \\
       \addr Department of Statistics\\ 
       London School of Economics and Political Science\\
        London WC2A 2AE, UK
       \AND
       \name Chengcheng Li \email lccvic@umich.edu \\
        \name Jing Ouyang \email jingoy@umich.edu  \\
        \name Gongjun Xu \email gongjun@umich.edu \\
       \addr Department of Statistics\\
       University of Michigan\\
       Ann Arbor, MI 48109, USA
     }
\editor{Ali Shojaie}

\maketitle

\begin{abstract}
We consider the statistical inference for noisy incomplete \yc{binary (or 1-bit)} matrix.
Despite the importance of uncertainty quantification to matrix completion, most of the categorical matrix completion literature focuses on point estimation and prediction. 
This paper moves one step further toward the statistical inference for \yc{binary} matrix completion. 
Under a popular  nonlinear factor analysis model, we obtain a point estimator and derive its asymptotic normality.  
Moreover, our analysis adopts a flexible missing-entry design that does not require a random sampling scheme as required by most of the existing asymptotic results for matrix completion.
Under reasonable conditions, the proposed estimator is statistically efficient and optimal in the sense that the Cramer-Rao lower bound is achieved asymptotically for the model parameters. 
Two applications are considered, including (1) linking  two forms of an educational test
 and  (2) linking the roll call voting records from multiple years in the United States Senate. 
The first application enables the comparison between examinees who took different test forms, and 
the second application allows us to compare the liberal-conservativeness of senators who did not serve in the Senate at the same time. 
\end{abstract}


\section{Introduction}\label{sec-intro}
Noisy low-rank matrix completion is concerned with the recovery of a low-rank matrix when only a fraction of noisy entries are observed. 
This topic has received much attention 
as a result of its vast applications in practical contexts such as collaborative filtering \citep{goldberg1992}, system identification \citep{liu2010} and sensor localization \citep{biswas2006}. While the majority of the literature considers the completion of real-valued observations 
\citep{candes2009exact, candes2010power,  keshavan2010matrix, koltchinskii2011, negahban2012,   chen2020noisy}, many practical problems involve 
categorical-valued matrices, such as the 
famous Netflix challenge. Several works have been done on matrix completion involving categorical variables, including 
\citet{davenport20141} and \citet{bhaskar20151} for 1-bit matrix \yc{whose  entries take binary values}, and 
\citet{klopp2015adaptive} and \citet{bhaskar2016probabilistic}
 for categorical matrix, 
and \citet{chen2022determining} for matrix of binary, count, and continuous variables. 
In these works, low-dimensional nonlinear probabilistic models are assumed. 



Despite the importance of uncertainty quantification to matrix completion, 
most of the matrix completion literature focuses on point estimation and prediction, while statistical inference has received attention only recently. Specifically,  \citet{chen2019inference} and \citet{xia2021statistical} considered statistical inference under the linear models and derived asymptotic normality results. The statistical inference for categorical matrices is more challenging due to the involvement of nonlinear models.  To our best knowledge, no work has been done to provide statistical inference for the completion of categorical matrices. 
In addition to nonlinearity, another challenge in modern theoretical analysis of matrix completion concerns  the double asymptotic regime where both the numbers of rows and columns are allowed to grow to infinity. Under this asymptotic regime, both the dimension of the parameter space and the number of observable entries grow with the numbers of rows and columns. However,  existing theory on the statistical inference for diverging number of parameters 
\citep{portnoy1988asymptotic,he2000parameters,wang2011gee} 
is not directly applicable, as the dimension of the parameter space in the current problem grows faster  than that is typically needed for asymptotic normality; see Section~\ref{sec-asym-properties} for further discussions. 


In this paper, we move one step further toward statistical inference for the completion of categorical matrices. 
Specifically, we consider the inference for binary matrix completion under a unidimensional nonlinear factor analysis model with the logit link.
Such a nonlinear factor model is one of the most popular models for multivariate binary data, and it has received much attention from the theoretical perspective \citep{andersen1970asymptotic,haberman1977maximum,lindsay1991semiparametric,rice2004equivalence}, as well as wide applications in various areas, including educational testing \citep{van2013handbook}, word acquisition analysis \citep{kidwell2011statistical}, syntactic comprehension \citep{gutman2011rasch}, and analysis of health outcomes  \citep{hagquist2017recent}. 
It is also referred to as the Rasch model \citep{rasch1960studies} in the psychometrics literature. 
Despite the popularity and extensive research of the model, its use for binary matrix completion and related statistical inferences for the latent factors and model parameters have not been explored.  
The considered nonlinear factor model  is also closely related to the Bradley-Terry model \citep{bradley1952rank,simons1999asymptotics, han2020asymptotic, gao2021uncertainty} for directed random graphs  and 
the $\beta$-model \citep{chatterjee2011random, yan2011wilks, rinaldo2013maximum} for undirected random graphs. 
In fact, the considered model can be viewed as a Bradley-Terry model or $\beta$-model for bipartite graphs \citep{rinaldo2013maximum}. 
{However, the asymptotic analysis of bipartite graphs
concerns a rectangle matrix which involves two diverging indices -- the numbers of rows and columns of the data matrix, while a standard random graph concerns a square matrix which only involves one diverging index. Thus, 
more refined analysis is needed in the asymptotic analysis of bipartite graphs, in order to approximate the asymptotic variance of the model parameters and derive  conditions under which consistency and asymptotic normality results hold. }



Specifically, we introduce a likelihood-based estimator under the nonlinear factor analysis model for binary matrix completion. Under a very flexible missing-entry setting that  does not require a random sampling scheme, asymptotic normality results are established that allow us to draw statistical inferences. 
%
These results  suggest that our estimator is asymptotically efficient and optimal, 
in the sense that the Cramer-Rao lower bound is achieved for model parameters. 
The proposed method and theory are applied to two real-world problems, including (1) linking  two forms of a college admission test that have common items and (2)  linking the voting records from multiple years in the United States Senate. In the first application, the proposed method allows us to answer the question ``for examinees A and B who took different test forms, would examinee A perform significantly better than examinee B if they had taken the same test form?". In the second application, it can answer the questions such as  ``Is Republican senator Marco Rubio significantly more conservative than Republican senator Judd Gregg?".
Note that Marco Rubio and Judd Gregg had not served in the United States Senate at the same time. 
We point out that the entry missingness in these applications does not satisfy  the commonly assumed random sampling schemes for matrix completion. 

The rest of the paper is organized as follows. In Section~\ref{sec-parameter-est}, we introduce the considered factor model and discuss its application to binary matrix completion. In Section~\ref{sec-asym-properties}, we establish the asymptotic normality 
for the maximum likelihood estimator. A simulation study is given in Section~\ref{sec-simulation}, and two real-data applications are presented in Section~\ref{sec-real-data}. We conclude with discussions on the limitations of the current work and future directions in Section~\ref{sec-discussion}. All the proofs for the theoretical results developed in the article and additional real-data application results are included in the appendices. The R code for our numerical experiments can be found in {https://github.com/Austinlccvic/A-Note-on-Statistical-Inference-for-Noisy-Incomplete-1-Bit-Matrix}.

{Throughout the paper, we adopt the following notations. For positive sequences $\{a_n\}$ and $\{b_n\}$, we denote $a_n \lesssim b_n$ if there exists a constant $C>0$ that $a_n \leq C b_n$ for all $n$. We denote $a_n \asymp b_n$ if $a_n \lesssim b_n$ and $b_n \lesssim a_n$. We denote $a_n \ll b_n$ if $b_n/a_n \rightarrow \infty$ as $n \rightarrow \infty$.} 

\section{Model and Estimation}\label{sec-parameter-est}

Let $Y$ be a binary (or 1-bit) matrix with $N$ rows and $J$ columns and $Y_{ij} \in \{0,1\}$ be the entries of $Y$, $i = 1, ..., N$, and $j = 1, ..., J$. Some entries of $Y$ are not observable. We use $z_{ij}$ to indicate the missing status of entry $Y_{ij}$, where $z_{ij} = 1$ indicates that $Y_{ij}$ is observed and $z_{ij} = 0$ otherwise. We let $Z=(z_{ij})_{N\times J}$ be the indicator matrix for data missingness. The main goal of binary matrix completion is to estimate 
$E(Y_{ij}\vert z_{ij} = 0)$.  

This problem is typically tackled under a probabilistic model \citep[see e.g.,][]{cai2013max, davenport20141,bhaskar20151,chen2022determining}, which assumes that $Y_{ij}$, $i = 1, ..., N$, $j = 1, ..., J$, are independent Bernoulli random variables, with success probability 
$\exp(m_{ij})/\{1 + \exp(m_{ij})\}$ or $\Phi(m_{ij})$, where $m_{ij}$ is a real-valued parameter and
$\Phi$ is the cumulative distribution function of the standard normal distribution. It is further assumed that the matrix $M = (m_{ij})_{N\times J}$ is either exactly   or approximately low-rank, where the approximate low-rankness is measured by the nuclear norm of $M$. Finally,  a random sampling scheme is typically assumed for $z_{ij}$. For example, \citet{davenport20141} considered a uniform sampling scheme where $z_{ij}$ are independent and identically distributed (i.i.d.)  Bernoulli random variables and \citet{cai2013max} considered a non-uniform sampling scheme. Under such a random sampling scheme, $Z$ and $Y$ are assumed to be independent, and thus, data missingness is ignorable \yc{in the sense that under suitable conditions, $M$ can be consistently estimated by maximizing the likelihood function  
for $M$ satisfying certain exactly or approximately low-rank constraints.}

It is of interest to draw statistical inferences on linear forms of $M$, including the inference of individual entries of $M$. This is a challenging problem under the above general setting for binary matrix completion, largely due to the presence of a non-linear link function. In particular, the existing results on the inference for matrix completion as established in  \cite{xia2021statistical} and \citet{chen2019inference}   are  under a linear model that observes $m_{ij} + \epsilon_{ij}$ for the non-missing entries, where $\epsilon_{ij}$ are mean-zero independent errors. 
Their analyses cannot be directly applied to non-linear models. 

As the first inference work of binary matrix completion with non-linear models, we start with a basic setting in which we assume the success probability takes a logistic form of $M$ and each $m_{ij}$ depends on a row effect and a column effect only.
Asymptotic normality results are then established for the inference of $M$. Specifically, this model assumes that
\begin{itemize}
    \item[(1)] given $M$, $Y_{ij}$, $i = 1, ..., N$, $j = 1, ..., J$, are independent Bernoulli random variables whose distributions do not depend on the missing indicators in $Z$, 
    \item[(2)] the success probability for 
$Y_{ij}$ is assumed to be $\exp(m_{ij})/\{1 + \exp(m_{ij})\}$ that follows a logistic link, 
\item[(3)] 
$M$ has the model parameterization that $m_{ij} = \theta_i - \beta_j$. 
\end{itemize}
\yc{This model is typically referred to as the Rasch model, one of the most popular item response theory models \citep{embretson2013item} to model item-level response data in educational testing  and psychological measurement. See Example~\ref{example:linking} below for the interpretation of $\theta_i$ and $\beta_j$ in educational testing.}
In the rest,  
$\theta_i$ and $\beta_j$ will be referred to as the row and column parameters, respectively. 
This parameterization allows the success probability of each entry to depend on both a row effect and a column effect.
We now introduce two real-world applications and discuss the interpretations of the row and column parameters in these applications. 

\begin{example}\label{example:linking}
In educational testing, each row of the data matrix represents an examinee, and each column represents an item (i.e., an exam question). Each binary entry $Y_{ij}$ records whether examinee $i$ correctly answers item $j$. The row parameter $\theta_i$ is interpreted as the ability of examinee $i$, which is an individual-specific latent factor. The column parameter $\beta_j$ is interpreted as the difficulty of item $j$. The probability of correctly answering an item increases with one's ability $\theta_i$ and decreases with the difficulty level $\beta_j$ of the item. 

In Section~\ref{sec-real-data1}, we apply the considered model to link two forms of an educational test, an important practical issue in educational assessment \citep{kolen2014test}. That is, consider two groups of examinees taking two different forms 
of an educational test, where the two forms share some common items but not all, resulting in missingness of the data matrix. As the two test forms may have different difficulty levels, it is usually not  fair to directly compare the total scores of two students who take different forms. The proposed method allows us to compare examinees' performance as if they had taken the same test form and to also quantify the estimation uncertainty. 


\end{example}

\begin{example}\label{example:voting}
Consider senators' roll call voting records in the United States Senate. In this application, each row of the data matrix corresponds to a senator, and each column corresponds to a bill voted in the Senate. Each binary response $Y_{ij}$ records 
whether the senator voted for or against the bill. It has been well recognized in
the political science literature \citep{poole1991dimensionalizing,poole1991patterns} that senate
voting behavior is essentially unidimensional, though slightly different latent variable models
are used in that literature. That is, it is believed that senators' voting behavior is driven by a unidimensional latent factor, often interpreted as the conservative-liberal political ideology. Moreover, it is a consensus that Republican senators tend to lie on the conservative side of the factor, and Democratic senators tend to lie on the liberal side. However, there are sometimes a very small number of exceptions. 
To apply our method to senators’  roll  call  voting  records, we pre-process the data as follows. If bill $j$ is more supported by
the Republican party than the Democratic party and senator $i$ voted for the bill, then we let $Y_{ij} = 1$. 
If bill $j$ is more supported by 
the Democratic party and 
senator $i$ voted against the bill, we let $Y_{ij} = 1$. Otherwise, $Y_{ij} = 0$. More details about this data pre-processing can be found in Section~\ref{sec-real-data}. 
Under the considered model, the row parameter may be interpreted as the conservativeness score of senator $i$. That is, the higher the conservativeness score of a senator, the higher chance for him/her to support a bill favored by the Republican party and to vote against a bill favored by the Democratic party. 
The column parameter characterizes the bill effect. 

In Section~\ref{sec-real-data2}, we apply the model to link the roll call voting records from multiple years, where different senators have different terms in the Senate, resulting in the missingness of the data matrix. The model allows us to compare senators in terms of their conservative-liberal political ideology, even if they have not served in the Senate at the same time.

\end{example}

As mentioned previously, the considered nonlinear factor model can be viewed as a   Bradley-Terry model \citep{bradley1952rank} for directed graphs that is commonly used for modeling pairwise comparisons. 
In Remark~\ref{rmk:BT} below, we discuss this connection and explain the reason why the existing results, such as \citet{han2020asymptotic}, do not apply to the current setting. 

\begin{remark}\label{rmk:BT}
Data $Y$ under our model setting can be viewed as a bipartite graph with $N+J$ nodes. Its adjacency matrix takes the form 
\begin{equation}\label{eq:bipartite}
\left(\begin{array}{cc}
    \text{NA}_{N,N} & Y  \\
    (1_{N,J}-Y)^T & \text{NA}_{J,J}
\end{array}\right),
\end{equation}
where $\text{NA}_{N,N}$ and $\text{NA}_{J,J}$ are two matrices whose entries are missing and 
$1_{N,J}$ is a matrix with all entries being 1. 
We let the value of $1 - Y_{ij}$ be missing if $Y_{ij}$ is missing (i.e., $z_{ij} = 0$). Such a directed graph can be modeled by the Bradley-Terry model; see \citet{bradley1952rank}. 
In \citet{han2020asymptotic}, asymptotic normality results are established for  $n$-by-$n$ adjacency matrices  that follow the Bradley-Terry model when the graph size $n$ grows to infinity. However, \citet{han2020asymptotic} only consider a uniformly missing setting. That is, the probability that the edges between two nodes are missing is assumed to be the same for all pairs of nodes. This assumption is not satisfied for the adjacency matrix \eqref{eq:bipartite}, due to the two missing matrices on the diagonal. In fact, the asymptotic analysis under the current setting is more involved due to the need to simultaneously consider two indices $N$ and $J$ and the increased complexity in approximating the asymptotic variance of model parameters.  

\end{remark}

Given data $\{Y_{ij}:z_{ij} = 1, i = 1, ..., N, j = 1, ..., J\}$, the log-likelihood function for parameters $\theta = (\theta_1, ..., \theta_N)^T$ and $\beta = (\beta_1, ..., \beta_J)^T$ takes the form 
  \begin{equation}\label{eq: likelihood-M}
  l(\theta, \beta) = \sum_{i,j:z_{ij}=1} \left[ Y_{ij}(\theta_i - \beta_j)-\log\{1+\exp(\theta_i - \beta_j)\}\right].
  \end{equation}
The identifiability of  parameters $\theta$ and $\beta$ is subject to a location shift. That is, the distribution of data remains unchanged if we add a common constant to all the $\theta_i$ and $\beta_j$, as the likelihood function in \eqref{eq: likelihood-M} only depends on all the differences $\theta_i - \beta_j$. To avoid ambiguity, we require $\sum_{i=1}^N\theta_i=0$ in the rest.
We point out that this requirement does not play a role when we draw inferences about any linear form of $M$ as 
the location shift of $\theta$ and $\beta$ does not affect the value of $M$, but it does involve when we draw inference on $\theta$ or $\beta$. 
We estimate $\theta$ and $\beta$ by the maximum likelihood estimator  
\begin{equation}\label{eq:mle}
(\hat \theta, \hat \beta) = \arg\min_{\theta, \beta}~ -l(\theta, \beta), s.t., \sum_{i=1}^N\theta_i=0. 
\end{equation}
The maximum likelihood estimator of $\theta$ and $\beta$ further leads to the maximum likelihood estimator of $M$, $\hat m_{ij} = \hat \theta_i - \hat \beta_j$. \yc{As shown in Theorem~\ref{thm-existence} below,  under mild conditions, with probability tending to 1,  optimization problem \eqref{eq:mle} has a unique solution in $\mathbb R^{N+J}$. 
We solve the optimization problem by a projected gradient descent algorithm
which is summarized in Algorithm \ref{alg} below. We define $\mbox{proj}(x)$ as a projection operator, mapping a vector in $\mathbb R^{N}$ to $\{\theta \in \mathbb R^{N}: \sum_{i=1}^N\theta_i=0\}$. This projection operator has a closed form $\mbox{proj}(x) = (x_1 - \bar x, x_2 - \bar x, ..., x_N- \bar x)$, where $\bar x = (\sum_{i=1}^N x_i)/N$. 
}

 
\begin{algorithm}[h]
\SetAlgoLined
\caption{Projected Gradient Descent Algorithm}
\KwIn{Partially observed data matrix $Y$, learning rates $\gamma_1$ and $\gamma_2$, tolerance $\epsilon$, and initial values $\theta^{(1)}=(\theta_1^{(1)}, ..., \theta_N^{(1)})^T$ and $\beta^{(1)}= (\beta_1^{(1)}, ..., \beta_J^{(1)})^T$.}


Initialize $l^{(0)}=-\infty$ and $l^{(1)}=l(\theta^{(1)},\beta^{(1)})$, and iteration number $t=1$\;

 \While{$(|l^{(t)} - l^{(t-1)}| > \epsilon)$}{
$t=t+1$\;
$\theta^{(t)} = \mbox{proj}(\theta^{(t-1)} + \gamma_1 \frac{\partial l(\theta, \beta^{(t-1)})}{\partial \theta}\vert_{\theta = \theta^{(t-1)}})$\; 
$\beta^{(t)} =\beta^{(t-1)} + \gamma_2 \frac{\partial l(\theta^{(t-1)}, \beta)}{\partial \beta}\vert_{\beta = \beta^{(t-1)}}$\;
     $l^{(t)}=l(\theta^{(t)},\beta^{(t)})$\;
}
\KwOut{$(\theta^{(I)}, \beta^{(I)})$ where $I$ is the last iteration number.}
\label{alg}
\end{algorithm}
 
\yc{The computational complexity in each iteration is $O(\sum_{i=1}^N\sum_{j=1}^J z_{ij})$. 
{It is easy to check that both the objective function and the constraint are convex. Because each $-l_{ij} (\theta_i, \beta_j)$ is convex, the objective function $-l(\theta, \beta) = \sum_{i,j:z_{ij}=1} -l_{ij}(\theta_i, \beta_j)$ with the constraint $  \sum_{i=1}^N \theta_i=0$ is also  convex   \citep{boyd2004convex}. Specifically, the Hessian matrix of the objective function is a $(N+J) \times (N+J)$ positive semidefinite matrix with the only non-zero entries 
\begin{align*}
&-\frac{\partial^2 l(\theta; \beta)}{\partial \theta_i^2}= \sum_{j: z_{ij}=1}\frac{\exp\{-(\theta_i - \beta_j)\}}{[1+\exp\{-(\theta_i - \beta_j)\}]^2}, \quad \text{ for } i = 1, \dots, N;\\
&-\frac{\partial^2 l(\theta; \beta)}{\partial \theta_i \beta_j}= -\frac{\exp\{-(\theta_i - \beta_j)\}}{[1+\exp\{-(\theta_i - \beta_j)\}]^2}, \quad \text{ for } i = 1, \dots, N;\ j \in \{l:z_{il}=1 \};\\
&-\frac{\partial^2 l(\theta; \beta)}{\partial \beta_j^2}= \sum_{i: z_{ij}=1}\frac{\exp\{-(\theta_i - \beta_j)\}}{[1+\exp\{-(\theta_i - \beta_j)\}]^2}, \quad \text{ for } j = 1, \dots, J;\\
&-\frac{\partial^2 l(\theta; \beta)}{\partial \beta_j\theta_i}= -\frac{\exp\{-(\theta_i - \beta_j)\}}{[1+\exp\{-(\theta_i - \beta_j)\}]^2}, \quad \text{ for } j = 1, \dots, J; \ i \in \{k:z_{kj}=1 \}.
\end{align*}
} With the convergence theory for the projected gradient descent algorithm established in \cite{beck2009gradient},   $(\theta^{(I)}, \beta^{(I)})$ from Algorithm~\ref{alg} is guaranteed to converge to $(\hat \theta, \hat \beta)$,  supposing that $(\hat \theta, \hat \beta)$ is the unique solution to optimization \eqref{eq:mle}. The convergence speed of this projected gradient descent algorithm is $O(1/I)$.} 

\section{Statistical Inference}\label{sec-asym-properties}

In this section, we consider the statistical inference of any linear form of $M$. Specifically, we use $g: \mathbb R^{N\times J} \mapsto \mathbb R$ to denote a linear function of $M$ that takes the form
\begin{equation}\label{eq:g}
g(M)=\sum_{i=1}^{N}\sum_{j = 1}^J w_{ij}m_{ij}, 
\end{equation}
where the weights $w_{ij}$ are pre-specified. It is straightforward that a point estimate of $g(M)$ is given by  $g(\hat M)= \sum_{i=1}^{N}\sum_{j = 1}^J w_{ij}\hat m_{ij}$. Our goal is to establish the asymptotic normality for $g(\hat M)$, based on which we can test hypotheses about $g(M)$ or construct confidence intervals. We provide two examples of $g(M)$ that may be of interest in practice.
\begin{example}
Consider $g(M)=m_{ij}$ for entry $(i,j)$ that is not observed, i.e.,  $z_{ij} = 0$. The asymptotic normality of $\hat m_{ij}$ allows us to quantify the uncertainty in our prediction $\exp(\hat m_{ij})/\{1+\exp(\hat m_{ij})\}$ of the unobserved entry, \yc{which can be done using the delta method.}   
\end{example}

\begin{example}
Consider $g(M)= \sum_{j=1}^J (m_{ij} - m_{i'j})/J = \theta_i - \theta_{i'}$, that is of interest in both educational testing and ranking. If we interpret the model as the Rasch model in educational testing, then $\theta_i$ can be regarded as examinee $i$'s ability level. Examinee $i$ is more likely to answer any question correctly than examinee $i'$ if $\theta_i > \theta_{i'}$, and vise versa.
Therefore, even when two examinees do not answer the same test form, the statistical inference of this quantity will allow us to compare their performance and further quantify the uncertainty in this comparison. 
On the other hand, if we draw connections to the Bradley-Terry model in ranking, then $\theta_i$ can be interpreted as subject $i$'s ranking criteria. The statistical inference on $(\theta_i - \theta_{i'})$ for any combination of $i, i'$
would allow us to quantify the uncertainty in the rankings of all $N$ subjects.
\end{example}

\yc{In what follows, we establish some asymptotic results under a double asymptotic regime where both $N$ and $J$ grow to infinity. Such an asymptotic regime is commonly adopted for matrix completion. As discussed in Remark~\ref{rmk:neyman} below, the estimation is inconsistent if $J$ is kept fixed and $N$ goes to infinity, which is typically known as the  Neyman-Scott phenomenon \citep{neyman1948consistent}. Remark~\ref{rmk:neyman} also discusses alternative estimators for the Rasch model.} 

\begin{remark}\label{rmk:neyman}
\yc{The Rasch model is closely related to the Neyman-Scott phenomenon discovered in \cite{neyman1948consistent}. More specifically, \cite{neyman1948consistent} give a setting under which the number of model parameters grows with the number of observations. Under this setting, they showed that the maximum likelihood estimator is statistically inconsistent when the number of observations grows to infinity. Although \cite{neyman1948consistent} considered a normal model, the same phenomenon also exists under the Rasch model. That is, as shown by \cite{andersen1973conditional}, \cite{haberman1977maximum} and \cite{ghosh1995inconsistent},  $(\hat \theta, \hat \beta)$ defined in \eqref{eq:mle} is statistically inconsistent when $J$ is fixed and there is no missing data (i.e., $z_{ij} = 1$ for all $i$ and $j$). This phenomenon naturally carries over to the matrix completion setting. 

With a fixed $J$, it is still possible to consistently estimate the column parameters $\beta_j$ in the Rasch model using a  conditional likelihood estimator \citep{andersen1970asymptotic,andersen1972numerical} or a marginal likelihood estimator \citep{lindsay1991semiparametric}. These methods treat $\theta_i$s as nuisance parameters and profile them out in the likelihood function. We believe that they can also be extended to the matrix completion setting. However, it is not straightforward to extend these estimation methods to a more general low-dimensional model for matrix completion, and their statistical efficiency and computational cost under a matrix completion setting  need further investigation.} 

 
\end{remark}

We first establish the existence and consistency for $M$, $\theta$, and $\beta$. We denote
$$J_{*}=\min\Big\{\sum_{j=1}^{J}z_{ij} : i = 1, ..., N \Big\} \mbox{ and } J^{*}=\max\Big\{\sum_{j=1}^{J}z_{ij}: i = 1, ..., N \Big\}$$ 
as the minimum and maximum numbers of observed entries per row, respectively. 
Similarly, we denote 
$$N_{*}=\min\Big\{\sum_{i=1}^{N}z_{ij}: j = 1, ..., J \Big\} \mbox{ and } N^{*}=\max\Big\{\sum_{i=1}^{N}z_{ij}: j = 1, ..., J \Big\}$$ as the minimum and maximum numbers of observed entries per column, respectively. Let $\Vert x \Vert_{\infty} = \max\{|x_i|: i = 1, ..., n\}$ be the infinity norm of a vector $x = (x_1, ..., x_n)^T$. 
Let $\theta^*$, $\beta^*$ and $M^*$ be the true values of $\theta$, $\beta$ and $M$, respectively. 
Without loss of generality, we assume $N\geq J.$  For simplicity, we also assume {$J_{*} \lesssim N_{*}$ and $J^* \lesssim N^*$.} 
We make the following assumptions.

\begin{condition}\label{cond:bound}
There exists a constant $c<\infty$ such that 
$\Vert \theta^*\Vert_{\infty} < c$ and $\Vert \beta^*\Vert_{\infty} < c$.
\end{condition}

\begin{condition} \label{cond:connect}
For any $(i, j)$, there exist $k\geq 1$ and $1\leq i_1, i_2,...,i_{k} \leq N$ and  $1\leq j_1,j_2, ...,j_k \leq J$ such that 
 $z_{ij_1}=z_{i_1j_1}=z_{i_1 j_2}=z_{i_2 j_2}=...=z_{i_kj_k}=z_{i_kj}=1.$
\end{condition} 

\yc{Condition~\ref{cond:bound} assumes that all the row and column parameters are bounded. This condition further guarantees that $|m_{ij}|\leq 2c$ for all $i$ and $j$. A similar requirement on $m_{ij}$ is needed for 1-bit matrix completion; see e.g., \cite{davenport20141}. Condition~\ref{cond:connect}   is necessary and sufficient for the identifiability of $\theta$, $\beta$ and $M$. We can view $Z$ as the adjacency matrix of a bipartite graph with $N+J$ nodes, where  there exists an edge between a row node $i$ and column node $j$ if and only if $z_{ij} = 1$.  Condition~\ref{cond:connect} is saying that this bipartite graph is a connected graph. If Condition~\ref{cond:connect} is not satisfied, then there exist $i$ and $j$ such that $m_{ij}$ is not identifiable and thus cannot be consistently estimated. We summarize this result in Proposition~\ref{prop:connect}. }

\begin{proposition}\label{prop:connect} 
If Condition~\ref{cond:connect} holds and given $m_{ij}$ for all $i$ and $j$ such that  $z_{ij} = 1$, 
then   $\theta$ and $\beta$ are uniquely determined by equations $\sum_{i=1}^N\theta_i=0$ and $\theta_i - \beta_j = m_{ij}$, $i = 1, ..., N, j = 1, ..., J$, for which $z_{ij} = 1$. That is, $\theta$ and $\beta$ can be uniquely determined by $m_{ij}$ values of the observed entries.  

On the other hand, if Condition~\ref{cond:connect} does not hold and given $m_{ij}$ for all $i$ and $j$ such that  $z_{ij} = 1$, then there exists $(\tilde \theta, \tilde \beta) \neq (\theta, \beta)$, such that  $\sum_{i=1}^N \tilde\theta_i=0$, $\sum_{i=1}^N \theta_i=0$, and $\theta_i - \beta_j = \tilde \theta_i - \tilde \beta_j = m_{ij}$, $i = 1, ..., N, j = 1, ..., J, z_{ij} = 1$. \yc{In that case, there exist $i$ and $j$ such that $z_{ij}=0$ and 
{$$ \theta_i - \beta_j \neq \tilde \theta_i - \tilde \beta_j,$$}
so that the corresponding $m_{ij}$ is not identifiable. 
}

\end{proposition}

\yc{We give an example where Condition~\ref{cond:connect} is not satisfied.} 

\begin{example}
\yc{Suppose that both $N$ and $J$ are even numbers. We let $z_{ij} = 0$ if $i \in \{N/2 + 1, ..., N\}$ or $j\in \{J/2 + 1, ..., J\}$, and 
$z_{ij} = 1$ otherwise. This indicator matrix is shown in Figure~\ref{fig:heatmap}. For any $(i,j)$ satisfying $z_{ij} = 0$, 
there is no $k\geq 1$ and $1\leq i_1, i_2,...,i_k \leq N$ and  $1\leq j_1,j_2, ...,j_k \leq J$ such that $z_{ij_1}=z_{i_1j_1}=z_{i_1 j_2}=z_{i_2 j_2}=...=z_{i_kj_k}=z_{i_kj}=1$.}
\end{example}

\begin{figure}
    \centering
    \includegraphics[scale=0.5]{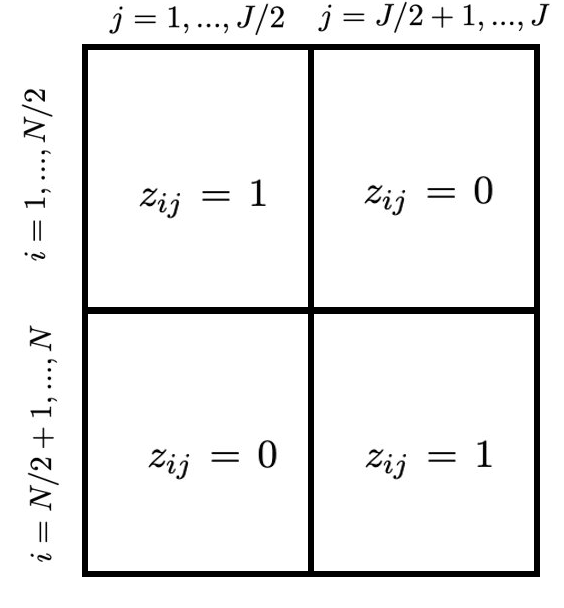}
    \caption{An indicator matrix for which Condition 2 is not satisfied.}
    \label{fig:heatmap}
\end{figure}

\yc{We remark that when Condition~\ref{cond:connect} is not satisfied, it is still possible to draw inference on $\theta_i$, $\beta_{j}$, and $m_{ij}$, for $i \in \mathcal R \subset \{1, ..., N\}$ and $j \in \mathcal C \subset \{1, ..., J\}$, when the bipartite graph corresponding to the submatrix $(z_{ij})_{i \in \mathcal R, j\in \mathcal C}$ is connected. In that case, we can apply Theorems~\ref{thm-existence} through \ref{thm-3-sufficient-condition} below to a subset of data with $i \in \mathcal R$ and  $j\in \mathcal C$. We further remark that Condition~\ref{cond:connect} is likely satisfied under mild conditions when the missing indicator matrix $Z$ is generated by a uniform random sampling scheme. Theorem~\ref{thm:connect} below provides a sufficient condition under which Condition 2 holds.}  

\begin{theorem} \label{thm:connect}

\yc{Suppose that $z_{ij}$ are i.i.d. Bernoulli random variables, satisfying $P(z_{ij} = 1) = p$. Let both $J$ and $p$ be functions of  $N$ satisfying $$N p \geq Jp \geq (\log(N))^4.$$
Then with probability tending to 1, Condition~\ref{cond:connect} holds if there exists an integer $n \geq 1$ such that  
$$p^n J^{(n-1)/2} N^{(n-1)/2} - \log(NJ) \rightarrow \infty$$
if $n$ is odd, and 
$$p^n J^{n/2} N^{(n/2)-1} - 2\log(N) \rightarrow \infty$$
if $n$ is even. 
} 
\end{theorem}
\yc{ Theorem~\ref{thm:connect} is implied by Theorem B  \cite{bollobas1984diameters} of which concerns the diameter of a random bipartite graph and the fact that a graph is connected if and only if its diameter is finite. For example, consider the setting $N=J$ and let $n=2$. Then Theorem~\ref{thm:connect} suggests that Condition~\ref{cond:connect} holds with high probability, if 
$p^2 N -2\log(N) \rightarrow \infty.$

  We next establish the estimation consistency. The following condition is needed. }

\begin{condition}\label{cond:speed}
As $N$ and $J$ grow to infinity, the following are satisfied:
\begin{itemize}
    \item[$(a)$] $J_*^{-1} \log N \to 0.$
    \item[$(b)$] {
    $ N_* J_*N^{-1}\to\infty$ and $ J_*^2J^{-1}\to\infty$.
    }
    \item[$(c)$] {$N_* \asymp N^*$.
    }
\end{itemize}

\end{condition}

{Condition~\ref{cond:speed}(a) is a mild technical condition requires that $J_*$ grows faster than $\log N$. Condition~\ref{cond:speed}(b) imposes constraints on the number of observations for parameters to grow at suitable rates. In particular, note that in the case of $N_*\asymp N^*$ and $J_*\asymp J^*$, the observed entries of the matrix can be of the order $O(N_*J_*)$=$O(N^*J^*)$; then the condition of $ N_* J_*N^{-1}\to\infty$  gives a natural requirement for the consistency theory that the number of observed entries needs to have a higher order than the number of unknown parameters, which is of the order $O(N)$.
Condition~\ref{cond:speed}(c) requires that $N_*$ and $N^*$ are of the same order for convenience of the proof. This assumption essentially requires a balanced missing data pattern that has a similar spirit as the random sampling regimes for missingness adopted in \citet{cai2013max} and \citet{davenport20141}.}

{Similar to Condition 2, the rate requirement of Condition 3 can also be shown to be held with high probability for random design under related requirements, when the missing indicator matrix
$Z$ is generated by a uniform random sampling scheme. To illustrate this,  let $z_{ij}$ be i.i.d. Bernoulli random variables with $P(z_{ij} = 1) = p$. Then   for any $j$, by Hoeffding's inequality, we have
 $P(|\sum_{i=1}^{N}z_{ij} -Np|> x_{N,J}) \leq 2J^{-(1+\epsilon)}$ where $x_{N,J}=  [N(1+\epsilon)\log(J)/2]^{1/2}$ and $\epsilon>0$ is a small constant.
 By union bound, we then have  $N_{*}\asymp N^{*}\asymp Np$ with high probability, if $N^{-1/2} (\log(J))^{1/2}  \lesssim p$. Similarly we have  $J_*\asymp J^*\asymp Jp$ with high probability if $J^{-1/2}(\log(N))^{1/2}  \lesssim p$. When $N\geq J$, it is easy to check that Condition 3 is satisfied with high probability if  $Jp \gg \log N$ and $J^{-1/2}(\log(N))^{1/2}  \lesssim p$ under this random design setting. 
 }

\yc{\begin{theorem}\label{thm-existence}
Assume that Conditions  \ref{cond:bound}, \ref{cond:connect} and \ref{cond:speed} hold. Then, as $N, J$ grow to infinity, maximum likelihood estimator $(\hat{\theta},\hat{\beta})$ exists  in $\mathbb R^{N+J}$ and is unique, with probability tending to 1. 
Furthermore,  
we have 
$$\Vert \hat \theta - \theta^*\Vert_{\infty} =  o_p(1),  \quad  \Vert \hat \beta - \beta^*\Vert_{\infty} =  o_p(1),$$ 
and 
$$\max_{i,j} \vert \hat m_{ij} - m_{ij}^*\vert =  o_p(1).$$ 
\end{theorem}
}

\yc{We note that the maximum likelihood estimator does not exist if there exists a row $i$ such that $Y_{ij'}$s take the same value for all $j'$ such that $z_{ij'} = 1$, or if there exists a column $j$ such that $Y_{i'j}$s take the same value for all $i'$ such that $z_{i'j} = 1$. In these cases, the corresponding $\theta_i$ and $\beta_j$ will converge to $\infty$ or $-\infty$. Theorem~\ref{thm-existence} suggests that these cases are unlikely to occur when both $N$ and $J$ are large. In practice, to avoid non-convergence, we can add the constraints that $\vert \theta_i\vert \leq C$ and $\vert \beta_j\vert \leq C$ for all $i$ and $j$ and a sufficiently large constant $C$.}

\yc{Note that Theorem~\ref{thm-existence} does not give the convergence rate.  We now give the optimal convergence rate under stronger conditions in addition to Condition~\ref{cond:speed}. }

\begin{condition}\label{cond:speed2}
As $N$ and $J$ grow to infinity, the following are satisfied:
\begin{itemize}
    \item[$(a)$] {$J_{*}^{-2}N_{*}(\log N)^2\to 0$.}
\item[$(b)$] {$N_*^{-1/2} \log J \rightarrow 0 $}

\item[$(c)$] {$J_* \asymp J^*$.} 
\end{itemize}

\end{condition}
{Condition \ref{cond:speed2}(a) is a stronger version of Condition \ref{cond:speed}(a) that requires $J_*$ grows faster than $N_*^{1/2}\log N$. Condition \ref{cond:speed2}(b) imposes additional constraints on the grow rate of $N_*$. Condition \ref{cond:speed2}(c) requires that $J_*$ and $J^*$ are of the same order. This set of conditions, together with Condition \ref{cond:speed} will guarantee the optimal convergence rates and asymptotic normality.
Similar to Conditions 2 and 3,  Condition 4 can also be shown to be held with high probability for random design, when the missing indicator matrix
$Z$ is generated by i.i.d.  Bernoulli random variables with the parameter  $p$ satisfies certain requirement. In particular, following the discussion for Condition 3, we can see that Condition 4 is satisfied when 
$J^2 p\gg N(\log N)^2$ and $J^{-1/2}(\log(N))^{1/2}  \lesssim p$.}


\begin{theorem}\label{thm-rate}
Assume that {Conditions \ref{cond:bound}--\ref{cond:speed2}} hold. Then, as $N, J$ grow to infinity, maximum likelihood estimator $(\hat{\theta},\hat{\beta})$ exists, with probability tending to 1. 
Furthermore, as $N$ and $J$ grow to infinity, 
we have 
$$\Vert \hat \theta - \theta^*\Vert_{\infty} =  O_p\big\{(\log N)^{\frac{1}{2}}J_{*}^{-\frac{1}{2}} \big\},  \quad  \Vert \hat \beta - \beta^*\Vert_{\infty} =  O_p\big\{(\log J)^{\frac{1}{2}}N_{*}^{-\frac{1}{2}} \big\},$$ 
and 
$$\max_{i,j} \vert \hat m_{ij} - m_{ij}^*\vert =  O_p\big\{(\log J)^{\frac{1}{2}}N_{*}^{-\frac{1}{2}} + (\log N)^{\frac{1}{2}}J_{*}^{-\frac{1}{2}} \big\}.$$ 
\end{theorem}

\begin{remark}
{Theorem~\ref{thm-rate} above gives the optimal convergence rates for 
$\Vert \hat \theta - \theta^*\Vert_{\infty}$, $\Vert \hat \beta - \beta^*\Vert_{\infty}$, and $\max_{i,j} \vert \hat m_{ij} - m_{ij}^*\vert$.
To illustrate this, consider an oracle setting that $\beta$ take true values; then the convergence rates for maximum likelihood estimators $\hat{\theta}_i$ are $\hat{\theta}_i - \theta_i^*=O_p(J_{*}^{-{1}/{2}})$ and they independently follow   asymptotic normal distributions. From the result that the maximum of $N$ i.i.d. standard normal random variables has the order of $(\log N)^{1/2}$ \citep{van2014probability}, we can see 
the optimal convergence rate of the max-norm of $\hat{\theta}$ is $\|\hat{\theta} - \theta^* \|_{\infty} = O_p\{(\log N)^{{1}/{2}}J_{*}^{-{1}/{2}} \}$.
Similar arguments can be applied to show the optimality of the convergence rate of $\Vert \hat \beta - \beta^*\Vert_{\infty} = O_p\{(\log J)^{{1}/{2}}N_{*}^{-{1}/{2}}\}$. As $\hat{m}_{ij} = \hat{\theta}
_i - \hat{\beta}_j$, the convergence rate of $\vert \hat m_{ij} - m_{ij}^*\vert = O_p \{(\log J)^{{1}/{2}}N_{*}^{-{1}/{2}} + (\log N)^{{1}/{2}}J_{*}^{-{1}/{2}} \}$ is optimal.   }
\end{remark}

To state the asymptotic normality result for $g(\hat M)$, we reexpress $$g(M) = w_g^T \theta +\tilde w_g^T \beta,$$ 
where $w_g=(w_{g1}, \cdots, w_{gN})^T$ and $\tilde w_g=(\Tilde w_{g1}, \cdots, \Tilde w_{gJ})^T$.  Note that this expression always exists 
by letting 
$w_{gi}=\sum_{j=1}^Jw_{ij}$ and $\Tilde w_{gj}=-\sum_{i=1}^Nw_{ij}$. \yc{Recall that $w_{ij}$s are weights defined in \eqref{eq:g}.} 
We introduce some notation. Let $\sigma_{ij}^2 =$ var$(Y_{ij}) = \exp(\theta_i^* -\beta_j^*)/\{1+\exp(\theta_i^* -\beta_j^*)\}^2$, 
 $\sigma_{i+}^2= \sum_{j = 1}^J z_{ij}\sigma_{ij}^2$, 
and $\sigma_{+j}^2= \sum_{i = 1}^N z_{ij}\sigma_{ij}^2.$ 
 Further denote $\hat \sigma_{ij}^2 = \exp(\hat \theta_i -\hat \beta_j)/\{1+\exp(\hat \theta_i -\hat \beta_j)\}^2$, 
 $\hat \sigma_{i+}^2= \sum_{j = 1}^J z_{ij}\hat \sigma_{ij}^2$, 
and $\hat \sigma_{+j}^2= \sum_{i = 1}^N z_{ij}\hat \sigma_{ij}^2$  to
 be the corresponding plug-in estimates.  We use $\Vert \cdot\Vert_1$ to denote the $L_1$ norm of a vector.
 The result is summarized in Theorem~\ref{thm-3-sufficient-condition} below.

\begin{theorem}\label{thm-3-sufficient-condition}
Assume {Conditions \ref{cond:bound}--\ref{cond:speed2}} hold. 
Consider a linear function $g(M) = w_g^T \theta +\tilde w_g^T \beta$ with $g(M) \neq 0$. 
 Further suppose that there exists a constant $C > 0$ such that $\| w_g\|_1 < C$ and $\| \tilde w_g\|_1 < C$. Then 
$$\tilde \sigma(g)^{-1}\big\{g(\hat{M})-g({M^*})\big\} \to N(0,1)~ \text{in distribution},$$
where 
$\tilde{\sigma}^2(g)=\sum_{i=1}^{N}w_{gi}^2(\sigma_{i+}^2)^{-1}+\sum_{j=1}^{J}\Tilde w_{gj}^2(\sigma_{+j}^2)^{-1}.$

Moreover, $\tilde{\sigma}(g)$ can be replaced by its plug-in estimator, i.e., 
\begin{equation}\label{eq:normality}
\hat \sigma(g)^{-1}\big\{g(\hat{M})-g({M^*})\big\} \to N(0,1)  \text{ in distribution},
\end{equation}
where 
$\hat{\sigma}^2(g)=\sum_{i=1}^{N}w_{gi}^2(\hat\sigma_{i+}^2)^{-1}+\sum_{j=1}^{J}\Tilde w_{gj}^2(\hat\sigma_{+j}^2)^{-1}.$

\end{theorem}

We now discuss the implications of Theorem~\ref{thm-3-sufficient-condition}. For each $\theta_i$, 
var$(\hat{\theta}_i)=(\sigma_{i+}^2)^{-1}\big\{1+ o(1)\big\}$. 
It is worth noting that by the classical theory of maximum likelihood estimation, 
$(\sigma_{i+}^2)^{-1}$ is the Cramer-Rao lower bound for the estimation of $\theta_i$ when the column parameters $\beta$ are known. Thus, the result of Theorem~\ref{thm-3-sufficient-condition} implies that $\hat \theta_i$ is an asymptotically optimal estimator for $\theta_i$. Similarly, for each $\beta_j$, 
var$(\hat{\beta}_j)=(\sigma_{+j}^2)^{-1}\big\{1+ o(1)\big\}$, which also achieves
the Cramer-Rao lower bound asymptotically, when the row parameters $\theta$ are known. Moreover, 
var$(\hat{m}_{ij}) =$ var$(\hat \theta_i - \hat \beta_j) =\big\{(\sigma_{i+}^2)^{-1}+(\sigma_{+j}^2)^{-1}\big\}\big\{1+ o(1)\big\}$.  
We end this section with a remark.  

\begin{remark}
The derived asymptotic theory is  different from that for non-linear regression models of increasing dimensions that has been studied in \citet{portnoy1988asymptotic}, \citet{he2000parameters} and \citet{wang2011gee}. 
To achieve asymptotic normality under the setting of these works, one requires the number of observations to grow faster than the square of the number of parameters. 
\yc{Under the setting of the current work, the model has $N+J-1$ free parameters, while the number of observed entries is allowed to grow much slower than $NJ \leq  (N+J-1)^2$.}
\end{remark}

\section{Simulation Study}\label{sec-simulation}


We study the finite-sample performance of the likelihood-based estimator. We consider two settings: (1)
$N=5000$ and $J=200$, and (2) $N=10000$ and $J=400$. Missing data are generated under a block-wise design. That is, we  split the rows into five equal-sized clusters and the columns into four  equal-sized clusters. 
We let each row cluster correspond to the columns from a distinct combination of two  column clusters. Rows from the same cluster have the same missing pattern. Specifically, their entries are observable and only observable on the columns that this row cluster corresponds to. 
This  missing data pattern can be illustrated by a five-by-four block-wise matrix $\{(1,0,0,1,0)^T, (1,1,0,0,1)^T, (0,1,1,1,0)^T, (0,0,1,0,1)^T\}$, where 1 and 0 represent a submatrix with $z_{ij} = 1$ and 0, respectively. An illustration of the missing pattern $Z$ is illustrated in Figure \ref{fig:simulation_data1}.
Under the first setting,  $N_{*}=2000, N^{*}=3000$, and $J_{*}=J^{*}=100$. Under the second setting, $N_{*}=4000, N^{*}=6000$, and $J_{*}=J^{*}=200$. For each setting, $\theta$ is simulated from a uniform distribution over the space
$\{x = (x_1, ..., x_N)^T: \sum_{i=1}^N x_i = 0,  -2 \leq x_i \leq 2 \}$, and $\beta$ is obtained by simulating $\beta_j$ independently from the uniform distribution over the interval $[-2,2]$. For each setting, 2000 independent datasets are generated from the considered model. 
\begin{figure}[ht]
  \centering
  \includegraphics[scale=1.6]{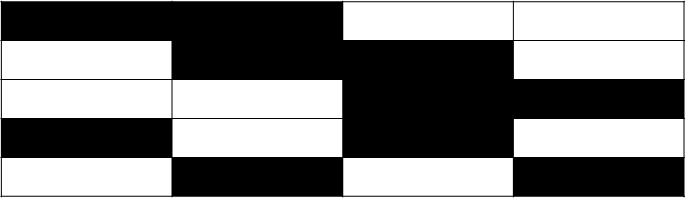}
  \caption{A heat map of $Z$. The black and white regions correspond to $z_{ij}=1$ and 0, respectively.}
  \label{fig:simulation_data1}
\end{figure} 

Under setting (1), the mean squared estimation errors for $M$, $\theta$, and $\beta$ are 0.067, 0.064, and 0.0028, respectively, across all relevant entries and all 2000 independent samples. Under setting (2), these values read 0.033,  0.031 and 0.0013, respectively. Unsurprisingly, increasing sample sizes can improve estimation accuracy.

We then examine the variance approximation in Theorem~\ref{thm-3-sufficient-condition}. We compare $\hat{\sigma}^2(g)$, $\tilde{\sigma}^2(g)$ and $s^2(g)$, where $s^2(g)$ denotes the sample variance of $g(\hat M)$ that is calculated based on the 2000 simulations.
As $\hat{\sigma}^2(g)$ varies across the datasets, we calculate 
$\bar{\sigma}^2(g)$ as the average of $\hat{\sigma}^2(g)$ over 2000 simulated datasets.  We consider  functions $g(M) = m_{ij}, \theta_i, \beta_j$, $i=1, ..., N, j = 1, ..., J$. The results are given in Figure~\ref{fig:var-pairs}, where
panels (a)-(c) show the scatter plots of ${s}^2(g)$ against $\bar{\sigma}^2(g)$ and panels (d)-(f) show those of $s^2(g)$ against $\tilde{\sigma}^2(g)$. These plots suggest that $\bar{\sigma}^2(g)$, $\tilde{\sigma}^2(g)$, and $s^2(g)$ are close to each other, for the specific forms of $g$ that are examined.


\begin{figure}[!ht]
  \centering
  \begin{subfigure}{\linewidth}
    \centering
    \includegraphics[scale=0.5]{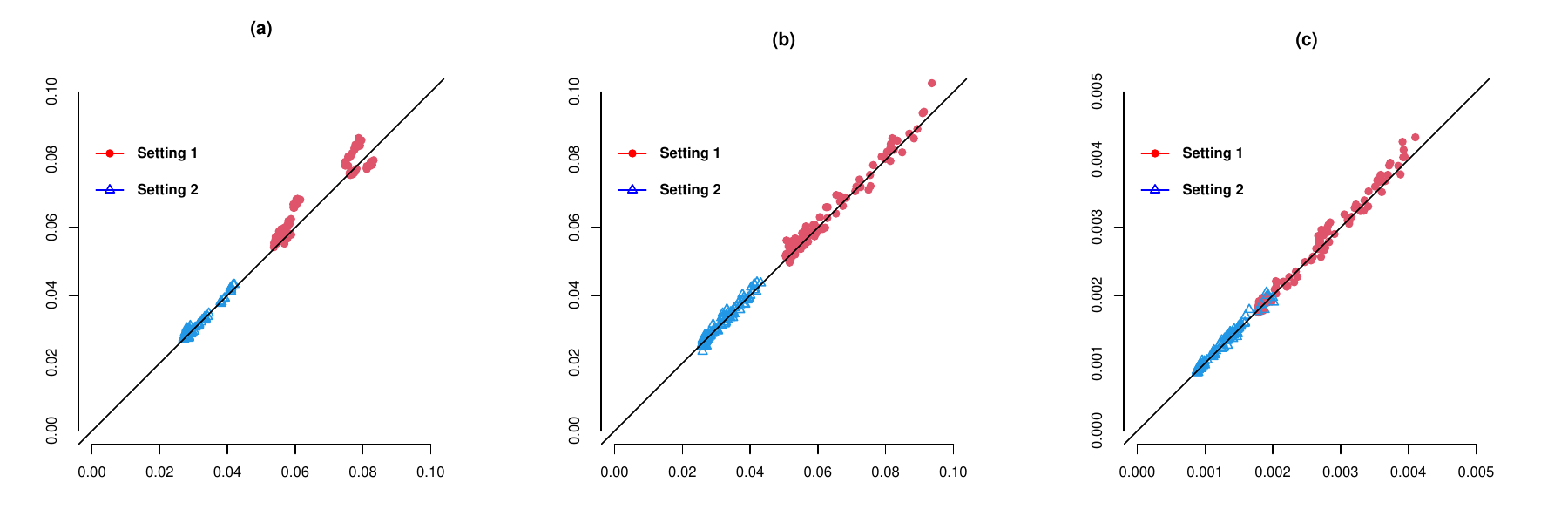}
  \end{subfigure}
  \begin{subfigure}{\linewidth}
    \centering
    \includegraphics[scale=0.5]{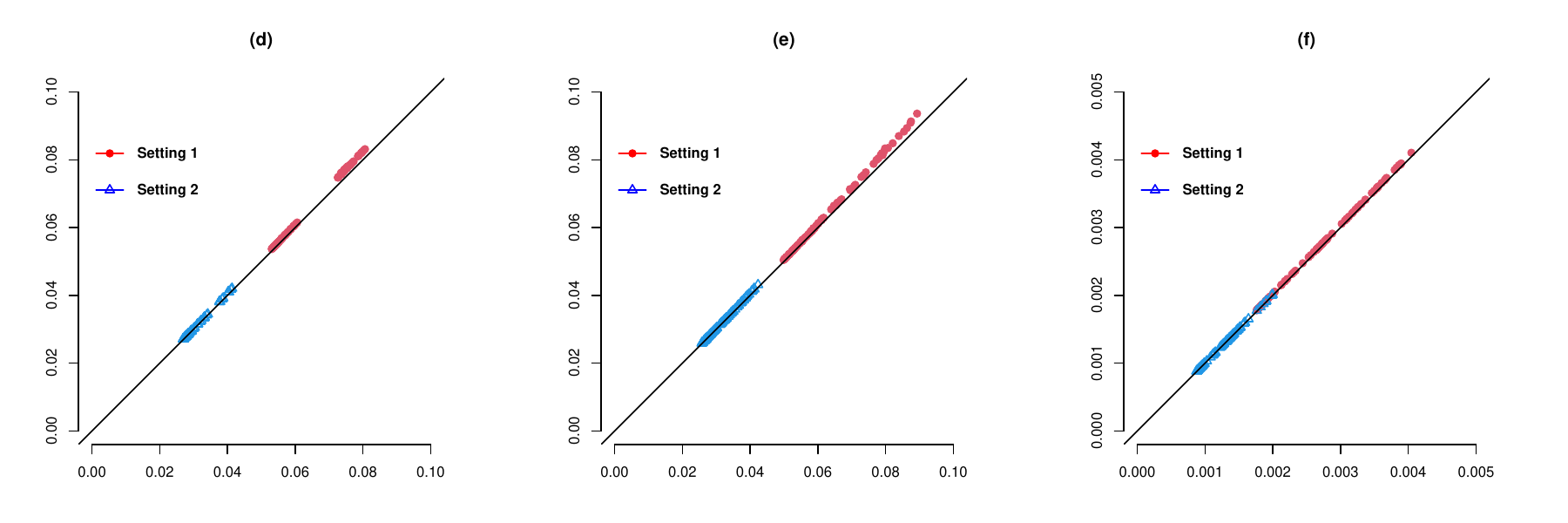}
  \end{subfigure}
  \caption{Panels (a)-(c) plot $s^2(g)$ against $\bar{\sigma}^2(g)$ for $g(M) = m_{ij}$, $\theta_i$, and $\beta_j$, respectively, for {fixed block-wise setting.}
  Panels (d)-(f) plot $s^2(g)$ against $\tilde{\sigma}^2(g)$ for $g(M) = m_{ij}$, $\theta_i$ and $\beta_j$, respectively, for {fixed block-wise setting.}
  Each panel shows 100 randomly sampled 
  $m_{ij}$, $\theta_i$, or $\beta_j$ under each setting.  The line $y = x$ is given as a reference. }  
  \label{fig:var-pairs}
\end{figure}  

To validate asymptotic normality, we compare the empirical densities of the 2000 sample estimates of $m_{11}$, $\theta_1$ and $\beta_1$ against their respective theoretical normal density curves in Figure \ref{fig:density-fixed} for illustration. We can observe from Figure \ref{fig:density-fixed} that the empirical distributions of the estimates agree well with their corresponding theoretical distributions. 
\begin{figure}[ht]
  \centering
  \begin{subfigure}{\linewidth}
    \centering
    \includegraphics[scale=0.5]{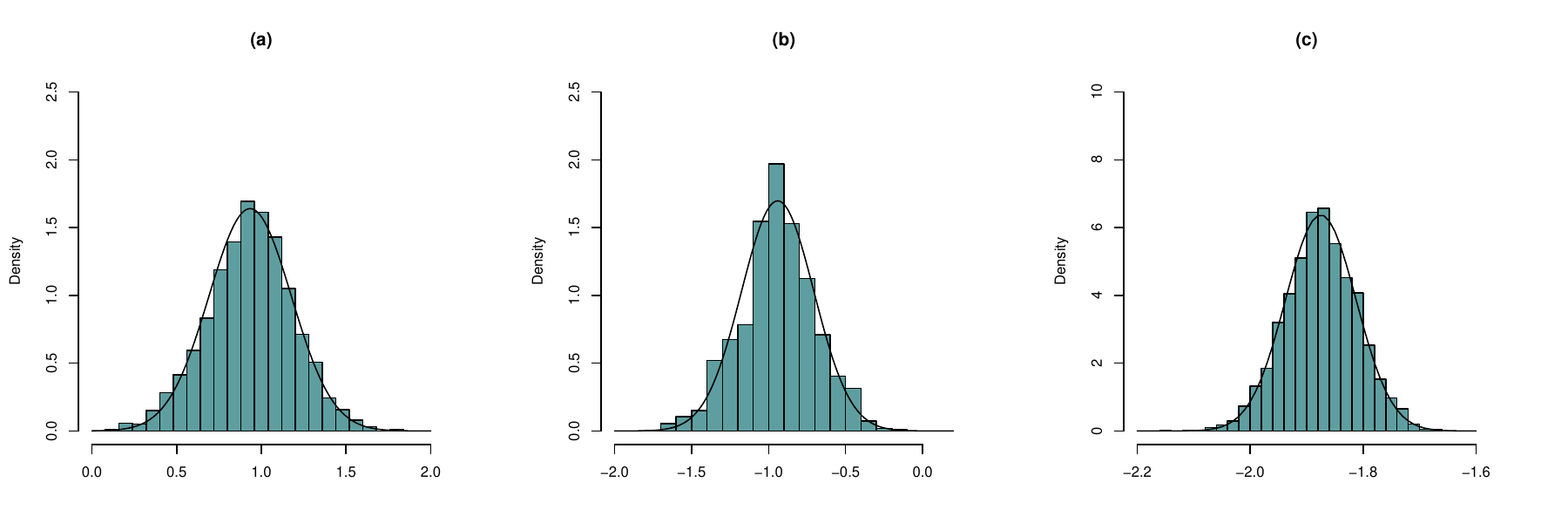}
  \end{subfigure}
  
  \begin{subfigure}{\linewidth}
    \centering
    \includegraphics[scale=0.5]{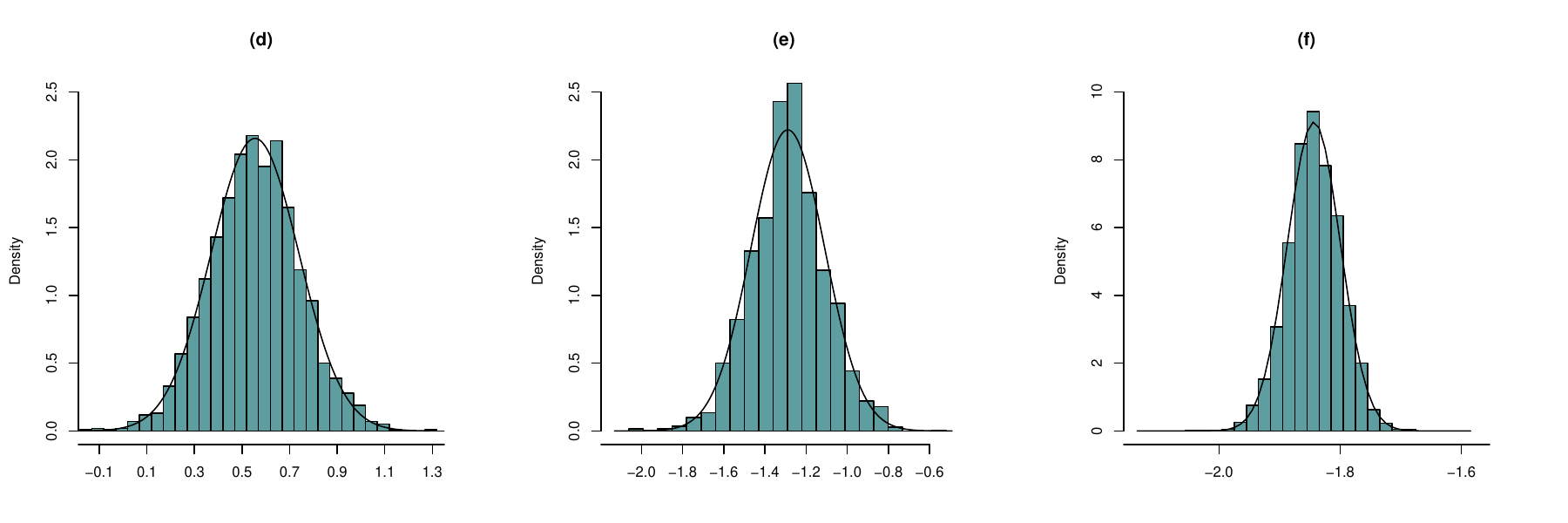}
  \end{subfigure}
\caption{Panels (a)-(c) presents the empirical densities (histograms) of $\hat{m}_{11}$, $\hat{\theta}_{1}$ and $\hat{\beta}_1$ under setting (1), respectively, out of 2000 simulations for {fixed block-wise setting.} 
 Panels (e)-(g) presents the empirical densities of $\hat{m}_{11}$, $\hat{\theta}_{1}$ and $\hat{\beta}_1$ under setting (2), respectively, out of 2000 simulations, for {fixed block-wise setting.} 
 The curves are theoretical density curves of N$(m_{11}, \Tilde \sigma^2(m_{11}))$, N$(\theta_{1}, \Tilde \sigma^2(\theta_{1}))$ and N$(\beta_{1}, \Tilde \sigma^2(\beta_{1})),$ respectively, included as references.} 
 \label{fig:density-fixed}
 \end{figure}

Furthermore, for each $m_{ij}$, $\theta_i$, and $\beta_j$, we construct its 95\% Wald interval based on \eqref{eq:normality}, for which the empirical coverage based on 2000 independent replications is computed.  This result is shown in  Figure~\ref{fig:coverage_prob}, where the two panels correspond to the two simulation settings, respectively. In each panel, the three box plots show the empirical coverage probabilities for entries of $M$, $\theta$, and $\beta$, respectively. As we can see, all these empirical coverage probabilities are close to  the nominal level of 95\%. 

{We also report the average number of iterations for convergence and the average CPU time per iteration as follows. For the above  designs, the  average number of iterations and average CPU time per iteration are (a) 184.70 and 9.24 seconds under setting 1; (b) 176.46 and 47.18 seconds under setting 2.
The convergence criteria is set to be the consecutive change in the joint log-likelihood is smaller than 0.001.} 


\begin{figure}[ht]
  \centering
  \includegraphics[scale=0.5]{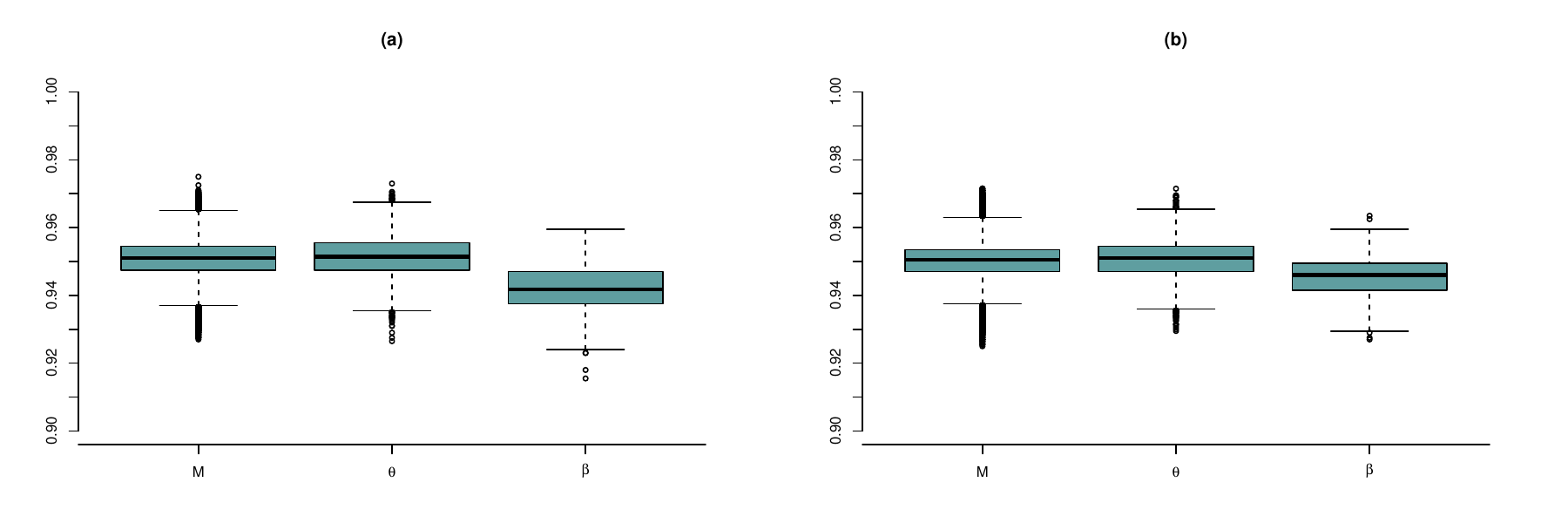}
  \caption{Panels (a) and (b) show the empirical coverage rates for the 95\% Wald intervals under {fixed block-wise} settings (1) and (2), respectively.}
  \label{fig:coverage_prob}
\end{figure}

%


{In addition, to further demonstrate the performance of the likelihood-based estimator, we also conduct a simulation study where $z_{ij}$ are randomly sampled under the setting that $N=5000$ and $J=200$. 
Let $z_{ij}$ be sampled i.i.d. from a Bernoulli distribution with $P(z_{ij} = 1) = 0.5$. The generation of the rest of the parameters and the evaluation techniques for the estimators are the same as in study under fixed block-wise setting. 
Under random sampling setting, the mean squared estimation errors for $M$, $\theta$, and $\beta$ are 0.068, 0.064, and 0.0027, respectively, across all relevant entries and all 2000 independent samples. 
The  average number of iterations and average CPU time per iteration are 182.55 and 13.93 seconds. }

{To examine the variance approximation under random sampling setting, we compare $\hat{\sigma}^2 (g)$, $\tilde{\sigma}^2(g)$ and $s^2(g)$ using the scatter plots of ${s}^2(g)$ against $\bar{\sigma}^2(g)$ in panels (a)-(c) of Figure \ref{fig:var-pairs-random} and the scatter plots of $s^2(g)$ against $\tilde{\sigma}^2(g)$ in panels (d)-(f) of Figure \ref{fig:var-pairs-random}, based on the 2000 simulation replications. From Figure \ref{fig:var-pairs-random}, we see that under random sampling setting, the $\bar{\sigma}^2(g)$, $\tilde{\sigma}^2(g)$, and $s^2(g)$ are close to each other for different $g(M)$.}

{To check the asymptotic normality under the random sampling setting, Figure \ref{fig:density-random} presents the empirical densities of the estimates densities of $m_{11}$, $\theta_1$ and $\beta_1$ over 2000 samples against theoretical curves. The plots show that the empirical distributions agree well with the theoretical normal distributions.}
{Figure \ref{fig:coverage_prob_random} further shows the empirical coverage of 95$\%$ Wald intervals over the 2000  replications for $M, \theta$, and $\beta$. These plots suggest the empirical coverage probabilities are close to the nominal level of 95$\%$.}

\begin{figure}[!ht]
  \centering
  \begin{subfigure}{\linewidth}
    \centering
    \includegraphics[scale=0.5]{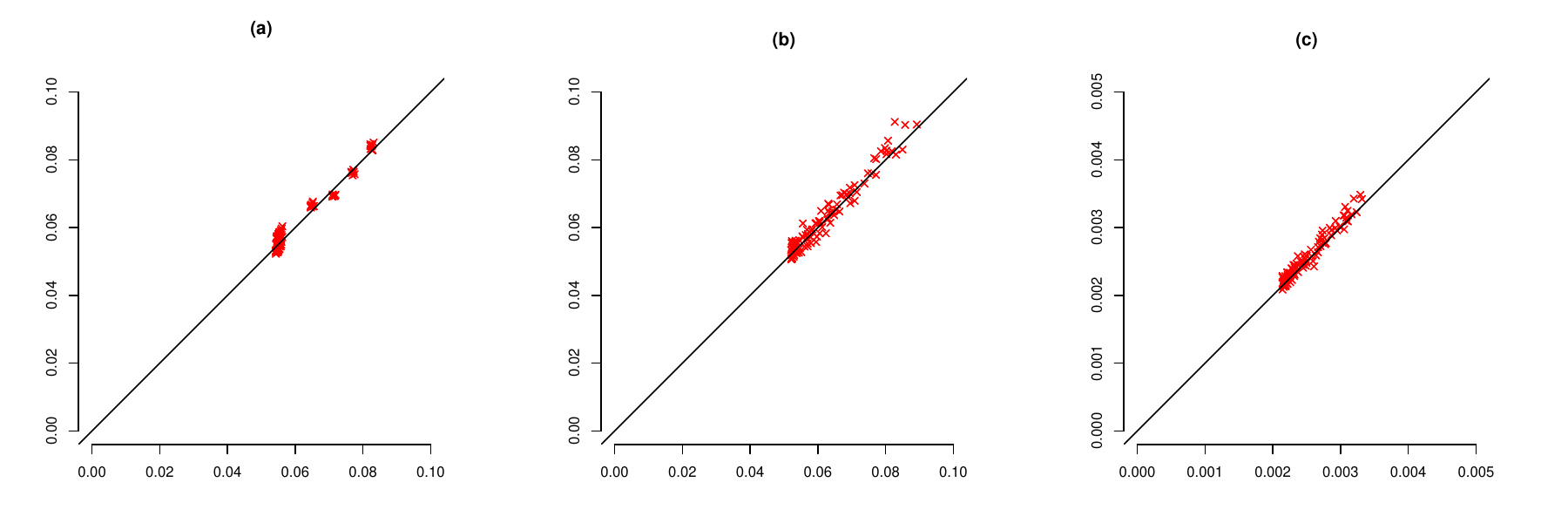}
  \end{subfigure}
  \begin{subfigure}{\linewidth}
    \centering
    \includegraphics[scale=0.5]{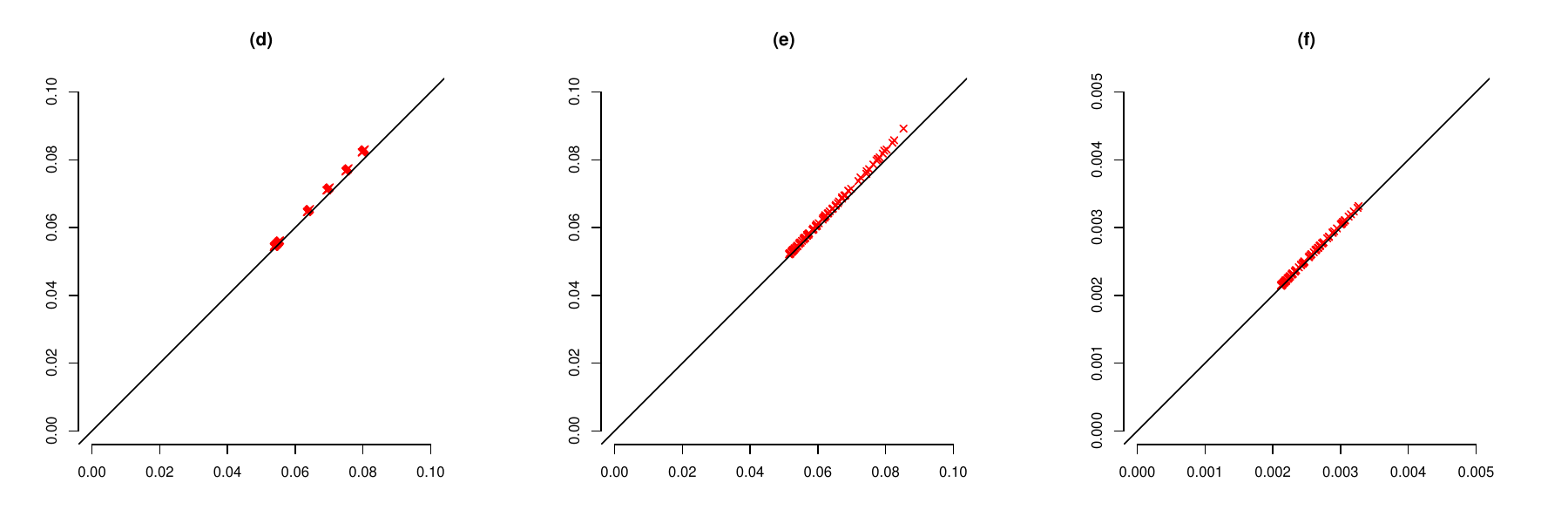}
  \end{subfigure}
  \caption{{Panels (a)-(c) plot $s^2(g)$ against $\bar{\sigma}^2(g)$ for $g(M) = m_{ij}$, $\theta_i$, and $\beta_j$, respectively, and
  Panels (d)-(f) plot $s^2(g)$ against $\tilde{\sigma}^2(g)$ for $g(M) = m_{ij}$, $\theta_i$ and $\beta_j$, respectively,  for the random design setting.
  Each panel shows 100 randomly sampled 
  $m_{ij}$, $\theta_i$, or $\beta_j$ under each setting.  The line $y = x$ is given as a reference.} }  
  \label{fig:var-pairs-random}
\end{figure}

\begin{figure}[ht]
  \centering
  \begin{subfigure}{\linewidth}
    \centering
    \includegraphics[scale=0.5]{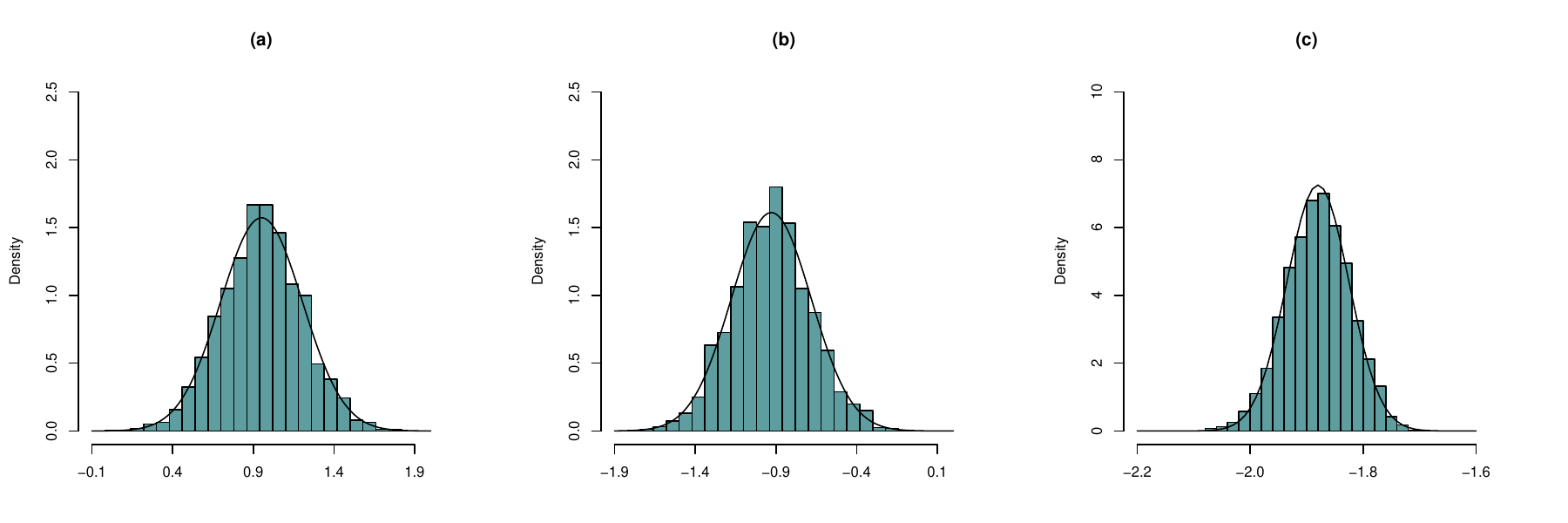}
  \end{subfigure}
  
\caption{{Panels (a)-(c) presents the empirical densities (histograms) of $\hat{m}_{11}$, $\hat{\theta}_{1}$ and $\hat{\beta}_1$ for the random design setting, respectively. 
 The curves are theoretical density curves of N$(m_{11}, \Tilde \sigma^2(m_{11}))$, N$(\theta_{1}, \Tilde \sigma^2(\theta_{1}))$ and N$(\beta_{1}, \Tilde \sigma^2(\beta_{1})),$ respectively, included as references.} }
 \label{fig:density-random}
 \end{figure}

\begin{figure}[ht]
  \centering
  \includegraphics[scale=0.5]{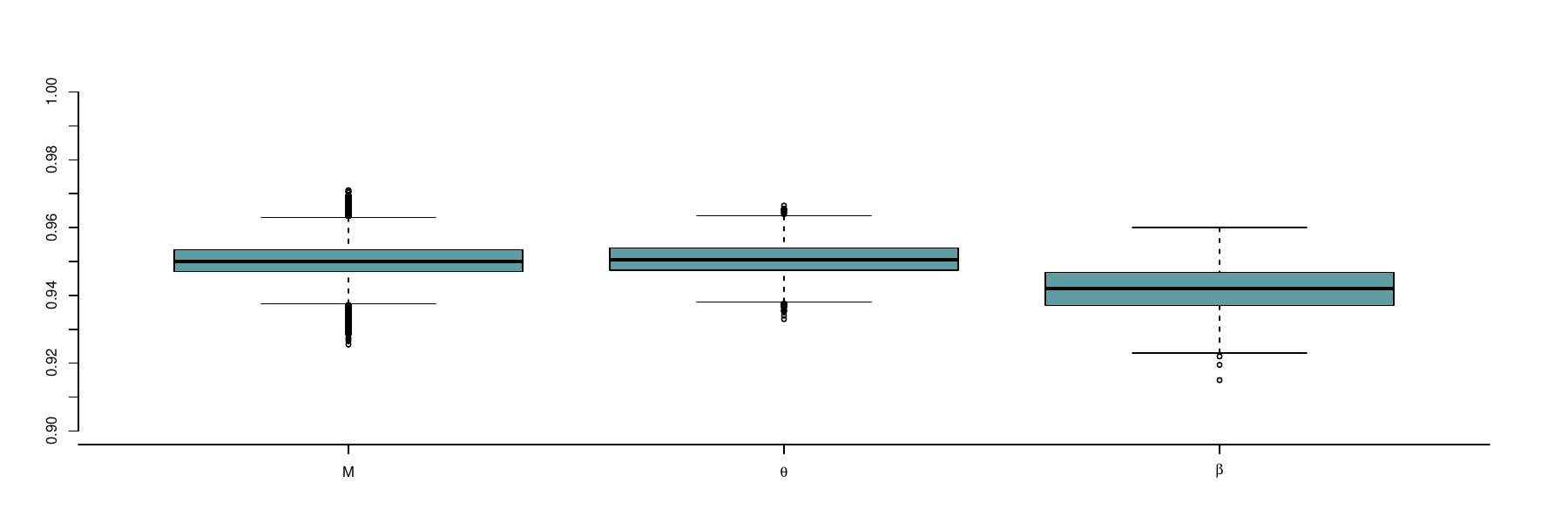}
  \caption{{Boxplots of the empirical coverage rates for the 95\% Wald intervals under the random design setting}.}
  \label{fig:coverage_prob_random}
\end{figure} 


\section{Real-data Applications}\label{sec-real-data}
In what follows, we consider two real-data applications. 


\subsection{Application to Educational Testing}\label{sec-real-data1}

We first apply the proposed method to link
two forms of an educational test that share 
 common items. The dataset is a benchmark dataset for studying linking methods for educational testing  \citep{gonzalez2017applying}. It contains binary responses from
 two forms of a 
college admission test. Each form has 120 items and is answered by 2000 examinees. 
There are 40 common items shared by the two test forms.  
There is no missing data within each test. 
Thus,  $N = 4000$, $J = 200$, and 40\% of the data entries are missing. 
We apply the proposed method to this dataset. Making use of Theorem~\ref{thm-3-sufficient-condition}, 95\% confidence intervals are obtained for both the row (i.e., person) parameters and  the column (i.e., item) parameters. The results allow us to compare students who took different test forms, as well as non-common items from the two forms. For illustration, we randomly choose 100 row parameters and 100 column parameters  and show their  95\% confidence intervals in Figure~\ref{fig: real-data-CI}. 
Such uncertainty quantification can be 
vital for colleges when making admission decisions. 

\begin{figure}[ht]
  \centering{\includegraphics[scale=0.5]{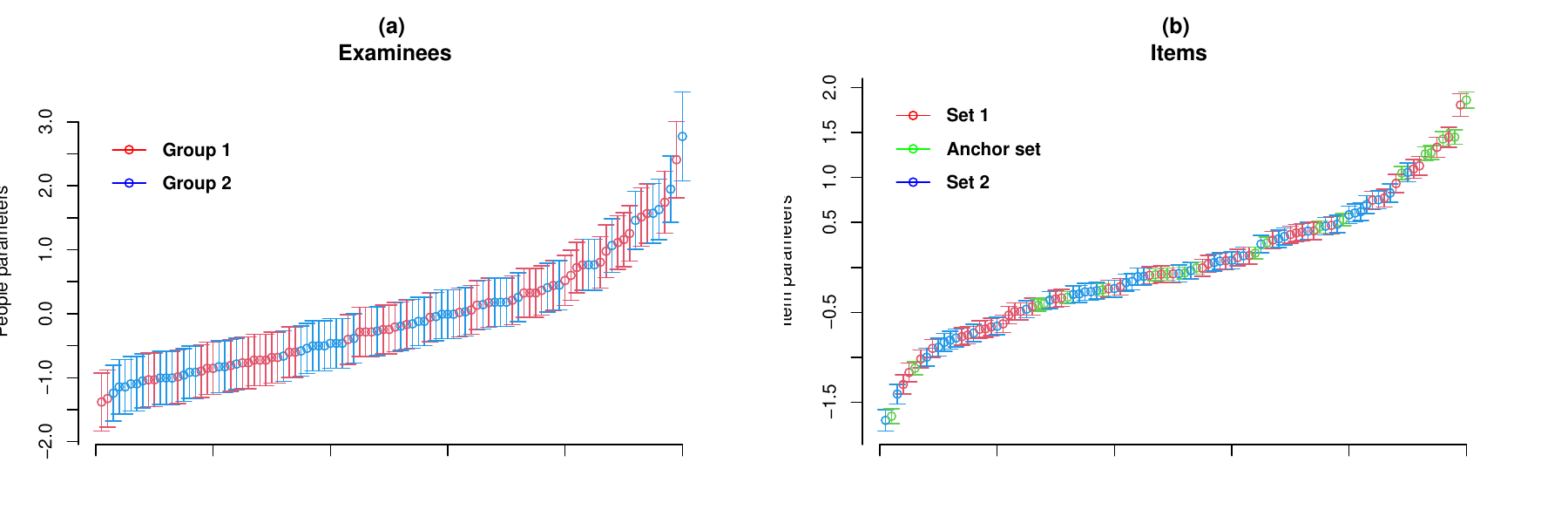}}
  \caption{(a) 95\% confidence intervals of 100 row parameters, with 50 randomly selected from each group. (b) 95\% confidence intervals of the 100 column parameters, with 40 each randomly chosen from group 1 and group 2 and 20 randomly selected from anchor items (i.e., common items).}
  \label{fig: real-data-CI}
\end{figure}

\subsection{Application to Senate Voting}\label{sec-real-data2}

We now apply the proposed method to the United States senate roll call voting data.
Data from the 111th through the 113th congress that include the voting records from 
January 11, 2009, to December 16, 2014. 
Quite a few senators did not serve for the entire period. 

To apply our method to senators’ roll call voting records with $\theta_i$ being interpreted as the conservativeness score of senator $i$, 
we pre-process the data as follows. 
First, five senators who did not serve for more than half a year during the period are removed from the dataset, including Edward M. Kennedy, Joe Biden, Hilary Clinton, Julia Salazar, and Carte Goodwin. Second, 191 bills are removed, as all the observed votes for each of these bills are the same, and consequently, their maximum likelihood estimates do not exist. After these two steps, the resulting dataset contains $N = 139$ senators and $J = 1648$ bills. Finally, for bill $j$ that has higher 
percentage support within the Republican  party than that within the Democratic party, we let 
$Y_{ij} = 1$ if senator $i$ voted for the bill and $Y_{ij} = 0$ if senator $i$ voted against it. For
bill $j$ that has higher 
percentage support within the Democratic  party than that within the Republican party, we let 
$Y_{ij} = 1$ if senator $i$ voted against the bill and  $Y_{ij} = 0$ if he/she voted for it. The value of $Y_{ij}$ is missing if the senator chose not to vote or  he/she was not in the senate when this bill was voted. 
For the final data being analyzed, the proportion of missing entries is 26.1\%, and the connectedness Condition~\ref{cond:connect} is satisfied. The missingness pattern of the dataset is given in 
Figure \ref{fig: heat-map-Z-voting}.  Note that in this example, $N<J$. However, our asymptotic results are still applicable if we simply switch the roles of $N$ and $J$ in the required conditions.

Our asymptotic results allow us to compare senators'  ideological positions, even if they  did not serve in the senate at the same time. For example, 
Judd Gregg served in the senate between January 3, 1993, and January 3, 2011, while 
Marco Rubio started his first term as a senator on January 3, 2011. In our model, Judd Gregg ($\theta_i$) and Marco Rubio ($\theta_k$) have estimated conservativeness scores of 2.59 and 4.25, respectively. Applying our asymptotic results, we have $\hat{\theta}_i-\hat{\theta}_k = -1.66$ and its standard error is 0.169. 
If we test $H_0: \theta_i=\theta_k$ against $H_1: \theta_i\neq \theta_k$, we obtain an extremely small p-value of $9.0\times10^{-23}.$ Therefore, we conclude that senator Marco Rubio is significantly more conservative than senator Judd Gregg.

In addition, we present in 
Tables \ref{table: voting-ranking-conservative} and \ref{table: voting-ranking-liberal} 
the ten senators with the largest row parameter estimates and the ten senators with the smallest row parameter estimates. 
These results align well with the public perceptions of these senators. For example, Jim Demint, who is ranked the most conservative senator in this dataset by our method, 
was also identified by Salon as one of the most conservative members of the Senate \citep{Kornacki2011why}. Our method ranks Mike Lee second, though his conservativeness score is not significantly different from that of Demint. In fact, in 2017, the New York Times used  the NOMINATE system \citep{poole2001d} to arrange Republican senators by ideology and ranked Lee
as the most conservative member of the Senate \citep{Parlapiano2017how}. For another example, 
Brian Schatz, ranked the most liberal senator by our method, is well-known as 
 a liberal Democrat. During his time in the Senate, he voted with the Democratic party on most issues.

 Finally, the  95\% confidence intervals for all the row parameters are shown in Figure \ref{fig:real-data-CI1}, and a full list of rankings for all 139 senators is given in the Appendices, where the corresponding row parameter estimates and their
 standard errors are also presented.

\begin{figure}[ht]
  \centering{\includegraphics[height=4cm, width=12cm]{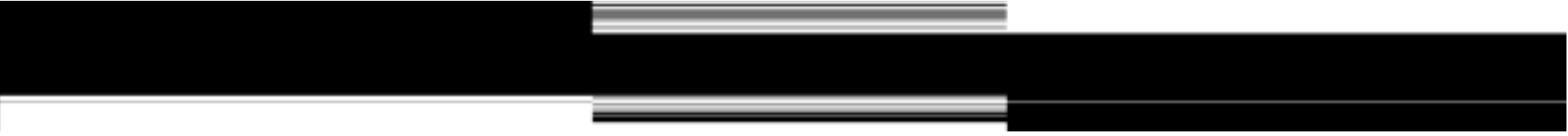}}
  \caption{A heat map of $Z$. The black and white regions correspond to $z_{ij}=1$ and 0, respectively.}
  \label{fig: heat-map-Z-voting}
\end{figure}

\begin{table}[ht]
\centering
\begin{tabular}{lllc} 
 \hline
 Rank &Senator (party) &State & Conservativeness Score (s.e.$(\hat{\theta}$))\\
 \hline
1 & Jim DeMint (Rep) & South Carolina & 5.87 (0.157) \\ 
  2 & Mike Lee   (Rep) & Utah   & 5.73 (0.138) \\ 
  3 & Ted Cruz  (Rep) & Texas   & 5.65 (0.195) \\ 
  4 & Tom Coburn (Rep) & Oklahoma & 5.25 (0.114) \\ 
  5 & Rand Paul (Rep) & Kentucky & 5.24 (0.129) \\ 
  6 & Tim Scott (Rep) & South Carolina & 5.17 (0.176) \\ 
  7 & Jim Bunning (Rep) & Kentucky & 4.92 (0.204) \\ 
  8 & Ron Johnson  (Rep) & Wisconsin & 4.84 (0.119) \\ 
  9 & James Risch (Rep) & Idaho   & 4.81 (0.102) \\ 
  10 & Jim Inhofe  (Rep) & Oklahoma & 4.69 (0.103) \\ 
 \hline
\end{tabular}
\caption{Ranking of the top 10 most conservative senators predicted by the model. Rep and Dem represent the Republican party and the Democratic party, respectively.}
\label{table: voting-ranking-conservative}
\end{table}

\begin{table}[ht]
\centering
\begin{tabular}{lllc} 
 \hline
 Rank & Senator (party) & State & Conservativeness Score (s.e.$(\hat{\theta})$) \\ 
  \hline
  1 & Brian Schatz (Dem) & Hawaii  & -4.74 (0.468) \\ 
  2 & Roland Burris (Dem) & Illinois & -4.43 (0.297) \\ 
  3 & Mazie Hirono (Dem) & Hawaii  & -4.17 (0.383) \\ 
  4 & Cory Booker (Dem) & New Jersey & -4.14 (0.572) \\ 
  5 & Tammy Baldwin (Dem) & Wisconsin & -3.90 (0.352) \\ 
  6 & Sherrod Brown (Dem) & Ohio  & -3.89 (0.168) \\ 
  7 & Tom Udall (Dem) & New Mexico & -3.85 (0.165) \\ 
  8 & Dick Durbin (Dem) & Illinois & -3.83 (0.164) \\ 
  9 & Ben Cardin (Dem) & Maryland & -3.82 (0.163) \\ 
  10& Sheldon Whitehouse (Dem) & Rhode Island & -3.74 (0.163) \\ 
 \hline
\end{tabular}
\caption{Ranking of the top 10 most liberal senators predicted by the model. Rep and Dem represent the Republican party and the Democratic party, respectively.}
\label{table: voting-ranking-liberal}
\end{table}


\begin{figure}[ht]
  \centering{\includegraphics[scale=0.25]{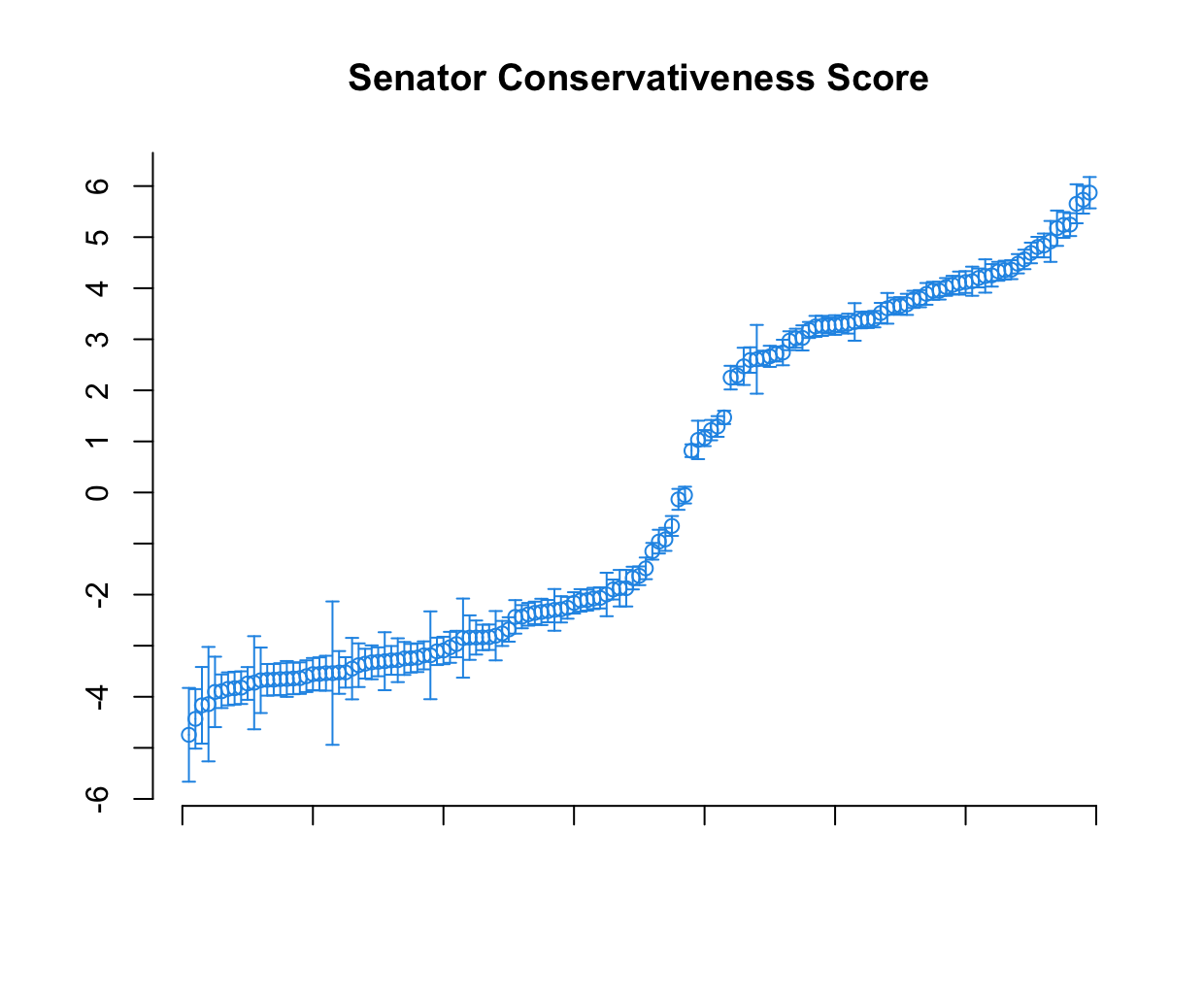}}
  \caption{95\% confidence intervals of 139 row (i.e. senator) parameters. }
  \label{fig:real-data-CI1}
\end{figure}

\section{Discussions}\label{sec-discussion}
This note considers the statistical inference for binary (or 1-bit) matrix completion under 
a unidimensional nonlinear factor model, the Rasch model. 
Asymptotic normality results are established. Our results suggest that the maximum likelihood estimator is statistically efficient, even though the number of parameters diverges. Our simulation study shows that the developed asymptotic result provides a good approximation to finite sample data, and 
two real-data examples demonstrate its usefulness in the areas of educational testing and political science. \yc{One limitation of the current asymptotic normality result is that it requires relatively strong conditions, especially Condition 4(a), which excludes settings where $J_* = O(N_*^{1/2})$. Thus, future research is needed to investigate the extent to which these conditions can be relaxed. } 

The current results can be easily extended to matrix completion problems with a quantized measurement that has a similar natural exponential family form.
Admittedly, the model considered may be oversimple for complex application problems, for example, certain collaborative filtering problems for which the rank of the underlying matrix $M$ may be higher than considered here, and the underlying latent factors may be multi-dimensional. The extension of the current results to more flexible models is left for future investigation. 
As the first inference result for binary matrix completion,  
we believe the current results will shed light on the statistical inference for more general matrix completion problems.


\acks{This research is partially supported by National Science Foundation CAREER SES-1846747 and Institute of Education Sciences R305D200015. }


\vskip 0.2in
\bibliography{bib}

\appendix

\newpage 
\noindent{\bf{\Large Appendix}}\\\\
\noindent The appendix contains the proofs of theorems and proposition in Appendix A, the proofs of the supporting lemmas in Appendix B, and additional real-data application results from Section \ref{sec-real-data2} ``Application to Senate Voting'' in Appendix C.
\section*{Appendix A: Proof of Theorems and Proposition}
Appendix A contains proofs of the theorems and the proposition developed in the main article. 

\begin{proof}[Proof of Theorem \ref{thm:connect}]
This result is directly implied by Theorem B of \cite{bollobas1984diameters}, which shows that under the conditions of Theorem~\ref{thm:connect}, with probability tending to 1, the corresponding bipartite random graph has diameter no larger than $n+1$. This result combined with the fact that a graph is connected if and only if its diameter is finite proves Theorem~\ref{thm:connect}. 
\end{proof}

We now focus on the rest of the theorems and propositions. We start with defining some notation.
Implicitly index $J$  with $N$ such that $J_N \to \infty$ as $N\to \infty$ for notation convenience.  Note that this does not impose any rate requirement for $N$ and $J.$
Let $\Omega_{N}=\big\{x=(x_{ij}: z_{ij}=1, i=1,...,N, j=1,...,J): x_{ij}=\theta_i-\beta_j, \theta_i, \beta_j\in \mathbb{R}, \sum_{i=1}^N \theta_i=0\big\}$ be a vector space.
Define on $\Omega_N$ a variance weighted inner product $[\cdot,\cdot]_{\sigma}$ with $[x, y]_{\sigma}=\sum_{i=1}^N\sum_{j\in S_J(i)}x_{ij}\sigma_{ij}^2y_{ij}$ for any $x, y \in \Omega_{N},$
where $S_J(i)=\{j=1,...,J : z_{ij}=1\}$,   $\sigma_{ij}^2=\exp(m_{ij}^*)/\{1+\exp(m_{ij}^*)\}^2$ and the subscript $\sigma$ means the inner product depends on $\sigma_{ij}^2, i=1,...,N, j=1,...,J, z_{ij}=1$.
Denote the associated norm as $\|\cdot\|_{\sigma}$ with $\|x\|_{\sigma}^2=\sum_{i=1}^N\sum_{j\in S_J(i)}x_{ij}^2\sigma_{ij}^2$ for $x\in \Omega_N.$ 
Let $M_N=\big(m_{ij}: z_{ij}=1, i=1,...,N, j=1,...,J, m_{ij}=\theta_i-\beta_j) \in \Omega_N$, $M_N^*=\big(m_{ij}^*: z_{ij}=1, i=1,...,N, j=1,...,J, m_{ij}^*=\theta_i^*-\beta_j^*) \in \Omega_N$ and $\hat{M}_N=\big(\hat{m}_{ij}: z_{ij}=1, i=1,...,N, j=1,...,J, \hat{m}_{ij}=\hat{\theta}_i-\hat{\beta}_j) \in \Omega_N.$
Note that as a result of Proposition \ref{prop:connect},
for any linear form $g$ of $M$, $g(M)$ can be re-expressed as a linear form of $x\in \Omega_N$, with $g(x)=\sum_{i=1}^N\sum_{j\in S_J(i)}w_{ij}x_{ij}$, where we denote $w_{ij}=w_{ij}(g)$,  which depends on $g$, for notation simplicity. 
Let $\Omega_{N}^*$ consist of all linear forms $g$ on $\Omega_N$ such that $g(x)=0$ if $x=0$ and $x\in \Omega_N.$
Without loss of generality, we will work with $g\in \Omega_{N}^*$ in the proofs.
For any subset $A\subset\Omega_{N}^*$, define  $\|\cdot\|_{\sigma}(A)$ to be the norm on $\Omega_N$ such that for any $x \in \Omega_{N}$, $\|x\|_{\sigma}(A)$ is the smallest non-negative number such that $\vert g(x)\vert \leq \|x\|_{\sigma}(A)\sigma(g)$ for any $g\in A,$ where $\sigma(g)=\sup_{x\in \Omega_N}\{\vert g(x)\vert: \|x\|_{\sigma}\leq 1\}.$
Let $$E_N=\Big(E_{ij}: z_{ij}=1, i=1,...,N, j=1,...,J\Big),$$  with $E_{ij}=\mathbb{E}[Y_{ij}]=e^{m_{ij}^*}/(1+e^{m_{ij}^*}),$ be the vector of expected responses corresponding to the observed entries. 
Further define $R_N \in \Omega_N$ satisfying $$[x, R_N]_{\sigma}= \sum_{i=1}^N\sum_{j\in S_J(i)}x_{ij}(Y_{ij}-E_{ij}), \quad\quad x\in \Omega_N.$$
Define an evaluation measure $U_N(\cdot, \cdot)$ such that for any $y, v \in \Omega_N$, $U_N(y, v) \in \Omega_N$ satisfies
$$[x, U_N(y, v)]_{\sigma}=\sum_{i=1}^N\sum_{j\in S_J(i)}x_{ij}\big\{\sigma^2(y_{ij})-\sigma_{ij}^2\big\}v_{ij}, \quad\quad x\in \Omega_N,$$
where $\sigma^2(y_{ij})= e^{y_{ij}}/(1+e^{y_{ij}})^2$. Note when $y$ is equal to $M_N^*$ or when $v$ is a zero vector, then $U_N(y, v)=0.$  Further denote that 
$w_{i+}=\sum_{j \in S_J(i)} w_{ij},$ $w_{+j}=\sum_{i\in S_N(j)}w_{ij}$ and
$w_{++}=\sum_{i=1}^N\sum_{j \in S_J(i)}w_{ij}$, where $S_N(j)=\{i=1,...,N: z_{ij}=1\}$. We first give proof for Theorem \ref{thm-existence} below.
\begin{proof}[Proof of Theorem \ref{thm-existence}]

We start with establishing the existence of $\hat{M}_{N}$ by applying the fixed point theorems of {\citet[pages 695-711]{kantorovich1964lv}}.
We start with constructing a function $F_N$ on $\Omega_N$ with a fixed point $\hat{M}_N$.
Consider $F_N(y)=y+r_N(y)$ for $y \in \Omega_{N}$, where $r_N: \Omega_N \mapsto \Omega_N$ is defined by the equation,
$$[x, r_N(y)]_\sigma=\sum_{i=1}^N \sum_{j\in S_J(i)}x_{ij}\big\{Y_{ij}-E(y_{ij})\big\},\quad\quad x\in\Omega_N,$$
where $E(y_{ij})=e^{y_{ij}}/(1+e^{y_{ij}}).$
Note that $F_N$ has a fixed point $\omega\in \Omega_N$ if and only if 
$$\sum_{i=1}^N\sum_{j\in S_J(i)}x_{ij}\big\{Y_{ij}-E(\omega_{ij})\big\}=0,\quad\quad x\in \Omega_N.$$
Let $P$ be the orthogonal projection onto $\Omega_N$. 
Let $\hat{E}=\{E(\hat{m}_{ij}): i=1,...,N, j=1,...,J, z_{ij}=1\}$ and $Y_{z}=\{Y_{ij}: i=1,...,N, j=1,...,J, z_{ij}=1\}.$ Then following from {\citet[pages 196-198]{berk1972consistency}}, $\hat{M}_N$ is a maximum likelihood estimator of $M_N^*$ if and only if $P\hat{E}=PY_z.$ Hence, $\hat{M}_N$ exists if and only if $\omega$ exists. 
Furthermore, since the log-likelihood $l(Y_z, \cdot)$ is strictly concave,
if the maximum likelihood estimator $\hat{M}_N$ of $M_N^*$ exists, then it must be unique. Therefore, if $\hat{M}_N$ exists, $\omega=\hat{M}_N.$ 
So, we just need to verify the conditions of the fixed point theorem to show that the fixed point $\omega$ indeed exists.

The Kantorovich \& Akilov's fixed point theorem requires construction of a sequence that converges to the fixed point. 
Consider the sequence $\{t_{Nk}: k=0,1,...\}$, with $t_{N0}=M_N^*$ and $t_{N(k+1)}=F_N(t_{Nk})$ for $k=0,1,...$
Note that $t_{N1}=M_N^*+R_N$. 
To check whether this sequence is well-defined and converges to $\hat{M}_N$, we need to examine the differential $dF_{Ny}$ of $F_N$ at $y \in \Omega_N$. Note that 
for $y+v \in \Omega_N$,
\begin{align*}
[x, F_N(y+v)-F_N(y)]_{\sigma} &= \sum_{i=1}^{N}\sum_{j \in S_{J}(i)}x_{ij}\sigma_{ij}^2\Big[v_{ij}+(\sigma_{ij}^2)^{-1}\big\{E(y_{ij})-E(y_{ij}+v_{ij})\big\}\Big]\\
&=-[x, U_N(y, v)]_{\sigma} + o(v),
\end{align*}
where $o(v)/\|v\|_{\sigma} \to 0$ as $\|v\|_{\sigma} \to 0$. It follows that $dF_{Ny}(v)=-U_N(y, v)$.
Denote $\|dF_{Ny}\|_{\sigma}(A)$ to be the smallest nonnegative number such that 
$$\|dF_{Ny}(v)\|_{\sigma}(A)\leq \|dF_{Ny}\|_{\sigma}(A)\|v\|_{\sigma}(A),\quad \quad v\in \Omega_N.$$
Let $A_{p}$ be the set consisting of all the point maps $f_{ij}$ on $\Omega_N$, i.e. $f_{ij}(x)=x_{ij}$ for any $x\in \Omega_N$. By Lemma  \ref{lemma-sequence-and-bound}(c) below, there exist sequences $f_N$ and $d_N$ such that 
$$
\|dF_{Ny}\|_{\sigma}(A_{p}) \leq d_N \|y-M_N^*\|_{\sigma}(A_{p}) \quad \text{whenever}\quad
\|y-M_N^*\|_{\sigma}(A_{p}) \leq f_N, \quad y\in \Omega_N.
$$
\begin{lemma}\label{lemma-sequence-and-bound}
Assume Conditions \ref{cond:bound}--\ref{cond:speed} hold. If $A_p=\{f_{ij}: i=1,...,N, j=1,...,J, z_{ij}=1\}$ such that $f_{ij}(x)=x_{ij}$ for $x \in \Omega_N$. Let $C_N=|A_p|$, the cardinality of $A_p$.  Then there exist sequences $f_N> 0$ and $d_N\geq 0$ satisfying the followings.

\noindent(a). As $N \to \infty$, $f_N^2/\log C_N \to \infty.$ 

\noindent(b). As $N \to \infty$, $f_N^2(N_{*}^{-1}+J_{*}^{-1}) \to 0.$

\noindent(c). If $y, v \in \Omega_N$ and $\|y-M_N^*\|_{\sigma}(A_p)\leq f_N$, then there exists $n<\infty$ such that for all $ N > n$, $\|U_N(y, v)\|_{\sigma}(A_p) \leq d_N \|y-M_N^*\|_{\sigma}(A_p)\|v\|_{\sigma}(A_p).$ Furthermore, $d_Nf_N \to 0$ as  $N \to \infty$.
\end{lemma}

As shown in {\citet[pages 695-711]{kantorovich1964lv}}, if $\|R_N\|_{\sigma}(A_{p}) < \frac{1}{2}f_N$ and $d_N\|R_N\|_{\sigma}(A_{p}) < \frac{1}{2}$, then $\hat{M}_N$ exists. 
By Lemma \ref{lemma-bound-Zn} below, we have pr$(\|R_N\|_{\sigma}(A_{p}) < \frac{1}{2}f_N) \to 1$ as $N \to \infty$. Therefore, it follows from Lemma \ref{lemma-sequence-and-bound}(c) that with probability tending to 1,
$d_N\|R_N\|_{\sigma}(A_{p})<\frac{1}{2}f_Nd_N\to 0$.

\begin{lemma}\label{lemma-bound-Zn}
Let $A \subset \Omega_{N}^*.$ 
Let $C_N$ denote the cardinality of $A$. 
If there exist sequences $f_N> 0$ and $d_N\geq 0$ satisfying
(a). $0<C_N<\infty$ and $f_N^2/\log C_n \to \infty$ as $N \to \infty,$ (b). If $y, v \in \Omega_N$ and $\|y-M_N^*\|_{\sigma}(A)\leq f_N$, then there exists $n<\infty$ such that for all $ N > n$, $\|U_N(y, v)\|_{\sigma}(A) \leq d_N \|y-M_N^*\|_{\sigma}(A)\|v\|_{\sigma}(A),$ (c). $d_Nf_N \to 0$ as  $N \to \infty$.
Then pr$\big(\|R_N\|_{\sigma}(A) < \frac{1}{2}f_N \big) \to 1$ as $N \to \infty.$
\end{lemma}

Hence, the conditions of the fixed point theorem are satisfied with probability approaching 1. It then follows that the maximum likelihood estimators $\hat{M}_N$ exists with probability tending to 1. Since Condition~\ref{cond:connect} holds, as a direct consequence of Proposition \ref{prop:connect}, the corresponding maximum likelihood estimators $\hat{\theta}_i$, $i=1,...,N$ and $\hat{\beta}_j$, $j=1,...,J$ can be uniquely determined given $\hat{M}_N$. Therefore, with probability approaching 1 that they all exist, as $N\to \infty$. The first part of the theorem then follows. 

Now we seek to prove the consistency results. 
Taking sequences $f_N$ and $d_N$ again as satisfying the results in Lemma \ref{lemma-sequence-and-bound} and $A=A_p.$ Then both Lemmas \ref{lemma-bound-Zn} and \ref{lemma-bound-diff-btw-estimate-truth} hold. 
From the results of Lemmas \ref{lemma-bound-Zn} and \ref{lemma-bound-diff-btw-estimate-truth}, it can be implied that as $N \to \infty$, with probability tending to 1 that, 
\begin{align*}
 \|\hat{M}_N-M_N^*\|_{\sigma}(A_{p})=O(f_N).\numberthis\label{eq: bound-parameter-diff} 
\end{align*}
From {\citet[pages 822-824]{haberman1977maximum}}, $\sigma(g)$ is in fact the standard deviation of $g(\hat{M}_N)$. We further note by Lemma \ref{lemma-peopleitemdiffvar} below,
\begin{align*}
 \max_{g\in A_p} \sigma(g)\leq\tau_2^{-1}(N_{*}^{-1}+J_{*}^{-1})^{\frac{1}{2}},\numberthis\label{eq: upper-bound-var-mij}
\end{align*} 
for some $0<\tau_2<\infty$.
\begin{lemma}\label{lemma-bound-diff-btw-estimate-truth}
Assume Conditions \ref{cond:bound}--\ref{cond:speed} hold.
Let $A \subset \Omega_{N}^*.$  If there exist sequences $f_N > 0$ and $d_N\geq 0$ satisfying
(a). pr$\big(\|R_N\|_{\sigma}(A) < \frac{1}{2}f_N \big) \to 1$ as $N \to \infty,$ (b). If $y, v \in \Omega_N$ and $\|y-M_N^*\|_{\sigma}(A)\leq f_N$, then there exists $n<\infty$ such that for all $ N > n$, $\|U_N(y, v)\|_{\sigma}(A) \leq d_N \|y-M_N^*\|_{\sigma}(A)\|v\|_{\sigma}(A),$ (c). $d_Nf_N \to 0$ as  $N \to \infty$. Then, as $N\to \infty$, with probability approaching 1 that,
\begin{align*}
\Big\vert\frac{\|\hat{M}_N-M_N^*\|_{\sigma}(A)}{\|R_N\|_{\sigma}(A)}-1\Big\vert\leq d_N^{\frac{1}{2}}\to 0 \quad\text{and}\quad
\|\hat{M}_N-M_N^*-R_N\|_{\sigma}(A)\leq d_N\|R_N\|_{\sigma}^2(A).
\end{align*}
\end{lemma}

\begin{lemma}\label{lemma-peopleitemdiffvar}
Assume Conditions \ref{cond:bound}--\ref{cond:speed} hold and $\sum_{i=1}^N\theta_i=0$, the asymptotic variance of the maximum likelihood estimator of $m_{ij}^*,$ $var(\hat{m}_{ij})$,  for any $i=1,...,N$ and $j=1,...,J$, takes the form,
\begin{align*}
var(\hat{m}_{ij}) = (\sigma_{i+}^2)^{-1}+(\sigma_{+j}^2)^{-1}+O (N_{*}^{-1}J_{*}^{-1} )\quad \text{as} \quad N \to \infty.
\end{align*}
\end{lemma}
Then as $N\to \infty$, we have with probability approaching 1 that
\begin{align*}
\max_{i,j, z_{ij=1}}\vert \hat{m}_{ij}-m_{ij}^* \vert 
&=\max_{i,j, z_{ij=1}}|f_{ij}(\hat{M}_N)-f_{ij}(M_N^*)|\\
&=\max_{i,j, z_{ij=1}}|f_{ij}(\hat{M}_N-M_N^*)|\\
&\leq \max_{i,j, z_{ij=1}}\sigma(f_{ij}) \|\hat{M}_N-M_N^*\|_{\sigma}(A_p)\\
&\leq \|\hat{M}_N-M_N^*\|_{\sigma}(A_p)\Big\{\max_{g\in A_p}\sigma(g)\Big\} \\
&=O\Big\{f_N\left(N_{*}^{-1}+J_{*}^{-1}\right)^{\frac{1}{2}}\Big\}\\
&\to 0. \numberthis\label{eq:consistency-of-mij}
\end{align*}
The second last line follows from \eqref{eq: bound-parameter-diff} and \eqref{eq: upper-bound-var-mij} and the last line follows from Lemma \ref{lemma-sequence-and-bound}(b). 


{By Proposition 1, given $\hat{m}_{ij}$ for $i=1,...,N, j=1,...,J, z_{ij}=1$,  all the $\hat{\theta}_i$, $i=1,...,N$ and $\hat{\beta}_j$, $j=1,...,J$ can be uniquely determined.
Since \eqref{eq:consistency-of-mij} holds, as a direct consequence of the Slutsky Theorem, we have with probability tending to 1 that  $\|\hat{\theta}-\theta^*\|_{\infty} \to 0$ and $\|\hat{\beta}-\beta^*\|_{\infty} \to 0$ as $N \to \infty$. From here, we have $\max_{i,j} \vert \hat{m}_{ij}-m_{ij}^* \vert \rightarrow 0$. Hence we complete the proof of consistency results in this theorem. }

\end{proof}

{The proof of Theorem \ref{thm-rate} is a continuum of the proof of Theorem \ref{thm-existence} with additional conditions. We next present the proof of Theorem \ref{thm-rate}.}
\begin{proof}[Proof of Theorem \ref{thm-rate}]

To derive explicit rates of convergence for $\|\hat{\theta}-\theta^*\|_{\infty}$ and  $\|\hat{\beta}-\beta^*\|_{\infty}$, we adopt a similar approach as in the derivation of convergence of $\max_{i,j: z_{ij}=1}|\hat{m}_{ij}-m_{ij}^*|.$ In particular, for the column parameters $\beta_j$, we consider linear functions $g_j \in \Omega_{N}^*$ such that $g_j(x)=\beta_j.$ 
We can construct $g_j$ as follows. The idea is to include all the row parameters $\theta_i$ so as to use the identifiability constraint $\sum_{i=1}^{N}\theta_i=0$. 
For any $i \in S_{N}(j),$ we use $m_{ij}=\theta_i-\beta_j$ in the construction. While for each $i \in S_{N_{\phi}}(j),$ where $S_{N_{\phi}}(j)=\{1,2,...,N\}\setminus S_{N}(j),$ by Condition~\ref{cond:connect}, there must exist $1\leq i_{i1}, i_{i2},...,i_{ik}\leq N$ and $1\leq j_{i1}, j_{i2},...,j_{ik}\leq J$ such that \begin{align*}
    z_{i,j_{i1}}=z_{i_{i1},j_{i1}}=z_{i_{i1},j_{i2}}=z_{i_{i2},j_{i2}}=...=z_{i_{ik},j_{ik}}=z_{i_{ik},j}=1.
\end{align*}
Therefore, we can construct $g_j$ as
\begin{align*}
g_j(x)=&-\frac{1}{N}\Big\{\sum_{i\in S_N(j)}m_{ij}\\
&+\sum_{i\in S_{N_{\phi}}(j)}\Big(m_{i,j_{i1}}-m_{i_{i1},j_{i1}}+m_{i_{i1},j_{i2}}-m_{i_{i2},j_{i2}}+...-m_{i_{ik},j_{ik}}+m_{i_{ik},j}\Big)\Big\}\\
=&\beta_j.
\end{align*}
Let
$
  A_{\beta}=\big\{g_j : j=1,...,J\big\}.
$
Now consider a sequence $f_N$ satisfying the rate requirements $f_N^2/\log J \to \infty$ and $f_N^2N_{*}^{-1/2}\to 0$ as $N \to \infty$. Then by Lemma \ref{lemma-sequence-and-bound-A-beta} below, we can pick a sequence $d_N$ satisfying  Lemma \ref{lemma-sequence-and-bound-A-beta}(a) and  Lemma \ref{lemma-sequence-and-bound-A-beta}(b).
Furthermore, by Lemma \ref{lemma-itemvar} below, we know that $\sigma^2(g_j)=(\sigma_{+j}^2)^{-1}+O(N_{*}^{-1}J_{*}^{-1})$ for any $g_j \in A_{\beta}$. Therefore, there exist positive $0<c_2<\infty$ and some $n$ such that for all $N>n,$
\begin{align*}
 \max_{j=1,...,J}\sigma(g_j) < c_2^{-1}N_{*}^{-\frac{1}{2}}.
\end{align*}

\begin{lemma}\label{lemma-sequence-and-bound-A-beta}
 Assume {Conditions \ref{cond:bound}--\ref{cond:speed2}} hold. If $A_{\beta}=\{g_{j}: j=1,...,J\}$ such that $g_j \in \Omega_{N}^*$ and $g_j(x)=\beta_j$ for $x \in \Omega_N$. Let $C_N=|A_{\beta}|=J$ be the cardinality of $A_{\beta}$.  
For any positive sequence $f_N$ such that $f_N^2/\log J \to \infty$ and $f_N^2N_{*}^{-1/2}\to 0$ as $N \to \infty$, there exists a sequence $d_N \geq 0$ satisfying the followings. 

\noindent(a). If $y, v \in \Omega_N$ and $\|y-M_N^*\|_{\sigma}(A_{\beta})\leq f_N$, then there exists $n<\infty$ such that for all $ N > n$, $\|U_N(y, v)\|_{\sigma}(A_{\beta}) \leq d_N \|y-M_N^*\|_{\sigma}(A_{\beta})\|v\|_{\sigma}(A_{\beta}).$ 

\noindent(b). $d_Nf_N^2 \to 0$ as  $N \to \infty$.



\end{lemma}

\begin{lemma}\label{lemma-itemvar}
Assume {Conditions \ref{cond:bound}--\ref{cond:speed2}} hold and $\sum_{i=1}^N\theta_i=0$. The asymptotic variance of the maximum likelihood estimator of an individual column parameter, var$(\hat{\beta}_j)$, asymptotically attains the oracle variance $(\sigma_{+j}^{2})^{-1}$ in the sense that
\begin{align*}
var(\hat{\beta}_j)=(\sigma_{+j}^{2})^{-1}+O (N_{*}^{-1}J_{*}^{-1} )\quad \quad \text{as} \quad N \to \infty.
\end{align*}
\end{lemma}

Note that by taking sequences $f_N$ and $d_N$ satisfying the conditions in Lemma \ref{lemma-sequence-and-bound-A-beta} and setting $A=A_{\beta},$ it can be shown easily that the results of Lemmas \ref{lemma-bound-Zn} and \ref{lemma-bound-diff-btw-estimate-truth} still hold. 
Hence,
it can be implied that as $N\to \infty,$ with probability tending to 1,
\begin{align*}
    \|\hat{M}_N-M_N^*\|_{\sigma}(A_{\beta})=O(f_N).
\end{align*}
Then as $N\to \infty$, we have with probability approaching 1 that, 
\begin{align*}
    \max_{j=1,...,J}|\hat{\beta}_{j}-\beta_{j}^*|&= \max_{j=1,...,J}|g_j(\hat{M}_N)-g_j(M_N^*)|\\
    &= \max_{j=1,...,J}|g_j(\hat{M}_N-M_N^*)|\\
    &\leq  \|\hat{M}_N-M_N^*\|_{\sigma}(A_{\beta})\max_{j=1,...,J}\sigma(g_j)\\
    &<  c_2^{-1} N_{*}^{-\frac{1}{2}}\|\hat{M}_N-M_N^*\|_{\sigma}(A_{\beta})\\
    &=O\Big\{(\log J)^{\frac{1}{2}}N_{*}^{-\frac{1}{2}}\Big\} \quad \text{as} \quad N \to \infty,
\end{align*}
where the last step can be implied from {the fact that $\|\hat{M}_N-M_N^*\|_{\sigma}(A_{\beta})=O\big(f_N\big)$ and the rate requirement of $f_N$ in Lemma \ref{lemma-sequence-and-bound-A-beta}, where the minimum order of $f_N$ is determined by $f_N^2/\log J \to \infty$ as $N\to \infty$. Specifically, it can be verified that for any $f_N$ satisfying $f_N^2/\log J \to \infty$, if $\|\hat{M}_N-M_N^*\|_{\sigma}(A_{\beta})=O\big(f_N\big)$, then $\|\hat{M}_N-M_N^*\|_{\sigma}(A_{\beta})=O\{(\log J)^{1/2}\}$.} Therefore,
\begin{align*}
\|\hat{\beta}-\beta^*\|_{\infty}=O_p\Big\{(\log J)^{\frac{1}{2}}N_{*}^{-\frac{1}{2}}\Big\}. \numberthis\label{eq: rate-beta}
\end{align*}

Now for the row parameters $\theta_i$, we adopt a similar strategy by constructing linear functions $g_i \in \Omega_{N}^*$ such that $g_i(x)=\theta_i.$

In specific, we can construct the linear function $g_i$ as follows. 
\begin{align*}
g_i(x)&=\frac{1}{\vert S_J(i)\vert }\sum_{j\in S_J(i)}\{m_{ij}+g_j(x)\}=\frac{1}{\vert S_J(i)\vert }\sum_{j\in S_J(i)}(\theta_i-\beta_j+\beta_j)=\theta_i,
\end{align*}
where $\vert S_J(i)\vert$ denotes the cardinality of $S_J(i).$ Let $A_{\theta}$ consist of $g_i, i=1,...,N$, i.e. $A_{\theta}=\big\{g_i: i=1,...,N\big\}.$ 
Take a positive sequence $f_N$ satisfying the rate requirements $f_N^2/\log N \to \infty$ and $f_N^2J_{*}^{-1}\to 0$ as $N \to \infty$, then by Lemma \ref{lemma-sequence-and-bound-A-theta} below, 
we can pick a sequence $d_N$ satisfying Lemma \ref{lemma-sequence-and-bound-A-theta}(a) and Lemma \ref{lemma-sequence-and-bound-A-theta}(b).
Furthermore, by Lemma \ref{lemma-peoplevar} below, we know that $\sigma^2(g_i)=(\sigma_{i+}^2)^{-1}+O (N_{*}^{-1}J_{*}^{-1} )$ for any $g_i \in A_{\theta}$. Hence, there exist positive $0< \gamma_2<\infty$ and such that 
$$\max_{i=1,...,N}\sigma(g_i) < \gamma_2^{-1} J_{*}^{-1/2}.$$

\begin{lemma}\label{lemma-sequence-and-bound-A-theta}
Assume {Conditions \ref{cond:bound}--\ref{cond:speed2}} hold. If $A_{\theta}=\{g_{i}: i=1,...,N\}$ such that $g_i \in \Omega_{N}^*$ and $g_i(x)=\theta_i$ for $x \in \Omega_N$. Let $C_N=|A_{\theta}|=N$ be the cardinality of $A_{\theta}$.  Then for any positive sequence $f_N$ such that $f_N^2/\log N \to \infty$ and $J_{*}^{-1}f_N^2 \to 0$ as $N\to \infty,$ there exists a sequence $d_N \geq 0$ satisfying the followings. 

\noindent(a). If $y, v \in \Omega_N$ and $\|y-M_N^*\|_{\sigma}(A_{\theta})\leq f_N$, then there exists $n<\infty$ such that for all $ N > n$, $\|U_N(y, v)\|_{\sigma}(A_{\theta}) \leq d_N \|y-M_N^*\|_{\sigma}(A_{\theta})\|v\|_{\sigma}(A_{\theta}).$ 

\noindent(b). $d_Nf_N \to 0$ as  $N \to \infty$.
\end{lemma}

\begin{lemma}\label{lemma-peoplevar}
Assume {Conditions \ref{cond:bound}--\ref{cond:speed2}} hold and $\sum_{i=1}^N\theta_i=0$, the asymptotic variance of an individual row parameter, var$(\hat{\theta}_i)$, asymptotically attains oracle variance $(\sigma_{i+}^2)^{-1}$ in the sense that
\begin{align*}
 var(\hat{\theta}_i)=(\sigma_{i+}^2)^{-1}+O (N_{*}^{-1}J_{*}^{-1} )\quad\quad \text{as}\quad N\to \infty.
\end{align*}
\end{lemma}
Note that by taking sequences $f_N$ and $d_N$ satisfying the conditions in Lemma \ref{lemma-sequence-and-bound-A-theta} and setting $A=A_{\theta},$ it can be implied easily that Lemmas \ref{lemma-bound-Zn} and \ref{lemma-bound-diff-btw-estimate-truth} still hold. Similarly, from pr$(\|R_N\|_{\sigma}(A_{\theta})<\frac{1}{2}f_N) \to 1$ and the results of Lemma \ref{lemma-bound-diff-btw-estimate-truth}, it can be implied as $N\to \infty,$ we have with probability tending to 1 that, 
\begin{align*}
    \|\hat{M}_N-M_N^*\|_{\sigma}(A_{\theta})=O(f_N).
\end{align*}
It follows, as $N\to \infty$, we have with probability approaching 1 that,
\begin{align*}
    \max_{i=1,...,N}|\hat{\theta}_{i}-\theta_{i}|&=\max_{i=1,...,N}|g_i(\hat{M}_N)-g_i(M_N^*)|\\
    &=\max_{i=1,...,N}|g_i(\hat{M}_N-M_N^*)|\\
    &\leq \|\hat{M}_N-M_N^*\|_{\sigma}(A_{\theta})\max_{i=1,...,N}\sigma(g_i)\\
    &<\gamma_2^{-1} J_{*}^{-\frac{1}{2}}\|\hat{M}_N-M_N^*\|_{\sigma}(A_{\theta})\\
    &=O\Big\{(\log N)^{\frac{1}{2}}J_{*}^{-\frac{1}{2}}\Big\}\quad \text{as}\quad N \to \infty,
\end{align*}
where the last step can be implied from {the fact that with probability tending to 1, $\|\hat{M}_N-M_N^*\|_{\sigma}(A_{\theta})=O\big(f_N\big)$, and  the rate requirement of $f_N$ in Lemma \ref{lemma-sequence-and-bound-A-theta}, where the minimum order of $f_N$ is determined by $f_N^2/\log N \to \infty$. Specifically, it can be verified that for any $f_N$ satisfying $f_N^2/\log N \to \infty$, if $\|\hat{M}_N-M_N^*\|_{\sigma}(A_{\theta})=O\big(f_N\big)$, then $\|\hat{M}_N-M_N^*\|_{\sigma}(A_{\theta})=O\{(\log N)^{1/2}\}$.}
It follows that,
\begin{align*}
\|\hat{\theta}-\theta^*\|_{\infty}=O_p\Big\{(\log N)^{\frac{1}{2}}J_{*}^{-\frac{1}{2}}\Big\}. \numberthis\label{eq: rate-theta}
\end{align*}
Combining \eqref{eq: rate-beta} and \eqref{eq: rate-theta}, we have 
$
    \max_{i,j} \vert \hat m_{ij} - m_{ij}^*\vert =  O_p\big\{(\log J)^{\frac{1}{2}}N_{*}^{-\frac{1}{2}} + (\log N)^{\frac{1}{2}}J_{*}^{-\frac{1}{2}}  \big\}.
$
Therefore, we complete the proof of the theorem.
\end{proof}
\newpage
Next, we give proof for Theorem \ref{thm-3-sufficient-condition} below.
\begin{proof}[Proof of Theorem \ref{thm-3-sufficient-condition}]
We first seek to show $\vert \sigma^2(g)/\tilde{\sigma}^2(g)-1\vert\to 0$ as $N\to \infty$, where $\sigma^2(g)=\sigma\{g(\hat{M})\}$. 
Since Conditions \ref{cond:bound}--\ref{cond:speed2} hold and $\Vert w_g\Vert_1, \Vert \tilde w_g\Vert_1 < C$, by Lemma \ref{sufficient-condition-var-approx} below, 
\begin{align*}
    \vert\sigma^2(g)-\tilde{\sigma}^2(g)\vert=O(N_{*}^{-1}J_{*}^{-1})\quad  \text{as}\quad N\to \infty.\numberthis\label{eq: diff-approxi-true-var}
\end{align*} 
Hence, it follows
\begin{align*}
\Big\vert \frac{\sigma^2(g)}{\tilde{\sigma}^2(g)}-1\Big\vert&=\frac{\vert\sigma^2(g)-\tilde{\sigma}^2(g)\vert}{\tilde{\sigma}^2(g)} \to 0\quad \text{as}\quad N\to \infty,
\end{align*}
where the last step follows from \eqref{eq: diff-approxi-true-var} and the definition of $\tilde \sigma^2(g)$.
\begin{lemma}\label{sufficient-condition-var-approx}
Assume {Conditions \ref{cond:bound}--\ref{cond:speed2}} hold and $\sum_{i=1}^{N}\theta_i=0$. Consider a linear function $g: \Omega_N \mapsto \mathbb{R}$ with  $g(x)=\sum_{i=1 }^Nh_i\theta_i + \sum_{j=1}^Jh_{j}'\beta_j.$  If there exists a positive $C < \infty$ such that $\sum_{i=1}^N\vert h_i\vert<C$ and $\sum_{j=1}^J \vert h_j'\vert <C$, then 
\begin{align*}
\sigma^2(g)=\sum_{i=1}^Nh_{i}^2(\sigma_{i+}^2)^{-1}+\sum_{j=1 }^Jh_{j}'^2(\sigma_{+j}^2)^{-1}+O(N_{*}^{-1}J_{*}^{-1})\quad \text{as}\quad N\to \infty.
\end{align*}
\end{lemma}

Then if we can show $\sigma(g)^{-1}\{g(\hat{M})-g(M^*)\} \to $ N$(0,1)$ in distribution, the first part of the theorem would follow directly. 
As a direct application of Proposition~\ref{prop:connect}, we can re-write function $g$ on $\Omega_N$ using $[\cdot, \cdot]_{\sigma}$ as follows. 
Let $c_N \in \Omega_N$ be defined by the equation $$g(x)=[c_N, x]_{\sigma}=\sum_{i=1}^N\sum_{j \in S_J(i)}c_{ij}x_{ij}\sigma_{ij}^2,\quad \quad x\in \Omega_N.$$ 
Then we can express,
\begin{align*}
g(\hat{M}_N)-g(M_N^*)&=g\big(\hat{M}_N-M_N^*\big)=\big[c_N, \hat{M}_N-M_N^*\big]_{\sigma}\\&=\big[c_N, \hat{M}_N-M_N^*-R_N\big]_{\sigma}+\big[c_N, R_N\big]_{\sigma}. \numberthis\label{eq:g-M1}
\end{align*}
Recall that $\sigma(g)=\sup_{x \in \Omega_N}\big\{\vert [c_N, x]_{\sigma}\vert : \|x\|_{\sigma}\leq 1\big\},$ the supremum is attained at $x=c_N/\|c_N\|_{\sigma}$, so $\sigma(g)=\|c_N\|_{\sigma}$.
We consider two possible cases, $w_g=0$ in case 1 and $w_g\neq 0$ in case 2, and we seek to prove the result of the theorem hold under both cases separately. 

We first consider case 1. Similar as in the proof of Theorem~\ref{thm-existence}, we consider a set $A_{\beta}$ consisting of linear functions $g_j \in \Omega_{N}^*$ on $\Omega_N$ such that $g_j(x)=\beta_j$ with $A_{\beta}=\{g_j:j=1,...,J\}.$
We now pick a positive sequence $f_N$ satisfying $f_N^2/\log J \to \infty$ and $f_N^2 N_{*}^{-1/2} \to 0$ as $N\to \infty.$ Then by Lemma \ref{lemma-sequence-and-bound-A-beta}, we can pick a sequence $d_N\geq 0$ satisfying Lemma \ref{lemma-sequence-and-bound-A-beta}(a) and Lemma \ref{lemma-sequence-and-bound-A-beta}(b).
Furthermore, it can be implied that Lemmas \ref{lemma-bound-Zn} and \ref{lemma-bound-diff-btw-estimate-truth} still hold by taking $A=A_{\beta}$. 
Moreover, Lemma \ref{lemma-itemvar} and Condition \ref{cond:speed}(c) imply that there exist $0<\gamma_1, \gamma_2<\infty$ and some $n$ such that for all $N>n,$
\begin{align*}
 \gamma_1^{-1}N_{*}^{-\frac{1}{2}}<\sigma(g_j) < \gamma_2^{-1}N_{*}^{-\frac{1}{2}}, \quad g_j\in A_{\beta}.\numberthis\label{eq: bound-var-single-item}   
\end{align*}
Now for any $x \in \Omega_N,$
\begin{align*}
\vert g(x)\vert &= \vert \Tilde w_g^T \beta \vert \\
&\leq \|\Tilde w_g\|_1\max_{ j=1,...,J}\{\vert\beta_j\vert\}\\
&\leq C \max_{g_j\in A_{\beta}}\{\vert g_j(x)\vert \}\\
&=C\max_{g_j\in A_{\beta}}\Big\{\frac{\vert g_j(x)\vert}{\sigma(g_j)}\sigma(g_j)\Big\}\\
&\leq C\Big\{\max_{g_j\in A_{\beta}}\frac{\vert g_j(x)\vert }{\sigma(g_j)}\Big\}\max_{g_j\in A_{\beta}}\sigma(g_j)\\
&=C\|x\|_{\sigma}(A_{\beta})\max_{g_j\in A_{\beta}}\sigma(g_j)\\
&\leq C \gamma_2^{-1}N_{*}^{-\frac{1}{2}}\|x\|_{\sigma}(A_{\beta}),\numberthis\label{eq:tau(A-beta, g)}
\end{align*}
where the second last step follows from the definition of $\|\cdot\|_{\sigma}(A_{\beta})$ and the last step follows from \eqref{eq: bound-var-single-item}. 
Since case 1 assumes $w_g=0$, so $g(M)\neq 0$ implies $\tilde w_g \neq 0$. Then as a direct consequence of Lemma \ref{sufficient-condition-var-approx}, there exists some $0<\gamma_{3}<\infty$ such that for all $N>n,$
\begin{align*}
\sigma(g)\geq \gamma_3 N_{*}^{-\frac{1}{2}}.\numberthis\label{eq:lower-bound-g-beta}
\end{align*}
As a result of \eqref{eq:tau(A-beta, g)}, we have
\begin{align*}
\Big|\big[c_N, \hat{M}_N-M_N^*-R_N\big]_{\sigma}\Big| \leq C \gamma_2^{-1} N_{*}^{-\frac{1}{2}}\|\hat{M}_N-M_N^*-R_N\|_{\sigma}(A_{\beta}).    \numberthis\label{eq: norm-bound-innproduct-beta}
\end{align*}
Note that from \eqref{eq:g-M1},
\begin{align*}
\frac{g(\hat{M}_N)-g(M_N^*)}{\sigma(g)}&=\frac{\big[c_N, \hat{M}_N-M_N^*-R_N\big]_{\sigma}+\big[c_N, R_N\big]_{\sigma}}{\sigma(g)}
\end{align*}
Rearrange gives as $N \to \infty$, with probability tending to 1 that,
\begin{align*}
\Big\vert\frac{g(\hat{M}_N)-g(M_N^*)}{\sigma(g)}-\frac{\big[c_N, R_N\big]_{\sigma}}{\sigma(g)}\Big\vert 
&=\frac{\Big\vert[c_N, \hat{M}_N-M_N^*-R_N]_{\sigma}\Big\vert}{\sigma(g)}\\
&\leq \frac{C \gamma_2^{-1} N_{*}^{-\frac{1}{2}}}{\sigma(g)}\|\hat{M}_N-M_N^*-R_N\|_{\sigma}(A_{\beta})\\
&\leq C\gamma_2^{-1}\gamma_3^{-1}d_N\big[\|R_N\|_{\sigma}(A_{\beta})\big]^2\\
&\leq \frac{1}{4}C\gamma_2^{-1}\gamma_3^{-1}d_Nf_N^2\\
&\to 0,
\numberthis\label{ineq: equivalent-normal-beta}
\end{align*}
where the second line follows from \eqref{eq: norm-bound-innproduct-beta}, the third line can be obtained from \eqref{eq:lower-bound-g-beta} and Lemma \ref{lemma-bound-diff-btw-estimate-truth}, the second last line can be implied by Lemma \ref{lemma-bound-Zn} and the last line follows from Lemma \ref{lemma-sequence-and-bound-A-beta}.
Hence, it turns out that it suffices to show $\big[c_N, R_N\big]_{\sigma}/\sigma(g) \to $ N$(0, 1)$.
Write $Z_N=[c_N, R_N]_{\sigma}/\sigma(g)=\sum_{i=1}^N\sum_{j\in S_J(i)}\big\{c_{ij}(Y_{ij}-E_{ij})\big\}/\|c_N\|_{\sigma}$ for simplicity. The strategy is to show the moment generating function of $Z_N$, denoted as $G_{Z_N}(t)$, converges to $\exp\{t^2/2\}$, the moment generating function of the standard Gaussian.
Write $c_{ij}'=c_{ij}/\|c_N\|_{\sigma}=c_{ij}/\sigma(g)$ for simplicity.
We consider the log moment generating function of $Z_N$, 
\begin{align*}
\log G_{Z_N}(t)
&=\log \mathbb{E}\big[e^{tZ_N}\big] =\log \mathbb{E}\Big[\exp\Big\{\frac{t}{\sigma(g)}\sum_{i=1}^N\sum_{j \in S_J(i)}c_{ij}(Y_{ij}-E_{ij})\Big\}\Big]\\
&=-t\sum_{i=1}^N\sum_{j \in S_J(i)}c_{ij}'E_{ij}+\log \prod_{i=1}^N\prod_{j \in S_J(i)}\mathbb{E}\big\{\exp(tc_{ij}'Y_{ij})\big\}\\
&=-t\sum_{i=1}^N\sum_{j \in S_J(i)}c_{ij}'E_{ij}+\sum_{i=1}^N\sum_{j \in S_J(i)}\log \mathbb{E}\big\{\exp(tc_{ij}'Y_{ij})\big\}\\
&=\sum_{i=1}^N\sum_{j \in S_J(i)}\Big[\log\big\{1+\exp(m_{ij}^*)\big\}^{-1}-\log\big\{1+\exp(tc_{ij}'+m_{ij}^*)\big\}^{-1}-tc_{ij}'E_{ij}\Big]\\
&=\sum_{i=1}^N\sum_{j \in S_J(i)}\Big[\log\big\{h(m_{ij}^*)\big\}-\log\big\{h(tc_{ij}'+m_{ij}^*)\big\}-tc_{ij}'E_{ij}\Big],\numberthis\label{eq:log-mgf}
\end{align*}
where $h(m_{ij})=\{1+\exp(m_{ij})\}^{-1}$. 
We can then apply Taylor expansion
to $\log\{h(tc_{ij}'+m_{ij}^*)\}$ about $m_{ij}^*$. For some $t'=\alpha t$ with $0 < \alpha < 1$,
\begin{align*}
\log\{h(tc_{ij}'+m_{ij}^*)\}&=\log \{h(m_{ij}^*)\} -E_{ij}tc_{ij}'-\frac{t^2}{2}c_{ij}'^2\sigma^2(m_{ij}^*+t'c_{ij}').
\end{align*}
Substitute into Equation \eqref{eq:log-mgf},
\begin{align*}
\log G_{Z_N}(t) &=\frac{t^2}{2}\sum_{i=1}^N\sum_{j \in S_J(i)}c_{ij}'^2\sigma^2(m_{ij}^*+t'c_{ij}'),\quad \quad \|t'c_N'\|_{\sigma}(A_{\beta})\leq f_N. \numberthis\label{eq:log-mgf-final}
\end{align*}
With $\|c_N'\|_{\sigma}=\|c_N\|_{\sigma}/\|c_N\|_{\sigma}=1$, the summation term in \eqref{eq:log-mgf-final} can be re-expressed as follows,
\begin{align*}
\sum_{i=1}^N\sum_{j \in S_J(i)}c_{ij}'^2\sigma^2(m_{ij}^*+t'c_{ij}')
&=\sum_{i=1}^N\sum_{j \in S_J(i)}c_{ij}'^2\big\{\sigma^2(m_{ij}^*+t'c_{ij}')-\sigma_{ij}^2+\sigma_{ij}^2\big\}\\
&=\|c_N'\|_{\sigma}^2+\sum_{i=1}^N\sum_{j \in S_J(i)}c_{ij}'^2\big\{\sigma^2(m_{ij}^*+t'c_{ij}')-\sigma_{ij}^2\big\}\\
&=1+\sum_{i=1}^N\sum_{j \in S_J(i)}c_{ij}'^2\big\{\sigma^2(m_{ij}^*+t'c_{ij}')-\sigma_{ij}^2\big\}.
\end{align*}
Note that
\begin{align*}
 \sum_{i=1}^N\sum_{j \in S_J(i)}c_{ij}'^2\big\{\sigma^2(m_{ij}^*+t'c_{ij}')-\sigma_{ij}^2\big\}
& =\frac{1}{\sigma(g)}\sum_{i=1}^N\sum_{j \in S_J(i)}c_{ij}\big\{\sigma^2(m_{ij}^*+t'c_{ij}')-\sigma_{ij}^2\big\}c_{ij}'\\
&=\frac{1}{\sigma(g)} g\big\{U_N(M_N^*+t'c_N',c_N')\big\}\\
&\leq \frac{C\gamma_2^{-1}N_{*}^{-\frac{1}{2}}}{\sigma(g)}\|U_N(M_N^*+t'c_N',c_N')\|_{\sigma}(A_{ \beta})\\
&\leq \frac{C\gamma_2^{-1}N_{*}^{-\frac{1}{2}}}{\sigma(g)}d_N\|t'c_N'\|_{\sigma}(A_{ \beta})\|c_N'\|_{\sigma}(A_{\beta})\\
&\leq \frac{C\gamma_2^{-1}N_{*}^{-\frac{1}{2}}}{\sigma(g)}d_Nf_N\\
&\leq C\gamma_2^{-1}\gamma_3^{-1}d_Nf_N\\
&\to 0 \quad \text{as} \quad N \to \infty.
\end{align*}
The second line follows from $U_{ij}(M_N^*+t'c_N',c_N')=(\sigma_{ij}^2)^{-1}\big\{\sigma(m_{ij}^*+t'c_{ij}')-\sigma_{ij}^2\big\}c_{ij}'.$
The third last step follows from $\|c_N'\|_{\sigma}(A_{\beta})\leq \|c_N'\|_{\sigma}=1$ and the last step can be implied from Lemma \ref{lemma-sequence-and-bound-A-beta}(b).
Therefore, $\log G_{Z_N}(t) \to {t^2}/{2}  \text{ as } N\to \infty.$


Now consider case 2. We adopt a similar strategy to derive asymptotic normality as in case 1.  Define set $A_{\theta, \beta}$ to consist of linear functions $g_i, g_j' \in \Omega_{N}^*$ on $\Omega_N$ such that $g_i(x)=\theta_i$ and $g_j'(x)=\beta_j$, with $A_{\theta, \beta}=\{g_i, g_j':i=1,...,N, j=1,...,J\}.$ The explicit forms of $g_i$ and $g_j'$ can be found in the proof of Theorem \ref{thm-existence}. 

From now onwards, we take sequences $f_N$ and $d_N$ as satisfying the conditions in Lemma \ref{lemma-sequence-and-bound-A-theta-beta} below. 
Note it can be implied that with such $f_N$ and $d_N$, Lemmas \ref{lemma-bound-Zn} and \ref{lemma-bound-diff-btw-estimate-truth} still hold by taking $A=A_{\theta, \beta}$. 
From Lemmas \ref{lemma-itemvar} and \ref{lemma-peoplevar}, we know that for any $f \in A_{\theta, \beta}$, there exist $0<c_1, c_2<\infty$ and some $n$ such that for all $N>n,$
\begin{align*}
 c_1^{-1}N_{*}^{-\frac{1}{2}}<\sigma(f) < c_2^{-1}J_{*}^{-\frac{1}{2}}.\numberthis\label{eq: bound-var-single-people-or-item}   
\end{align*}

\begin{lemma}\label{lemma-sequence-and-bound-A-theta-beta}
Assume {Conditions \ref{cond:bound}-- \ref{cond:speed2} hold}. If $A_{\theta, \beta}=\{g_i, g_j':i=1,...,N, j=1,...,J\}$ such that $g_i, g_j' \in \Omega_{N}^*$, and $g_i(x)=\theta_i$ and $g_j'(x)=\beta_j$ for $x \in \Omega_N$. Let $C_N=|A_{\theta, \beta}|$, the cardinality of $A_{\theta, \beta}$.  Then there exist sequences $f_N>0$ and $d_N\geq 0$ satisfying the followings.

\noindent(a). As $N \to \infty$, $f_N^2/\log C_N \to \infty.$ 

\noindent(b). If $y, v \in \Omega_N$ and $\|y-M_N^*\|_{\sigma}(A_{\theta, \beta})\leq f_N$, then there exists $n<\infty$ such that for all $ N > n$, $\|U_N(y, v)\|_{\sigma}(A_{\theta, \beta}) \leq d_N \|y-M_N^*\|_{\sigma}(A_{\theta, \beta})\|v\|_{\sigma}(A_{\theta, \beta}).$ Furthermore, $d_Nf_N^2 \to 0$ as  $N \to \infty$.
\end{lemma}
Now for any $x \in \Omega_N,$
\begin{align*}
\vert g(x)\vert &= \vert w_g^T \theta + \Tilde w_g^T \beta \vert \\
&\leq \Big(\|w_g\|_1+\|\Tilde w_g\|_1\Big)\max_{i=1,...,N, j=1,...,J}\{\vert\theta_i\vert, \vert\beta_j\vert\}\\
&=\Big(\|w_g\|_1+\|\Tilde w_g\|_1\Big)\max_{f\in A_{\theta, \beta}}\{\vert f(x)\vert\}\\
&=\Big(\|w_g\|_1+\|\Tilde w_g\|_1\Big)\max_{f\in A_{\theta, \beta}}\Big\{\frac{\vert f(x)\vert }{\sigma(f)}\sigma(f)\Big\}\\
&\leq \Big(\|w_g\|_1+\|\Tilde w_g\|_1\Big)\Big\{\max_{f\in A_{\theta, \beta}}\frac{\vert f(x)\vert }{\sigma(f)}\Big\}\Big\{\max_{f\in A_{\theta, \beta}}\sigma(f)\Big\}\\
&< 2C c_2^{-1}J_{*}^{-\frac{1}{2}}\|x\|_{\sigma}(A_{\theta, \beta}),\numberthis\label{eq:tau(A, g)}
\end{align*}
where the last step follows from the definition of $\|\cdot\|_{\sigma}(A_{\theta, \beta})$, \eqref{eq: bound-var-single-people-or-item} and the assumption that $\|w_g\|_1, \|\Tilde w_g\|_1<C$. 
Further note that since $w_g \neq 0,$  as a direct consequence of Lemma \ref{sufficient-condition-var-approx} and {Condition \ref{cond:speed2}(c)}, there exists some $0<c_3<\infty$ such that for all $N>n,$
\begin{align*}
\sigma(g)\geq c_3 J_{*}^{-\frac{1}{2}}.\numberthis\label{eq:lower-bound-g}
\end{align*}
As a result of \eqref{eq:tau(A, g)}, 
\begin{align*}
\Big|\big[c_N, \hat{M}_N-M_N^*-R_N\big]_{\sigma}\Big| \leq 2C c_2^{-1}J_{*}^{-\frac{1}{2}}\|\hat{M}_N-M_N^*-R_N\|_{\sigma}(A_{\theta, \beta}).    
\end{align*}
Again, we have
\begin{align*}
\frac{g(\hat{M}_N)-g(M_N^*)}{\sigma(g)}&=\frac{\big[c_N, \hat{M}_N-M_N^*-R_N\big]_{\sigma}+\big[c_N, R_N\big]_{\sigma}}{\sigma(g)}
\end{align*}
As $N \to \infty$, re-arrange gives with probability tending to 1 that,
\begin{align*}
\Big\vert\frac{g(\hat{M}_N)-g(M_N^*)}{\sigma(g)}-\frac{\big[c_N, R_N\big]_{\sigma}}{\sigma(g)}\Big\vert 
&=\frac{\Big\vert[c_N, \hat{M}_N-M_N^*-R_N]_{\sigma}\Big\vert}{\sigma(g)}\\
&\leq \frac{2C c_2^{-1}J_{*}^{-\frac{1}{2}}}{\sigma(g)}\|\hat{M}_N-M_N^*-R_N\|_{\sigma}(A_{\theta, \beta})\\
&\leq \frac{2C c_2^{-1}J_{*}^{-\frac{1}{2}}}{\sigma(g)}d_N\big[\|R_N\|_{\sigma}(A_{\theta, \beta})\big]^2\\
&\leq \frac{1}{2}C c_2^{-1}c_3^{-1}d_Nf_N^2\\
&\to 0,
\numberthis\label{ineq: equivalent-normal}
\end{align*}
Again, we can denote $Z_N=\big[c_N, R_N\big]_{\sigma}/\sigma(g)$ for notation simplicity. Similar as in case 1, we just need to show $Z_N \to $ N$(0, 1)$. 
We consider the log moment generating function of $Z_N$, denoted as $\log G_{Z_N}(t)$. Write $c_{ij}':=c_{ij}/\sigma(g)$. Then similarly as in the proof for case 1, we obtain
\begin{align*}
\log G_{Z_N}(t)
&=\frac{t^2}{2}\sum_{i=1}^N\sum_{j \in S_J(i)}c_{ij}'^2\sigma^2(m_{ij}^*+t'c_{ij}'),\quad \quad \|t'c_N'\|_{\sigma}(A_{\theta, \beta})\leq f_N,
\end{align*}
where,
\begin{align*}
\sum_{i=1}^N\sum_{j \in S_J(i)}c_{ij}'^2\sigma^2(m_{ij}^*+t'c_{ij}')
&=1+\sum_{i=1}^N\sum_{j \in S_J(i)}c_{ij}'^2\big\{\sigma^2(m_{ij}^*+t'c_{ij}')-\sigma_{ij}^2\big\}.
\end{align*}
Note that
\begin{align*}
  \sum_{i=1}^N\sum_{j \in S_J(i)}c_{ij}'^2\big\{\sigma^2(m_{ij}^*+t'c_{ij}')-\sigma_{ij}^2\big\} 
&=\frac{1}{\sigma(g)}\sum_{i=1}^N\sum_{j \in S_J(i)}c_{ij}\big\{\sigma^2(m_{ij}^*+t'c_{ij}')-\sigma_{ij}^2\big\}c_{ij}'\\
&=\frac{1}{\sigma(g)} g\big\{U_N(M_N^*+t'c_N',c_N')\big\}\\
&\leq \frac{2C c_2^{-1}J_{*}^{-\frac{1}{2}}}{\sigma(g)}\|U_N(M_N^*+t'c_N',c_N')\|_{\sigma}(A_{\theta, \beta})\\
&\leq \frac{2C c_2^{-1}J_{*}^{-\frac{1}{2}}}{\sigma(g)}d_N\|t'c_N'\|_{\sigma}(A_{\theta, \beta})\|c_N'\|_{\sigma}(A_{\theta, \beta})\\
&\leq \frac{2C c_2^{-1}J_{*}^{-\frac{1}{2}}}{\sigma(g)}d_Nf_N\\
&\leq 2C c_2^{-1}c_3^{-1}d_Nf_N\\
&\to 0 \quad \text{as} \quad N \to \infty.
\end{align*}
The second line follows from $U_{ij}(M_N^*+t'c_N',c_N')=(\sigma_{ij}^2)^{-1}\big\{\sigma(m_{ij}^*+t'c_{ij}')-\sigma_{ij}^2\big\}c_{ij}'.$
The third last step follows from $\|c_N'\|_{\sigma}(A_{\theta,\beta})\leq \|c_N'\|_{\sigma}=1$ and the last step can be implied from Lemma \ref{lemma-sequence-and-bound-A-theta-beta}(b).
Therefore, $\log G_{Z_N}(t) \to \frac{t^2}{2}   \text{ as }  N\to \infty.$
Hence, the first part of the theorem follows.

Now we seek to prove the second part of the theorem.
The strategy is to show $\vert \hat{\sigma}^2(g)-\tilde{\sigma}^2(g)\vert /\tilde{\sigma}^2(g)\to 0$ in probability as $N\to \infty.$ Consider
\begin{align*}
\frac{\vert \hat{\sigma}^2(g)-\tilde{\sigma}^2(g)\vert}{\tilde{\sigma}^2(g)}
&= \frac{\Big\vert\sum_{i=1}^N w_{gi}^2\big\{(\hat{\sigma}_{i+}^2)^{-1}-(\sigma_{i+}^2)^{-1}\big\}+\sum_{j=1}^J \tilde{w}_{gj}^2\big\{(\hat{\sigma}_{+j}^2)^{-1}-(\sigma_{+j}^2)^{-1}\big\}\Big\vert}{\sum_{i=1}^N w_{gi}^2(\sigma_{i+}^2)^{-1}+\sum_{j=1}^J\tilde{w}_{gj}^2(\sigma_{+j}^2)^{-1}}\\
&= \frac{\Big\vert\sum_{i=1}^N w_{gi}^2\Big\{\frac{\sigma_{i+}^2-\hat{\sigma}_{i+}^2}{(\hat{\sigma}_{i+}^2)(\sigma_{i+}^2)}\Big\}+\sum_{j=1}^J \tilde{w}_{gj}^2\Big\{\frac{\sigma_{+j}^2-\hat{\sigma}_{+j}^2}{(\hat{\sigma}_{+j}^2)(\sigma_{+j}^2)}\Big\}\Big\vert}{\sum_{i=1}^N w_{gi}^2(\sigma_{i+}^2)^{-1}+\sum_{j=1}^J\tilde{w}_{gj}^2(\sigma_{+j}^2)^{-1}}\\
&\leq \frac{\Big\vert\sum_{i=1}^N w_{gi}^2\Big\{\frac{\sum_{j\in S_J(i)}\vert\sigma_{ij}^2-\hat{\sigma}_{ij}^2\vert}{(\hat{\sigma}_{i+}^2)(\sigma_{i+}^2)}\Big\}+\sum_{j=1}^J \tilde{w}_{gj}^2\Big\{\frac{\sum_{i \in S_N(j)}\vert\sigma_{ij}^2-\hat{\sigma}_{ij}^2\vert}{(\hat{\sigma}_{+j}^2)(\sigma_{+j}^2)}\Big\}\Big\vert}{\sum_{i=1}^N w_{gi}^2(\sigma_{i+}^2)^{-1}+\sum_{j=1}^J\tilde{w}_{gj}^2(\sigma_{+j}^2)^{-1}}.\numberthis\label{eq: bound-var-diff}
\end{align*}
Since $m_{ij}^*, \hat{m}_{ij} \in \mathbb{R},$ $0<\sigma_{ij}^2, \hat{\sigma}_{ij}^2 < 1.$ Note that there exist $0<c_4, c_5<\infty$ that
\begin{align*}
    \sigma_{i+}^2, \hat{\sigma}_{i+}^2 > c_4 J_{*},\quad \quad 
\sigma_{+j}^2, \hat{\sigma}_{+j}^2 > c_5 N_{*}.
\end{align*}
Further note that there exists a positive $c_6< \infty$ such that 
\begin{align*}
\max_{i,j,z_{ij}=1}\vert\sigma_{ij}^2-\hat{\sigma}_{ij}^2\vert
&\leq c_6\max_{i,j,z_{ij}=1}\vert m_{ij}^*-\hat{m}_{ij}\vert\\
&=o_p(1),\quad \quad \text{as $N \to \infty$.}
\end{align*}
where the last line follows from \eqref{eq:consistency-of-mij}.
It follows 
\begin{align*}
&\frac{\sum_{j\in S_J(i)}\vert\sigma_{ij}^2-\hat{\sigma}_{ij}^2\vert}{(\hat{\sigma}_{i+}^2)(\sigma_{i+}^2)}=o_p\big(J_{*}^{-1}\big),\numberthis\label{eq: bound-sigma_{i+}}\\
&\frac{\sum_{i \in S_N(j)}\vert\sigma_{ij}^2-\hat{\sigma}_{ij}^2\vert}{(\hat{\sigma}_{+j}^2)(\sigma_{+j}^2)} = o_p\big(N_{*}^{-1}\big).\numberthis\label{eq: bound-sigma_{+j}}
\end{align*}
Moreover, we note that $\|w_g\|_1, \|\tilde{w}_g\|_1 <C$ implies that $\sum_{i=1}^N w_{gi}^2<c_7$ and $\sum_{j=1}^J \tilde{w}_{gj}^2 <c_7$ for some $c_7 < \infty.$
From \eqref{eq: bound-var-diff}, it can be implied that 
\begin{align*}
\frac{\vert \hat{\sigma}^2(g)-\tilde{\sigma}^2(g)\vert}{\tilde{\sigma}^2(g)}=o_p(1),\quad \quad \text{as $N\to\infty,$}
\end{align*}
where the above result follows from \eqref{eq: bound-sigma_{i+}}, \eqref{eq: bound-sigma_{+j}} and the assumption that $g(x) \neq 0$ for any $x \in \Omega_N.$ Since we have shown $\tilde{\sigma}(g)^{-1}\{g(\hat{M})-g(M^*)\} \to$ N$(0, 1)$ in distribution in the first part of the proof, it follows that $\hat{\sigma}(g)^{-1}\{g(\hat{M})-g(M^*)\} \to$ N$(0, 1)$ in distribution as $N\to\infty.$
\end{proof}

Next, we give proof of Proposition \ref{prop:connect} below.

\begin{proof}[Proof of Proposition \ref{prop:connect}]

We prove the first part of the proposition by direct construction; in particular, we find the solutions for $\theta$ and $\beta$, respectively, given equations $\sum_{i=1}^N\theta_i=0$ and $\theta_i - \beta_j = m_{ij}$, $i = 1, ..., N, j = 1, ..., J$, for which $z_{ij} = 1$. We first construct the solution for $\beta_j$, $j=1,...,J$. The idea is to include all the row parameters $\theta_i$ so that we can apply the constraint $\sum_{i=1}^{N}\theta_i=0$. 

Denote $S_J(i)=\{j=1,...J :z_{ij}=1\}$, $S_N(j)=\{i=1,...,N:z_{ij}=1\}$, and $S_{N_{\phi}}(j)=\{1,2,...,N\}\setminus S_{N}(j).$
Then for any $i \in S_{N}(j),$ we use $m_{ij}=\theta_i-\beta_j$ in the construction. 
While for each $i \in S_{N_{\phi}}(j),$ applying Condition~\ref{cond:connect}, there must exist $1\leq i_{i1}, i_{i2},...,i_{ik}\leq N$ and $1\leq j_{i1}, j_{i2},...,j_{ik}\leq J$ such that 
\begin{align*}
    z_{i,j_{i1}}=z_{i_{i1},j_{i1}}=z_{i_{i1},j_{i2}}=z_{i_{i2},j_{i2}}=...=z_{i_{ik},j_{ik}}=z_{i_{ik},j}=1,
\end{align*}
with
\begin{align*}
    &m_{i,j_{i1}}-m_{i_{i1},j_{i1}}+m_{i_{i1},j_{i2}}-m_{i_{i2},j_{i2}}+...-m_{i_{ik},j_{ik}}+m_{i_{ik},j}\\
    =&(\theta_i-\beta_{j_{i1}})-(\theta_{i_{i1}}-\beta_{j_{i1}})+(\theta_{i_{i1}}-\beta_{j_{i2}})-(\theta_{i_{i2}}-\beta_{j_{i2}})+...-(\theta_{i_{ik}}-\beta_{j_{ik}})+(\theta_{i_{ik}}-\beta_j)\\
    =&\theta_i-\beta_j.
\end{align*}
Therefore, the solution for $\beta_j$ is simply
\begin{align*}
\beta_j=&-\frac{1}{N}\Big\{\sum_{i\in S_N(j)}m_{ij}\\
&+\sum_{i\in S_{N_{\phi}}(j)}\Big(m_{i,j_{i1}}-m_{i_{i1},j_{i1}}+m_{i_{i1},j_{i2}}-m_{i_{i2},j_{i2}}+...-m_{i_{ik},j_{ik}}+m_{i_{ik},j}\Big)\Big\}.
\end{align*}
To find solution for $\theta_i,$ 
\begin{align*}
\theta_i=&\frac{1}{\vert S_J(i)\vert }\sum_{j\in S_J(i)}\Big[m_{ij}-\frac{1}{N}\Big\{\sum_{i'\in S_N(j)}m_{i'j}\\
&+\sum_{i'\in S_{N_{\phi}}(j)}\Big(m_{i',j_{i'1}}-m_{i'_{i'1},j_{i'1}}+m_{i'_{i'1},j_{i'2}}-m_{i'_{i'2},j_{i'2}}+...-m_{i'_{i'k},j_{i'k}}+m_{i'_{i'k},j}\Big)\Big\}\Big],
\end{align*}
where $\vert S_J(i)\vert$ denotes the cardinality of $S_j(i).$
This concludes the proof for the first part of the proposition.

We can view the row parameters and column parameters as a bipartite graph $\mathcal{G}$, with one part consisting of row parameters as nodes (denoted as $\{i=1,...,N\}$ for simplicity) and the other consisting of column parameters as nodes (denoted as $\{j=1,...,J\}$ for simplicity). If $z_{ij}=1,$ then there is an edge connecting $i$ and $j$ in $\mathcal{G}.$
For the second part of the proposition, note if Condition~\ref{cond:connect} is not satisfied, then there exists at least one pair of $(i,j)$ such that there does not exist a path connecting them in graph $\mathcal{G}$. This means (claim): $\mathcal{G}$ can be separated into at least two sub-graphs.
Denote the two sub-graphs by $\mathcal{G}_1$ and $\mathcal{G}_2$ respectively. 
The above claim can be proved by a contradiction argument as follows. Suppose not, then there exist either $i_1' \in \mathcal{G}_1$ and $j_2' \in \mathcal{G}_2$ with $z_{i_1'j_2'}=1$, or $j_1' \in \mathcal{G}_1$ and $i_2' \in \mathcal{G}_2$ with $z_{i_2'j_1'}=1.$ By assumption there must exist a path connecting any two nodes within each of the two sub-graphs, otherwise we could split $\mathcal{G}$ into two sub-graphs. Therefore, there must exist a path connecting the pair $(i,j)$. A contradiction.

Now, denote $\{\theta_{i_1}, \beta_{j_1}: 1\leq i_1\leq N, 1\leq j_1\leq J\}$ and $\{\theta_{i_2}$, $\beta_{j_2}: 1\leq i_2\leq N, 1\leq j_2\leq J\}$  as the values associated with the nodes in $\mathcal{G}_1$ and in $\mathcal{G}_2$ respectively and together also serving as a solution set satisfying $\sum_{i=1}^N\theta_i=0$ and $\theta_i - \beta_j = m_{ij}$, $i = 1, ..., N, j = 1, ..., J, z_{ij} = 1$. 
Let $n_{i_1}$ and $n_{i_2}$ denote the number of row parameters in $\mathcal{G}_1$ and in $\mathcal{G}_2$ respectively. Let $\tau=n_{i_1}/n_{i_2}.$
For any constant $a$, let  $\Tilde{\theta}_{i_1}=\theta_{i_1}+a, \Tilde{\beta}_{j_1}=\beta_{j_1}+a$ and $\Tilde{\theta}_{i_2}=\theta_{i_2}-\tau a, \Tilde{\beta}_{j_2}=\beta_{j_2}- \tau a$. We can check easily that $(\Tilde{\theta}, \Tilde{\beta})$ is also a solution to the system but $(\Tilde{\theta}, \Tilde{\beta})\neq (\theta, \beta).$ 
{To show $m_{ij}$ is not identifiable for $z_{ij}=0$, we consider the same construction as above. Note that for any $\theta_{i_1} \in \mathcal{G}_1$ and $\beta_{j_2} \in \mathcal{G}_2$ so that $z_{i_1, j_2}=0,$ $\Tilde{\theta}_{i_1}-\Tilde{\beta}_{j_2}=\theta_{i_1}-\beta_{j_2}+(1+\tau)a \neq \theta_{i_1}-\beta_{j_2}$ unless $a=0.$ Therefore, $m_{ij}$ is not identifiable for $z_{ij}=0.$
This concludes the proof for the second part of the proposition.}
\end{proof}

\section*{Appendix B: Proofs of Supporting Lemmas}\label{appendix-proof-lemmas}
Appendix B includes the proofs of the supporting lemmas used in the proofs of the theorems and the proposition developed in the main article.

To begin with, we first give some intuition on how to obtain the approximation formula for $\sigma^2(g)$, as summarized in Lemmas \ref{lemma-sigma(g)}, \ref{lemma-d'n} and \ref{lemma-d''n} below.  {Specifically, Lemmas \ref{lemma-sigma(g)}, \ref{lemma-d'n} and \ref{lemma-d''n} hold under all conditions~\ref{cond:bound}--\ref{cond:speed2} and} will be used in the proofs of other supporting lemmas, which will be given later in Appendix B.

First note that it is a property of the exponential family that $\sigma(g)=\sup_{x\in\Omega_N}\{\vert g(x)\vert : \|x\|_{\sigma}^2\leq 1\}$ 
(see e.g. page 823 of \cite{haberman1977maximum}).
$\sigma^2(g)$ can be viewed as the solution to a constrained quadratic programming problem, i.e. 
\begin{align}
\max_{\theta, \beta} \Big\{\sum_{i=1}^N \sum_{j \in S_J(i)}w_{ij}(\theta_i-\beta_j)\Big\}^2 \quad \text{such that} \quad \sum_{i=1}^N \sum_{j \in S_J(i)}\sigma_{ij}^2(\theta_i-\beta_j)^2\leq 1,  \sum_{i=1}^N\theta_i=0. \label{eq:quandratic programming1}  
\end{align} 
An explicit form is often difficult to derive, so an approximation is desired for both implementation and inference purposes. We consider a three-way decomposition of the coefficients of $g$ that lies in the constrained solution space, and convert this quadratic programming to a linear system from which $\sigma^2(g)$ can be solved. The results are summarized in Lemma \ref{lemma-sigma(g)} below.

\begin{lemma}\label{lemma-sigma(g)}
Define a vector $d(g)=\big\{d_{ij}(g): i=1,...,N, j=1,...,J, z_{ij}=1, d_{ij}(g) \in \mathbb{R}\big\}$ with a three-way decomposition  $d_{ij}(g)=b(g)+f_{i}(g)+m_{j}(g),$ such that $[d(g), x]_{\sigma}=g(x)$ for $x\in \Omega_N$ and $ f_i(g), m_j(g)$ satisfying 
\begin{align*}
&\sum_{i=1}^{N} \sigma_{i+}^2 f_{i}(g) = 0, \numberthis\label{eq:fn}    
\\
&\sum_{j=1}^{J} \sigma_{+j}^2 m_{j}(g)= 0. \numberthis\label{eq:mn}  
\end{align*}
Then, we have 
\begin{align*}
 \sigma^2(g)
 &=b^2(g)\sigma_{++}^2+\sum_{i=1}^{N}\sigma_{i+}^2f_{i}^2(g)+\sum_{j=1}^{J}\sigma_{+j}^2m_{j}^2(g)+2\sum_{i=1}^{N}\sum_{j \in S_J(i)} \sigma_{ij}^2f_{i}(g)m_{j}(g).
\numberthis\label{eq: sigma-expression}   
\end{align*}
\end{lemma}
\begin{proof}
Note $\sigma^2(g)$ is a solution to the quadratically constrained quadratic programming problem \eqref{eq:quandratic programming1}. From {\citet[pages 835-837]{haberman1977maximum}},
the construction of $d(g)$ in the lemma lies in the required solution space of \eqref{eq:quandratic programming1}. As a result, $\sigma^2(g)$ can be expressed directly as $\sigma^2(g)=\|d(g)\|_{\sigma}^2$. We just need to find an explicit expression of $\|d(g)\|_{\sigma}^2$ in terms of $b(g), f_i(g), m_j(g).$ 

First consider $x\in \Omega_N$ such that $x_{ij}=y$ are identical for all $i=1,...,N, j=1,...,J, z_{ij}=1$. Then in such cases,
\begin{align*}
g(x)
&=[d(g), x]_{\sigma}\\
&=\sum_{i=1}^{N}\sum_{j \in S_J(i)} \big\{b(g)+f_{i}(g)+m_{j}(g)\big\}\sigma_{ij}^2y\\
&=b(g)\sigma_{++}^2y + \sum_{i=1}^{N}\Big(\sum_{j \in S_J(i)}\sigma_{ij}^2\Big)f_{i}(g)y+ \sum_{j=1}^{J}\Big(\sum_{i \in S_N(j)}\sigma_{ij}^2\Big)m_{j}(g)y\\
&=b(g)\sigma_{++}^2y+\sum_{i=1}^{N}\sigma_{i+}^2f_{i}(g)y+\sum_{j=1}^{J}\sigma_{+j}^2m_{j}(g)y\\
&=b(g)\sigma_{++}^2y,
\numberthis\label{eq:bn1}
\end{align*}
where the last step follows from \eqref{eq:fn} and \eqref{eq:mn}. 
Also by the original definition of $g$, we have 
\begin{align*}
g(x)
=\sum_{i=1}^{N}\sum_{j \in S_J(i)} w_{ij}y=w_{++}y.
\numberthis\label{eq:bn2}
\end{align*}
Since \eqref{eq:bn1} and \eqref{eq:bn2} hold for any $y$, we must have 
\begin{align*}
b(g)
=(\sigma_{++}^2)^{-1}w_{++}.
\numberthis\label{eq:bn-expression}
\end{align*}
Next consider $x\in \Omega_N$ such that $x_{ij}=y_i$, $y_i \in \mathbb{R}$, for any $i=1,...,N, j=1,...,J, z_{ij}=1$, then 
\begin{align*}
g(x)
&=[d(g), x]_{\sigma} 
 =\sum_{i=1}^N\sum_{j\in S_J(i)}d_{ij}(g)\sigma_{ij}^2y_i =\sum_{i=1}^Ny_i\Big(\sum_{j\in S_J(i)}d_{ij}(g)\sigma_{ij}^2\Big).
\numberthis\label{eq:fn1}
\end{align*}
From the original definition of $g$, 
\begin{align*}
g(x)
=\sum_{i=1}^N\sum_{j\in S_J(i)} w_{ij}y_{i}=\sum_{i=1}^Ny_i\Big(\sum_{j\in S_J(i)} w_{ij}\Big).
\numberthis\label{eq:fn2}
\end{align*}
Since $\eqref{eq:fn1}=\eqref{eq:fn2}$ for any $y_i$, it follows that
\begin{align*}
\sum_{j\in S_J(i)}d_{ij}(g)\sigma_{ij}^2
=\sum_{j\in S_J(i)} w_{ij}=w_{i+}, \quad i=1,...,N.
\numberthis\label{eq:dn-sigma-wi+}
\end{align*}
Consider 
\begin{align*}
f_{i}(g)+m_{j}(g)
&=d_{ij}(g)-b(g)\\
\sum_{j \in S_J(i)}\big\{f_{i}(g)+m_{j}(g)\big\}\sigma_{ij}^2&=\sum_{j \in S_J(i)}\big\{d_{ij}(g)-b(g)\big\}\sigma_{ij}^2\\
\sigma_{i+}^2f_{i}(g)+\sum_{j \in S_J(i)}\sigma_{ij}^2m_{j}(g)&= \sum_{j \in S_J(i)} d_{ij}(g) \sigma_{ij}^2 - \sigma_{i+}^2b(g)\\
\sigma_{i+}^2f_{i}(g)+\sum_{j \in S_J(i)}\sigma_{ij}^2m_{j}(g)&= w_{i+} - (\sigma_{++}^2)^{-1}w_{++}\sigma_{i+}^2, \quad i=1,...,N,
\numberthis\label{eq:f-relation}
\end{align*}
where the last line follows from \eqref{eq:bn-expression} and \eqref{eq:dn-sigma-wi+}. Similarly, we consider $x\in \Omega_N$ such that $x_{ij}=y_j$, $y_j \in \mathbb{R}$ for any $i=1,...,N, j=1,...,J, z_{ij}=1$, then
\begin{align*}
g(x)&=[d(g), x]_{\sigma} =\sum_{j=1}^J\sum_{i\in S_N(j)}d_{ij}(g)\sigma_{ij}^2y_j\\
&=\sum_{j=1}^Jy_j\Big(\sum_{i\in S_N(j)}d_{ij}(g)\sigma_{ij}^2\Big). 
\numberthis\label{eq:mn1}
\end{align*}
Again by the original definition of $g$,
\begin{align*}
g(x)
=\sum_{j=1}^J\sum_{i\in S_N(j)} w_{ij}y_j=\sum_{j=1}^Jy_j\Big(\sum_{i\in S_N(j)} w_{ij}\Big). 
 \numberthis\label{eq:mn2}
\end{align*}
Since $\eqref{eq:mn1}=\eqref{eq:mn2}$ for any $y_j \in \mathbb{R}$, it follows
\begin{align*}
\sum_{i\in S_N(j)}d_{ij}(g)\sigma_{ij}^2
=\sum_{i\in S_N(j)} w_{ij}=w_{+j}.
\numberthis\label{eq:d-sigma-w+j}
\end{align*}
Similarly,
\begin{align*}
f_{i}(g)+m_{j}(g)
&=d_{ij}(g)-b(g)\\
\sum_{i \in S_N(j)}\big\{f_{i}(g)+m_{j}(g)\big\}\sigma_{ij}^2&=\sum_{i \in S_N(j)}\big\{d_{ij}(g)-b(g)\big\}\sigma_{ij}^2\\
\sigma_{+j}^2m_{j}(g)+\sum_{i \in S_N(j)}\sigma_{ij}^2f_{i}(g)&= \sum_{i \in S_N(j)} d_{ij}(g) \sigma_{ij}^2 - \sigma_{+j}^2b(g)\\
\sigma_{+j}^2m_{j}(g)+\sum_{i \in S_N(j)}\sigma_{ij}^2f_{i}(g)&= w_{+j} - (\sigma_{++}^2)^{-1}w_{++}\sigma_{+j}^2, \quad j=1,...,J,
\numberthis\label{eq:m-relation}
\end{align*}
where the last line follows from \eqref{eq:bn-expression} and \eqref{eq:d-sigma-w+j}. 
Note that all $b(g), f_i(g), m_j(g)$ can be obtained by solving a system of $N+J+1$ linear equations from \eqref{eq:bn-expression}, \eqref{eq:f-relation} and \eqref{eq:m-relation}.
Now we seek to derive a simplified expression for $\|d(g)\|_{\sigma}^2$ in terms of $b(g), f_i(g), m_j(g)$. 
\begin{align*}
\sigma^2(g)
=& \|d(g)\|_{\sigma}^2  \\
=& \sum_{i=1}^{N}\sum_{j \in S_J(i)} \sigma_{ij}^2\big\{b(g)+f_{i}(g)+m_{j}(g)\big\}^2\\
=&b(g)\sum_{i=1}^{N}\sum_{j \in S_J(i)}\sigma_{ij}^2\big\{b(g)+f_{i}(g)+m_{j}(g)\big\} \numberthis\label{eq:dn-e1}\\
&+\sum_{i=1}^{N}\sum_{j \in S_J(i)}f_{i}(g)\sigma_{ij}^2\big\{b(g)+f_{i}(g)+m_{j}(g)\big\}\numberthis\label{eq:dn-e2}\\
&+\sum_{i=1}^{N}\sum_{j \in S_J(i)}m_{j}(g)\sigma_{ij}^2\big\{b(g)+f_{i}(g)+m_{j}(g)\big\}.\numberthis\label{eq:dn-e3}
\end{align*}
Let us consider each of these three terms separately,
\begin{align*}
(\ref{eq:dn-e1})&=b^2(g)\sigma_{++}^2+b(g)\sum_{i=1}^{N}f_i(g)\Big(\sum_{j \in S_J(i)}\sigma_{ij}^2\Big)+b(g)\sum_{j=1}^{J}m_{j}(g)\Big(\sum_{i \in S_N(j)}\sigma_{ij}^2\Big)\\
&=b^2(g)\sigma_{++}^2+b(g)\sum_{i=1}^{N}\sigma_{i+}^2f_{i}(g)+b(g)\sum_{j=1}^{J}\sigma_{+j}^2m_{j}(g)\\
&=b^2(g)\sigma_{++}^2.
\end{align*}
\begin{align*}
(\ref{eq:dn-e2})
&=b(g)\sum_{i=1}^{N}f_{i}(g)\Big(\sum_{j \in S_J(i)}\sigma_{ij}^2\Big)+\sum_{i=1}^Nf_{i}^2(g)\Big(\sum_{j \in S_J(i)}\sigma_{ij}^2\Big)+\sum_{i=1}^{N}\sum_{j \in S_J(i)} \sigma_{ij}^2f_{i}(g)m_{j}(g)\\
&=b(g)\sum_{i=1}^{N}f_{i}(g)\sigma_{i+}^2+\sum_{i=1}^{N}\sigma_{i+}^2f_{i}^2(g)+\sum_{i=1}^{N}\sum_{j \in S_J(i)} \sigma_{ij}^2f_{i}(g)m_{j}(g)\\
&=\sum_{i=1}^{N}\sigma_{i+}^2f_{i}^2(g)+\sum_{i=1}^{N}\sum_{j \in S_J(i)} \sigma_{ij}^2f_{i}(g)m_{j}(g).
\end{align*}
\begin{align*}
(\ref{eq:dn-e3})&=b(g)\sum_{j=1}^{J}m_{j}(g)\Big(\sum_{i \in S_N(j)}\sigma_{ij}^2\Big)+\sum_{j=1}^Jm_{j}^2(g)\Big(\sum_{i \in S_N(j)}\sigma_{ij}^2\Big)+\sum_{i=1}^{N}\sum_{j \in S_J(i)} \sigma_{ij}^2f_{i}(g)m_{j}(g)\\
&=b(g)\sum_{j=1}^{J}\sigma_{+j}^2m_{j}(g)+\sum_{j=1}^{J}\sigma_{+j}^2m_{j}^2(g)+\sum_{i=1}^{N}\sum_{j \in S_J(i)} \sigma_{ij}^2f_{i}(g)m_{j}(g)\\
&=\sum_{j=1}^{J}\sigma_{+j}^2m_{j}^2(g)+\sum_{i=1}^{N}\sum_{j \in S_J(i)} \sigma_{ij}^2f_{i}(g)m_{j}(g).
\end{align*}

Combining three terms together,  the result of the lemma follows with
\begin{align*}
\sigma^2(g)
&=\|d(g)\|_{\sigma}^2=b^2(g)\sigma_{++}^2+\sum_{i=1}^{N}\sigma_{i+}^2f_{i}^2(g)+\sum_{j=1}^{J}\sigma_{+j}^2m_{j}^2(g)+2\sum_{i=1}^{N}\sum_{j \in S_J(i)} \sigma_{ij}^2f_{i}(g)m_{j}(g).
\end{align*}
\end{proof}

As in the proof of Lemma \ref{lemma-sigma(g)}, we can solve a system of $N+J+1$ linear equations from \eqref{eq:bn-expression}, \eqref{eq:f-relation} and \eqref{eq:m-relation} for $f_i(g), i=1,...,N$, $m_j(g), j=1,...,J$ and $b(g)$. Then an exact expression for $\sigma^2(g)$ can be obtained by substituting these values into \eqref{eq: sigma-expression}. However, when $N$ and $J$ are large, it is difficult to solve this large system of linear equations. Furthermore, to study the order of $\sigma^2(g)$, we need an analytical form for analysis. The following set-ups are used to find an approximation for $\sigma^2(g)$. Define $\gamma_N >0$ to be the largest number such that for all $i=1,...,N, j=1,...,J, z_{ij}=1$, \begin{align*}
x^2\sigma_{ij}^2 \geq \gamma_N\Big(\frac{1}{|S_J(i)|}x^2\sigma_{i+}^2+\frac{1}{|S_N(j)|}x^2\sigma_{+j}^2\Big), 
\quad x \in \mathbb{R},\numberthis\label{eq:gamma_N-def}    
\end{align*}
where $|S_J(i)|$ and $|S_N(j)|$ are the cardinalities of $S_J(i)$ and $S_N(j)$ respectively. Note that there exist some $\gamma>0$ such that $\gamma_N>\gamma$ for all $N.$
For $i=1,...,N$ and $j=1,...,J$, further define 
\begin{align*}
 &f'_{i}(g)=(\sigma_{i+}^2)^{-1}w_{i+}-(\sigma_{++}^2)^{-1}w_{++}, \numberthis\label{eq:f_i'-def}\\  
 &m'_{j}(g)=(\sigma_{+j}^2)^{-1}w_{+j}-(\sigma_{++}^2)^{-1}w_{++},\numberthis\label{eq:m_j'-def}\\
 &f''_{i}(g)=f_{i}(g)-f'_{i}(g), \numberthis\label{eq:f_i''-def}\\ &m''_{j}(g)=m_{j}(g)-m'_{j}(g),\numberthis\label{eq:m_j''-def}.
\end{align*}
Then for $i=1,...,N, j=1,...,J$ with $z_{ij}=1$, define
\begin{align*}
 &\Check{\sigma}_{ij}^2=\sigma_{ij}^2-\gamma_N\Big(\frac{1}{|S_J(i)|}\sigma_{i+}^2+\frac{1}{\vert S_N(j)\vert}\sigma_{+j}^2\Big), \numberthis\label{eq:check-sigma-def}\\
 &d'_{ij}(g)=b(g)+f'_{i}(g)+m'_{j}(g), \numberthis\label{eq:dn'-def}\\
 &d''_{ij}(g)=f''_{i}(g)+m''_{j}(g)=d_{ij}(g)- d'_{ij}(g).\numberthis\label{eq:dn''-def}
\end{align*} 
By triangle inequality, \eqref{eq:dn''-def} then implies $$\|d'(g)\|_{\sigma}-\|d''(g)\|_{\sigma} \leq \|d(g)\|_{\sigma} \leq \|d'(g)\|_{\sigma}+\|d''(g)\|_{\sigma}.$$
We seek to use $\|d'(g)\|_{\sigma}$ as an approximation to $\sigma(g)=\|d(g)\|_{\sigma}$ while showing $\|d''(g)\|_{\sigma}$ is a negligible term  asymptotically under certain conditions. The analytical expression for $\|d'(g)\|_{\sigma}$ is given in Lemma \ref{lemma-d'n} below.

\begin{lemma}\label{lemma-d'n}
If $d'(g)$ is defined as in \eqref{eq:dn'-def}, then
\begin{align*}
    \|d'(g)\|_{\sigma}^2 =&\sum_{i=1}^{N}w_{i+}^2(\sigma_{i+}^2)^{-1}+\sum_{j=1}^{J}w_{+j}^2(\sigma_{+j}^2)^{-1} \\
    &+ 2\sum_{i=1}^{N}\sum_{j \in S_J(i)}w_{i+}w_{+j}\sigma_{ij}^2(\sigma_{i+}^2)^{-1}(\sigma_{+j}^2)^{-1}
    -3w_{++}^2(\sigma_{++}^2)^{-1}.
\end{align*}
\end{lemma}
\begin{proof}
Following from the definition of $d'(g)$, we can write 

\begin{align*}
\|d'(g)\|_{\sigma}^2 
=& \sum_{i=1}^{N}\sum_{j \in S_J(i)} \sigma_{ij}^2\big\{b(g)+(\sigma_{i+}^2)^{-1}w_{i+}+(\sigma_{+j}^2)^{-1}w_{+j}-2(\sigma_{++}^2)^{-1}w_{++}\big\}^2\\
=&~b(g)\sum_{i=1}^{N}\sum_{j \in S_J(i)}\sigma_{ij}^2\big\{b(g)+(\sigma_{i+}^2)^{-1}w_{i+}+(\sigma_{+j}^2)^{-1}w_{+j}
-2(\sigma_{++}^2)^{-1}w_{++}\big\} \numberthis\label{eq:dn'1}\\
&+\sum_{i=1}^{N}\sum_{j \in S_J(i)}\sigma_{ij}^2(\sigma_{i+}^2)^{-1}w_{i+}\big\{b(g)+(\sigma_{i+}^2)^{-1}w_{i+}
+(\sigma_{+j}^2)^{-1}w_{+j}-2(\sigma_{++}^2)^{-1}w_{++}\big\}\numberthis\label{eq:dn'2}\\
&+\sum_{i=1}^{N}\sum_{j \in S_J(i)}\sigma_{ij}^2(\sigma_{+j}^2)^{-1}w_{+j}\big\{b(g)+(\sigma_{i+}^2)^{-1}w_{i+}
+(\sigma_{+j}^2)^{-1}w_{+j}-2(\sigma_{++}^2)^{-1}w_{++}\big\}\numberthis\label{eq:dn'3}\\
&-2\sum_{i=1}^{N}\sum_{j \in S_J(i)}\sigma_{ij}^2(\sigma_{++}^2)^{-1}w_{++}\big\{b(g)+(\sigma_{i+}^2)^{-1}w_{i+}
+(\sigma_{+j}^2)^{-1}w_{+j}-2(\sigma_{++}^2)^{-1}w_{++}\big\}.\numberthis\label{eq:dn'4}
\end{align*}
We evaluate each of these four terms separately. For the first term,
\begin{align*}
\eqref{eq:dn'1}
=&b^2(g)\sigma_{++}^2+b(g)\sum_{i=1}^{N}\Big(\sum_{j \in S_J(i)}\sigma_{ij}^2\Big)(\sigma_{i+}^2)^{-1}w_{i+}-b(g)w_{++}+b(g) \sum_{j=1}^{J}w_{+j}-b(g)w_{++}\\
=&b^2(g)\sigma_{++}^2\\
=&(\sigma_{++}^2)^{-1}w_{++}^2,
\end{align*}
where the last line follows from \eqref{eq:bn-expression}. Now consider the second term,
\begin{align*}
(\ref{eq:dn'2})
=&b(g)w_{++}+\sum_{i=1}^Nw_{i+}^2(\sigma_{i+}^2)^{-1} + \sum_{i=1}^{N}\sum_{j \in S_J(i)}\sigma_{ij}^2(\sigma_{i+}^2)^{-1}w_{i+}(\sigma_{+j}^2)^{-1}w_{+j} - 2(\sigma_{++}^2)^{-1}w_{++}^2\\
=&-(\sigma_{++}^2)^{-1}w_{++}^2 + \sum_{i=1}^{N}w_{i+}^2(\sigma_{i+}^2)^{-1}+ \sum_{i=1}^{N}\sum_{j \in S_J(i)}w_{i+}w_{+j}\sigma_{ij}^2(\sigma_{i+}^2)^{-1}(\sigma_{+j}^2)^{-1}.
\end{align*}
Now consider the third term,
\begin{align*}
(\ref{eq:dn'3})=&b(g)w_{++}+\sum_{i=1}^{N}\sum_{j \in S_J(i)}\sigma_{ij}^2(\sigma_{+j}^2)^{-1}w_{+j}(\sigma_{i+}^2)^{-1}w_{i+}+\sum_{j=1}^{J}w_{+j}^2(\sigma_{+j}^2)^{-1}- 2 (\sigma_{++}^2)^{-1}w_{++}^2\\
=&-(\sigma_{++}^2)^{-1}w_{++}^2 +\sum_{j=1}^{J}w_{+j}^2(\sigma_{+j}^2)^{-1} +\sum_{i=1}^{N}\sum_{j \in S_J(i)}w_{i+}w_{+j}\sigma_{ij}^2(\sigma_{i+}^2)^{-1}(\sigma_{+j}^2)^{-1}.
\end{align*}

Now consider the last term,
\begin{align*}
(\ref{eq:dn'4})
&=-2b(g)w_{++}-2(\sigma_{++}^2)^{-1}w_{++}^2
-2(\sigma_{++}^2)^{-1}w_{++}^2+4b(g)w_{++}\\
&=-2(\sigma_{++}^2)^{-1}w_{++}^2-2(\sigma_{++}^2)^{-1}w_{++}^2-2(\sigma_{++}^2)^{-1}w_{++}^2+4(\sigma_{++}^2)^{-1}w_{++}^2\\
&=-2w_{++}^2(\sigma_{++}^2)^{-1}.
\end{align*}
Combining all these four terms together, we obtain
\begin{align*}
    \|d'(g)\|_{\sigma}^2 =&\sum_{i=1}^{N}w_{i+}^2(\sigma_{i+}^2)^{-1}+\sum_{j=1}^{J}w_{+j}^2(\sigma_{+j}^2)^{-1} \\
    &+ 2\sum_{i=1}^{N}\sum_{j \in S_J(i)}w_{i+}w_{+j}\sigma_{ij}^2(\sigma_{i+}^2)^{-1}(\sigma_{+j}^2)^{-1}
    -3w_{++}^2(\sigma_{++}^2)^{-1}.
\end{align*}
Hence the result of the lemma follows.
\end{proof}

Lemma \ref{lemma-d''n} below gives an analytical upper bound for  $\|d''(g)\|_{\sigma}$ so that we can show it is a negligible term under certain conditions. Define
\begin{align*}
 &l_i=-\sum_{j \in S_J(i)}w_{+j}\sigma_{ij}^2(\sigma_{+j}^2)^{-1}+w_{++}\sigma_{i+}^2(\sigma_{++}^2)^{-1}, \quad i=1,...,N,\numberthis\label{eq:l_i-def}\\
 &v_j=-\sum_{i \in S_N(j)}w_{i+}\sigma_{ij}^2(\sigma_{i+}^2)^{-1}+w_{++}\sigma_{+j}^2(\sigma_{++}^2)^{-1}, \quad j=1,...,J.\numberthis\label{eq:v_j-def}
\end{align*}

\begin{lemma}\label{lemma-d''n}
If $l_i$ and $v_j$ are defined as in 
\eqref{eq:l_i-def} and \eqref{eq:v_j-def}, respectively, then 
\begin{align*}
    \|d''(g)\|_{\sigma} \leq &\gamma_N^{-1}\Big[\sum_{i =1}^{N}l_i^2(\sigma_{i+}^2)^{-1}+\sum_{j=1}^{J}v_j^2(\sigma_{+j}^2)^{-1}\Big].
\end{align*}
\end{lemma}
\begin{proof}
From the definitions of  $f_i'', m_j'', l_i$ and $v_j$ as in \eqref{eq:f_i''-def}, \eqref{eq:m_j''-def}, \eqref{eq:l_i-def} and \eqref{eq:v_j-def}, respectively, it can be easily verified that
\begin{align*}
  &\sigma_{i+}^2 f_i''+\sum_{j \in S_J(i)}\sigma_{ij}^2m_j''=l_i, \quad i=1,...,N,\\
  &\sigma_{+j}^2 m_j''+\sum_{i \in S_N(j)}\sigma_{ij}^2f_i''=v_j,\quad j=1,...,J.
\end{align*}
It can be shown $\|d''(g)\|_{\sigma}^2=\sum_{i=1}^{N}f_i''l_i+\sum_{j=1}^{J}m_j''v_j$, which can be seen as follows,
\begin{align*}
\sum_{i=1}^{N}f_i''l_i+\sum_{j=1}^{J}m_j''v_j
&=\sum_{i=1}^Nf_i''\big(\sigma_{i+}^2f_i''+\sum_{j \in S_J(i)}\sigma_{ij}^2m_j''\big)+\sum_{j=1}^{J}m_j''\big(\sigma_{+j}^2m_j''+\sum_{i \in S_N(j)}\sigma_{ij}^2f_i''\big)\\
&=\sum_{i=1}^{N}\sigma_{i+}^2f_i''^2+\sum_{j=1}^{J}\sigma_{+j}^2m_j''^2+2\sum_{i=1}^N\sum_{j \in S_J(i)}f_i''m_j''\sigma_{ij}^2\\
&= \sum_{i=1}^{N}\sum_{j \in S_J(i)}(f_i''+m_j'')^2\sigma_{ij}^2\\
&=\|d''(g)\|_{\sigma}^2.
\end{align*}
Furthermore, by {\citet[page 60]{rao1973linear}}, $\sum_{i=1}^{N}f_i''l_i+\sum_{j=1}^{J}m_j''v_j$ is the largest value of $\big(\sum_{i=1}^N x_il_i +\sum_{j=1}^{J}y_jv_j\big)^2,$
for $x_i \in \mathbb{R}$, $i=1,...,N$ and $y_j \in \mathbb{R}$, $j=1,...,J$ such that
\begin{align*}
&\sum_{i \in S_N(j)}\frac{1}{|S_J(i)|}\sigma_{i+}^2x_i=0,\quad j=1,...,J,\\
&\sum_{j\in S_J(i)}\frac{1}{|S_N(j)|}\sigma_{+j}^2y_j=0,\quad i=1,...,N,\\
&D(x, y)=\sum_{i=1}^N\sum_{j \in S_J(i)}\sigma_{ij}^2(x_i+y_j)^2\leq 1.
\end{align*}
Note
\begin{align*}
&\sum_{i=1}^N\sum_{j \in S_{J}(i)}(x_i+y_j)^2\Check{\sigma}_{ij}^2\\
=&\sum_{i=1}^N\sum_{j \in S_{J}(i)}(x_i+y_j)^2\Big\{\sigma_{ij}^2-\gamma_N\big(\frac{1}{|S_J(i)|}\sigma_{i+}^2+\frac{1}{|S_N(j)|}\sigma_{+j}^2\big)\Big\}\\
=&D(x, y)-\gamma_N\sum_{i=1}^N\sum_{j \in S_{J}(i)}(x_i+y_j)^2\Big\{\frac{1}{|S_J(i)|}\sigma_{i+}^2+\frac{1}{|S_N(j)|}\sigma_{+j}^2\Big\}\\
=&D(x, y)-\gamma_N\Big\{\sum_{i=1}^N\big(x_i^2\sigma_{i+}^2+\sum_{j \in S_j(i)}\frac{1}{|S_N(j)|}x_i^2\sigma_{+j}^2\big)+\sum_{j=1}^J\big(y_j^2\sigma_{+j}^2+\sum_{i \in S_N(j)}\frac{1}{|S_J(i)|}y_j^2\sigma_{i+}^2\big)\Big\}\\
&-2\gamma_N\sum_{j=1}^Jy_j\Big\{\sum_{i \in S_N(j)}\frac{1}{|S_J(i)|}\sigma_{i+}^2x_i\Big\}-2\gamma_N\sum_{i=1}^Nx_i\Big\{\sum_{j \in S_J(i)}\frac{1}{|S_N(j)|}\sigma_{+j}^2y_j\Big\}\\
=&D(x, y)-\gamma_N\Big\{\sum_{i=1}^N\big(x_i^2\sigma_{i+}^2+\sum_{j \in S_j(i)}\frac{1}{|S_N(j)|}x_i^2\sigma_{+j}^2\big)+\sum_{j=1}^J\big(y_j^2\sigma_{+j}^2+\sum_{i \in S_N(j)}\frac{1}{|S_J(i)|}y_j^2\sigma_{i+}^2\big)\Big\}.
\end{align*}
Re-arranging the above expression gives,
\begin{align*}
D(x, y)
=&\gamma_N\Big\{\sum_{i=1}^N\big(x_i^2\sigma_{i+}^2+\sum_{j \in S_j(i)}\frac{1}{|S_N(j)|}x_i^2\sigma_{+j}^2\big)+\sum_{j=1}^J\big(y_j^2\sigma_{+j}^2+\sum_{i \in S_N(j)}\frac{1}{|S_J(i)|}y_j^2\sigma_{i+}^2\big)\Big\} \\
&+ \sum_{i=1}^N\sum_{j \in S_{J}(i)}(x_i+y_j)^2\Check{\sigma}_{ij}^2\\
&\geq \gamma_N\Big\{\sum_{i=1}^N x_i^2\sigma_{i+}^2+\sum_{j=1}^Jy_j^2\sigma_{+j}^2\Big\}.
\end{align*}
It follows that 
$$
    \|d''(g)\|_{\sigma} \leq \gamma_N^{-1}\Big[\sum_{i =1}^{N}l_i^2(\sigma_{i+}^2)^{-1}+\sum_{j=1}^{J}v_j^2(\sigma_{+j}^2)^{-1}\Big].
$$
\end{proof}

Next, we give proofs for the supporting lemmas used in the proofs of Proposition~\ref{prop:connect} and the proofs of Theorems~\ref{thm-existence} and \ref{thm-3-sufficient-condition}.\\\\

{\sc Lemma \ref{lemma-sequence-and-bound}.
} {\it Assume Conditions \ref{cond:bound}--\ref{cond:speed} hold. If $A_p=\{f_{ij}: i=1,...,N, j=1,...,J, z_{ij}=1\}$ such that $f_{ij}(x)=x_{ij}$ for $x \in \Omega_N$. Let $C_N=|A_p|$, the cardinality of $A_p$.  There exist sequences $f_N > 0$ and $d_N\geq 0$ satisfying the followings.

\noindent(a). As $N \to \infty$, $f_N^2/\log C_N \to \infty.$ 

\noindent(b). As $N \to \infty$, $f_N^2(N_{*}^{-1}+J_{*}^{-1}) \to 0.$

\noindent(c). If $y, v \in \Omega_N$ and $\|y-M_N^*\|_{\sigma}(A_p)\leq f_N$, then there exists $n<\infty$ such that for all $ N > n$, $\|U_N(y, v)\|_{\sigma}(A_p) \leq d_N \|y-M_N^*\|_{\sigma}(A_p)\|v\|_{\sigma}(A_p).$ Furthermore, $d_Nf_N \to 0$ as  $N \to \infty$.}

\begin{proof}
 Condition \ref{cond:speed}(a) assumes $J_{*}^{-1}\log N \to 0$, which implies that $\log N^{*} \ll J_{*}$.
 Then there must exist a sequence $f_N > 0$ such that $\log N^{*} \ll f_N^2 \ll J_{*}$.
\begin{align*}
f_N^2/\log C_N &\geq 
\frac{f_N^2}{\log(J^{*}N^{*})}\\
&= \frac{f_N^2}{\log J^{*} + \log N^{*}}\\
&\geq \frac{f_N^2}{2\log N^{*}} 
 \to \infty \quad \text{as}\quad N \to \infty.
\end{align*}
The first inequality follows from the fact that $J_{*}N_{*}\leq C_N\leq J^{*}N^{*}$. 
The last line follows from $\log N^{*} \ll  f_N^2$. Therefore, the result of part (a) is satisfied.
We further note
\begin{align*}
   f_N^2(N_{*}^{-1}+J_{*}^{-1})&\leq \frac{2f_N^2}{J_{*}} \to 0 \quad \text{as}\quad N \to \infty.\numberthis\label{f_N^2-J-inverse}
\end{align*}
The last line follows from $f_N^2\ll J_{*}$. Therefore, part (b) of the lemma follows. 
To verify part (c), first note by Lemma \ref{lemma-peopleitemdiffvar}, for any point maps $f_{ij} \in A_p$, there exist $0<\tau_1, \tau_2<\infty$ such that for all $N > n$,
\begin{align*}
  \tau_1^{-1}{\left(N^{*-1}+J^{*-1}\right)^{\frac{1}{2}}} <\sigma(f_{ij})< \tau_2^{-1}\left(N_{*}^{-1}+J_{*}^{-1}\right)^{\frac{1}{2}}.\numberthis\label{ineq: var-f}  
\end{align*}
By the definition of $\|\cdot\|_{\sigma}(A_p)$, we have for any $y \in \Omega_N, f_{ij} \in A_p,$
\begin{align*}
&|f_{ij}(y)| \leq \|y\|_{\sigma}(A_p)\sigma(f_{ij}).\numberthis\label{ineq: A}  
\end{align*}
It follows from \eqref{ineq: var-f} and \eqref{ineq: A} that for any $i=1,...,N, j=1,...,J, z_{ij}=1,$
\begin{align*}
&|y_{ij}| \leq \tau_2^{-1}\|y\|_{\sigma}(A_p)\left(N_{*}^{-1}+J_{*}^{-1}\right)^{\frac{1}{2}}.\numberthis\label{eq: bound-yij}
\end{align*}
Since $\vert \sigma^2(y_{ij})-\sigma_{ij}^2\vert \leq 1$, note that there exists a positive $\tau_3< \infty$ such that
for any $y \in \Omega_N,$ one has for any $i=1,...,N, j=1,...,J, z_{ij}=1,$
\begin{align*}
\vert\sigma^2(y_{ij})-\sigma_{ij}^2\vert \leq \tau_3\vert y_{ij}-m_{ij}^*\vert.\numberthis\label{eq:bound-vad-diff}
\end{align*}
Since $A_p$ consists of point maps only, by the definition of $\|\cdot\|_{\sigma}(A_p)$, we have $\|U_N(y,v)\|_{\sigma}(A_p)$ is the maximum value of 
$\vert f_{ij}\big\{U_N(y,v)\big\}\vert/\sigma(f_{ij})$ over $f_{ij}\in A_p.$ Therefore, upper bounding $\|U_N(y,v)\|_{\sigma}(A_p)$ is equivalent to upper bounding all $\vert U_{ij}(y,v)\vert/\sigma(f_{ij})$. 
Note that for any $i=1,...,N, j=1,...,J, z_{ij}=1$, 
\begin{align*}
   |U_{ij}(y,v)|&=\Big|\sum_{i'=1}^N\sum_{j' \in S_J(i')}\Big[d_{i'j'}(f_{ij})\big\{\sigma^2(y_{i'j'})-\sigma_{i'j'}^2\big\}v_{i'j'}\Big]\Big|\\
   &\leq \sum_{i'=1}^N\sum_{j' \in S_J(i')}\big\{|d_{i'j'}(f_{ij})|\big\}\big\{|\sigma^2(y_{i'j'})-\sigma_{i'j'}^2|\big\} \big\{|v_{i'j'}|\big\}\\
   &\leq \sum_{i'=1}^N\sum_{j' \in S_J(i')}\big\{|d_{i'j'}(f_{ij})|\big\}  \big\{\tau_3 |y_{i'j'}-m_{i'j'}^*|\big\}\big\{|v_{i'j'}|\big\}\\
   &\leq \tau_2^{-2}\tau_3\left(N_{*}^{-1}+J_{*}^{-1}\right)\Big\{\|y-M_N^*\|_{\sigma}(A_p)\|v\|_{\sigma}(A_p)\Big\}\Big\{\sum_{i'=1}^N\sum_{j' \in S_J(i')}|d_{i'j'}(f_{ij})|\Big\},
\end{align*}
where the second last line follows from \eqref{eq:bound-vad-diff} and the last line follows from \eqref{eq: bound-yij}.
Further note that 
\begin{align*}
\sum_{i'=1}^N\sum_{j' \in S_J(i')}|d_{i'j'}(f_{ij})| \leq \sum_{i'=1}^N\sum_{j' \in S_J(i')}|d_{i'j'}'(f_{ij})|+\sum_{i'=1}^N\sum_{j' \in S_J(i')}|d_{i'j'}''(f_{ij})|.
\end{align*}
By definition, 
$
 d'_{i'j'}(g)=(\sigma_{i'+}^2)^{-1}w_{i'+}+(\sigma_{+j'}^2)^{-1}w_{+j'}-(\sigma_{++})^{-1}w_{++}
$ for any $g\in \Omega_{N}^*.$
When $g=f_{ij}$, $w_{i'+}=0$ if $i'\neq i,$ and $w_{i'+}=1$ if $i'=i$, $w_{+j'}=0$ if $j'\neq j,$ and $w_{+j'}=1$ if $j'=j,$ and $w_{++}=1$. Therefore, we can rewrite
\begin{align*}
 \sum_{i'=1}^N\sum_{j' \in S_J(i')}|d_{i'j'}'(f_{ij})| =&\sum_{i'=1}^N\sum_{j' \in S_J(i')}\Big\vert(\sigma_{i'+}^2)^{-1}w_{i'+}+(\sigma_{+j'}^2)^{-1}w_{+j'}-(\sigma_{++})^{-1}w_{++}\Big\vert\\
 \leq& \sum_{i'=1}^N\sum_{j' \in S_J(i')}(\sigma_{i'+}^2)^{-1}\vert w_{i'+}\vert +\sum_{i'=1}^N\sum_{j' \in S_J(i')}(\sigma_{+j'}^2)^{-1}\vert w_{+j'}\vert\\
 &+\sum_{i'=1}^N\sum_{j' \in S_J(i')}(\sigma_{++})^{-1}\vert w_{++}\vert\\
 =&\sum_{j'\in S_J(i)}(\sigma_{i+}^2)^{-1}+\sum_{i'\in S_N(j)}(\sigma_{+j}^2)^{-1}+\sum_{i'=1}^N\sum_{j' \in S_J(i')}(\sigma_{++})^{-1}\leq \tau_4,
\end{align*}
where $\tau_4$ is some positive constant such that $\tau_4 < \infty.$ Note also that there exists $\tau_5<\infty$ such that
\begin{align*}
\sum_{i'=1}^N\sum_{j' \in S_J(i')}|d_{i'j'}''(f_{ij})| \leq (J^*N^*)^{\frac{1}{2}}\|d''(f_{ij})\|_{\sigma}\leq \tau_5.
\end{align*}
{where the last inequality is from~\eqref{eq:d2g rate} that $\|d''(f_{ij})\|_{\sigma} = o(N^{*-1})$.} As a result,
\begin{align*}
\|U_N(y,v)\|_{\sigma}(A_p) \leq \tau_1\tau_2^{-2}\tau_3(\tau_4+\tau_5)\left(N_{*}^{-1}+J_{*}^{-1}\right)^{\frac{1}{2}}\Big\{\|y-M_N^*\|_{\sigma}(A_p)\|v\|_{\sigma}(A_p)\Big\}.
\end{align*}
Therefore, we can set $d_N=\tau_1\tau_2^{-2}\tau_3(\tau_4+\tau_5)\left(N_{*}^{-1}+J_{*}^{-1}\right)^{\frac{1}{2}}$.
By \eqref{f_N^2-J-inverse}, we have $f_N(N_{*}^{-1}+J_{*}^{-1})^{1/2} \to 0$ as $N \to \infty.$ Therefore, it follows 
\begin{align*}
d_Nf_N&=\tau_1\tau_2^{-2}\tau_3(\tau_4+\tau_5)\left(N_{*}^{-1}+J_{*}^{-1}\right)^{\frac{1}{2}}f_N \to 0,\quad \text{as}\quad N \to \infty.
\end{align*}
Hence, the result of part (c) is also satisfied.
\end{proof}

{\sc Lemma \ref{lemma-bound-Zn}.}
{\it Let $A \subset \Omega_{N}^*.$ 
Let $C_N$ denote the cardinality of $A$. 
If there exist sequences $f_N> 0$ and $d_N\geq 0$ satisfying
(a). $0<C_N<\infty$ and $f_N^2/\log C_n \to \infty$ as $N \to \infty,$ (b). If $y, v \in \Omega_N$ and $\|y-M_N^*\|_{\sigma}(A)\leq f_N$, then there exists $n<\infty$ such that for all $ N > n$, $\|U_N(y, v)\|_{\sigma}(A) \leq d_N \|y-M_N^*\|_{\sigma}(A)\|v\|_{\sigma}(A),$ (c). $d_Nf_N \to 0$ as  $N \to \infty$.
Then pr$\big(\|R_N\|_{\sigma}(A) < \frac{1}{2}f_N \big) \to 1$ as $N \to \infty.$}
\begin{proof}
Denote $A=\{g_k:k=1,...,C_N\}$ and let $w_{k} \in \Omega_N$ be defined for $k=1,...,C_N$ by 
 $g_k(x)=[w_{k}, x]_{\sigma},  x\in \Omega_N.$
Let $W_{k}=\|w_{k}\|_{\sigma}^{-1}\sum_{i=1}^N\sum_{j \in S_J(i)}w_{ijk}(Y_{ij}-E_{ij})$ for $k=1,...,C_N$ so that 
 $\|R_N\|_{\sigma}(A)=\max_{k=1,...,C_N}\vert W_{k}\vert.$ 
We consider the log moment generating function of $W_{k}$, denoted as $\log G_{k}(t)$. Write $w'_{k}=w_{k}/\|w_{k}\|_{\sigma}$, $k=1,...,C_N,$ for simplicity, and we have
\begin{align*}
\log G_{k}(t)
&=\log \mathbb{E}[e^{tW_{k}}]\\
&=\log \mathbb{E}\Big[\exp\Big\{\frac{t}{\|w_{k}\|_{\sigma}}\sum_{i=1}^N\sum_{j \in S_J(i)}w_{ijk}(Y_{ij}-E_{ij})\Big\}\Big]\\
&=-t\sum_{i=1}^N\sum_{j \in S_J(i)}w_{ijk}'E_{ij}+\log \prod_{i=1}^N\prod_{j \in S_J(i)}\mathbb{E}\big\{\exp(tw_{ijk}'Y_{ij})\big\}, \quad \text{by independence}\\
&=-t\sum_{i=1}^N\sum_{j \in S_J(i)}w_{ijk}'E_{ij}+\sum_{i=1}^N\sum_{j \in S_J(i)}\log \mathbb{E}\big\{\exp(tw_{ijk}'Y_{ij})\big\}\\
&=\sum_{i=1}^N\sum_{j \in S_J(i)}\Big[\log\{1+\exp(m_{ij}^*)\}^{-1}-\log\{1+\exp(tw_{ijk}'+m_{ij}^*)\}^{-1}-tw_{ijk}'E_{ij}\Big]\\
&=\sum_{i=1}^N\sum_{j \in S_J(i)}\Big[\log\{h(m_{ij}^*)\}-\log\{h(tw_{ijk}'+m_{ij}^*)\}-tw_{ijk}'E_{ij}\Big],\numberthis\label{eq:mgf}
\end{align*}
where we have denoted $h(x)=\{1+\exp(x)\}^{-1}$. We apply Taylor expansion to $\log\{h(tw_{ijk}'+m_{ij}^*)\}$ with respect to $m_{ij}^*$. For some $t'=\alpha t$ with $0 < \alpha < 1$, we have,
\begin{align*}
\log\{h(tw_{ijk}'+m_{ij}^*)\}&=\log h(m_{ij}^*) -E_{ij}tw_{ijk}'-\frac{t^2}{2}w_{ijk}'^2\sigma^2\big(m_{ij}^*+t'w_{ijk}'\big).
\end{align*}
Substitute into \eqref{eq:mgf},
\begin{align*}
\log G_{k}(t) &=\frac{t^2}{2}\sum_{i=1}^N\sum_{j \in S_J(i)}w_{ijk}'^2\sigma^2\big(m_{ij}^*+t'w_{ijk}'\big),\quad \quad |t|\leq f_N.
\end{align*}
{Applying Markov inequality, we have}
{
\begin{eqnarray}
pr\Big(W_{k}\geq \frac{1}{2}f_N\Big) &=& pr\Big\{\exp(f_N W_{k}/2)\geq \exp(f_N^2/4)\Big\} \nonumber \\
&\leq & \frac{\mathbb{E}\{ \exp(f_N W_{k}/2)\}}{\exp\Big(f_N^2/4\Big)} \nonumber \\
&=&\exp\Big(-\frac{1}{4}f_N^2\Big)G_{k}\Big(\frac{1}{2}f_N\Big), \quad k=1,...,C_N, \nonumber
\end{eqnarray}
and similarly we get
$$pr\Big(-W_{k}\geq \frac{1}{2}f_N\Big)\leq \exp\Big(-\frac{1}{4}f_N^2\Big)G_{k}\Big(-\frac{1}{2}f_N\Big), \quad k=1,...,C_N.$$ 
}
Furthermore note that,
$$\log G_{k}\Big(\frac{1}{2}f_N\Big), \log G_{k}\Big(-\frac{1}{2}f_N\Big) \leq \frac{1}{8}f_N^2\Big(1+\frac{d_Nf_N}{2}\Big)\quad k=1,...,C_N.$$
Applying the Bonferroni inequality, 
\begin{align*}
 pr\Big\{\|R_N\|_{\sigma}(A)\geq \frac{1}{2}f_N\Big\}&\leq 2C_N\exp\Big\{-\frac{1}{8}f_N^2\big(1-\frac{d_Nf_N}{2}\big)\Big\}   \\
 &= 2\exp\Big\{\log C_N-\frac{1}{8}f_N^2\big(1-\frac{d_Nf_N}{2}\big)\Big\}\\
 &\to 0 \quad \text{as}\quad N\to \infty,
\end{align*}
where the last step follows from the assumption $f_N^2/\log C_N \to \infty$ as $N\to \infty.$ Hence the result of the lemma follows.
\end{proof}
{\sc Lemma \ref{lemma-bound-diff-btw-estimate-truth}.} {\it Assume Conditions~ \ref{cond:speed}, \ref{cond:bound} and \ref{cond:connect} hold. Let $A \subset \Omega_{N}^*.$  If there exist sequences $f_N > 0$ and $d_N\geq 0$ satisfying
(a). pr$\big(\|R_N\|_{\sigma}(A) < \frac{1}{2}f_N \big) \to 1$ as $N \to \infty,$ (b). If $y, v \in \Omega_N$ and $\|y-M_N^*\|_{\sigma}(A)\leq f_N$, then there exists $n<\infty$ such that for all $ N > n$, $\|U_N(y, v)\|_{\sigma}(A) \leq d_N \|y-M_N^*\|_{\sigma}(A)\|v\|_{\sigma}(A),$ (c). $d_Nf_N \to 0$ as  $N \to \infty$. Then, as $N\to \infty$, with probability approaching 1 that,}
\begin{align*}
\Big\vert\frac{\|\hat{M}_N-M_N^*\|_{\sigma}(A)}{\|R_N\|_{\sigma}(A)}-1\Big\vert\leq d_N^{\frac{1}{2}} \to 0 \quad\text{and}\quad
\|\hat{M}_N-M_N^*-R_N\|_{\sigma}(A)\leq d_N\|R_N\|_{\sigma}^2(A).
\end{align*}
\begin{proof}
Write $z_N= \|R_N\|_{\sigma}(A)$ for simplicity. 
Consider a sequence \{$h_{Nk}: k=0,1,...\}$, with $h_{N0}=0$ and   $h_{N(k+1)}=z_N+ d_Nh_{Nk}^2/2$ for $k=0, 1, 2,\ldots.$  Define another sequence $$l_N= \frac{2z_N}{1+(1-2z_Nd_N)^{\frac{1}{2}}}.$$
By Kantorovich and Akilov (1964, pages 695-711), if $z_N < \frac{1}{2}f_N$ and $z_N d_N < \frac{1}{2}$ (which hold with probability tending to 1 by (a), (b) and (c)), it follows 
\begin{align*}
\|t_{Nk}-\hat{M}_N\|_{\sigma}(A) \leq l_N-h_{Nk}, \quad  k= 0,1,2,..., \numberthis\label{eq:residual-parameter-rate}
\end{align*}
where $\{t_{Nk}: k=0,1,...\}$ is the sequence constructed in the proof of Theorem~\ref{thm-existence}.
When $k=0$,  \eqref{eq:residual-parameter-rate} implies
$
\|M_N^*-\hat{M}_N\|_{\sigma}(A) \leq l_N.
$
When $k=1$, \eqref{eq:residual-parameter-rate} implies
\begin{align*}
\|M_N^*+R_N-\hat{M}_N\|_{\sigma}(A)\leq l_N-z_N.\numberthis\label{eq:upper-bound-diff-M-residual}
\end{align*}
It follows that 
$
\big|\|M_N^*-\hat{M}_N\|_{\sigma}(A)-\|R_N\|_{\sigma}(A)\big| \leq l_N-z_N,
$
where  
\begin{align*}
l_N-z_N &=\frac{z_N\{1-(1-2z_Nd_N)^{\frac{1}{2}}\}}{1+(1-2z_Nd_N)^{\frac{1}{2}}}.
\end{align*}
If we view $x=z_Nd_N$ and $f(x)=\{1-(1-2x)^{1/2}\}/\{1+(1-2x)^{1/2}\}.$ 
We note $f(0)=0$, $f(1/2)=1$ and $f'(0)=1/4 < 1$ and $f''(x)>0$ for all $x<1/2$. Therefore, $f(x)<x$ for all $x<1/2.$ Hence, whenever $d_Nz_N<1/2,$  we must have
$l_N-z_N \leq d_Nz_N^2.$
We know that with probability tending to 1 that $d_Nz_N<1/2.$ 
Hence the second part of the lemma follows from \eqref{eq:upper-bound-diff-M-residual}.
Also as $N\to \infty$, with probability approaching 1 that, 
\begin{align*}
  \Big\vert\|\hat{M}_N-M_N^*\|_{\sigma}(A)-\|R_N\|_{\sigma}(A)\Big\vert^{2}\leq d_N\|R_N\|_{\sigma}^{2}(A). \numberthis\label{eq: consistency1}
\end{align*}
Re-write \eqref{eq: consistency1}, the result of the first part of the lemma then follows.
\end{proof}

{\sc Lemma \ref{lemma-peopleitemdiffvar}.}
{\it Assume Conditions \ref{cond:speed}, \ref{cond:bound} and \ref{cond:connect} hold and $\sum_{i=1}^N\theta_i=0$, the asymptotic variance of the maximum likelihood estimator of $m_{ij}^*,$ $var(\hat{m}_{ij})$,  for any $i=1,...,N$ and $j=1,...,J$, takes the form,}\yc{
\begin{align*}
var(\hat{m}_{ij}) = (\sigma_{i+}^2)^{-1}+(\sigma_{+j}^2)^{-1}+O\left(N_{*}^{-1}J_{*}^{-1}\right)\quad \text{as} \quad N \to \infty.
\end{align*}}
\begin{proof}
If $z_{ij}=1$, then we can simply use a linear function $f_{ij}$ with $f_{ij}(x)=x_{ij}.$
We apply $\|d'(f_{ij})\|_{\sigma}^2$ to approximate $\sigma^2(f_{ij})$. With $w_{i+}=1, w_{k+}=0,$ for all $k=1,...,i-1,i+1,...,N, w_{+j}=1, w_{+l}=0$ for all $l=1,...,j-1,j+1,...,J$ and $w_{++}=1$. We obtain
\begin{align*}
\|d'(f_{ij})\|_{\sigma}^2
&= (\sigma_{i+}^2)^{-1}+(\sigma_{+j}^2)^{-1}+O\left(N_{*}^{-1}J_{*}^{-1}\right)\quad \quad\text{as} \quad N\to \infty.
\end{align*}
If $z_{ij}= 0$, then we can apply Condition \ref{cond:connect}, there must exist $1 \leq i_1, i_2, ...., i_k \leq N$ and $1 \leq j_1, j_2,....,j_k \leq J$ such that $
z_{ij_1}=z_{i_1j_1}=z_{i_1j_2}=z_{i_2j_2}=...=z_{i_kj_k}=z_{i_kj}=1.
$
Consider a linear function $g_2$ defined as 
\begin{align*}
g_{ij}(x)&=x_{ij_1}-x_{i_1j_1}+x_{i_1j_2}-x_{i_2j_2}+...+x_{i_{k-1}j_k}-x_{i_kj_k}+x_{i_kj}\\
&=\theta_i-\beta_j.
\end{align*}
In this case, similarly we have $w_{i+}=1, w_{k+}=0,$ for all $k=1,...,i-1,i+1,...,N, w_{+j}=1, w_{+l}=0$ for all $l=1,...,j-1,j+1,...,J$ and $w_{++}=1$. Note these values are exactly the same as those of $g_1$. Therefore,
\begin{align*}
\|d'(g_{ij})\|_{\sigma}^2
&= (\sigma_{i+}^2)^{-1}+(\sigma_{+j}^2)^{-1}+O\left(N_{*}^{-1}J_{*}^{-1}\right)\quad \quad\text{as} \quad N\to \infty.
\end{align*}
In both cases, {$\|d''(g)\|_{\sigma}^2 = o( N^{*-2})$}. To see this, note that in both cases above, 
\begin{align*}
l_p&=\begin{cases}
   -\sigma_{pj}^2(\sigma_{+j}^2)^{-1}+\sigma_{p+}^2(\sigma_{++}^2)^{-1}& \text{if} \quad z_{pj}=1 \\
\sigma_{p+}^2(\sigma_{++}^2)^{-1}&\text{if} \quad z_{pj}=0
    \end{cases}\\
    &=O\left(N_{*}^{-1}\right)\quad \quad\text{as} \quad N\to \infty, \quad\quad p=1,...,N.
\end{align*}
\begin{align*}
v_q&= \begin{cases}
       -\sigma_{iq}^2(\sigma_{i+}^2)^{-1}+\sigma_{+q}^2(\sigma_{++}^2)^{-1}& \text{if} \quad z_{iq}=1 \\
       \sigma_{+q}^2(\sigma_{++}^2)^{-1}&\text{if} \quad z_{iq}=0
    \end{cases}\\
    &=O\left(J_{*}^{-1}\right)\quad \quad\text{as} \quad N\to \infty, \quad\quad q=1,...,J.
\end{align*}
It follows that 
{\begin{eqnarray}
    \|d''(f_{ij})\|_{\sigma}^2=\|d''(g_{ij})\|_{\sigma}^2&\leq & \gamma_N^{-2}\Big\{\sum_{p=1}^N l_p^2(\sigma_{p+}^2)^{-1}+\sum_{q=1}^Jv_q^2(\sigma_{+q}^2)^{-1}\Big\}^2 \nonumber \\
    &\leq & \gamma^{-2} \Big\{\sum_{p=1}^N l_p^2(\sigma_{p+}^2)^{-1}+\sum_{q=1}^Jv_q^2(\sigma_{+q}^2)^{-1}\Big\}^2 \label{eq:gamma N bound} \\
    &=&o\left( N^{*-2}\right)\quad \quad\text{as} \quad N\to \infty, \label{eq:d2g rate}
\end{eqnarray}}
{where~\eqref{eq:gamma N bound} is from the definition for $\gamma_N$ that there exist some $\gamma>0$ such that $\gamma_N>\gamma$ for all $N$. The last equation~\eqref{eq:d2g rate} follows from Condition \ref{cond:speed}(b)--(c).} 
Since for any $g\in \Omega_{N}^*$, $$\big(\|d'(g)\|_{\sigma}-\|d''(g)\|_{\sigma}\big)^2 \leq \sigma^2(g) \leq \big(\|d'(g)\|_{\sigma}+\|d''(g)\|_{\sigma}\big)^2,$$
it follows \yc{var$(\hat{m}_{ij})= (\sigma_{i+}^2)^{-1}+(\sigma_{+j}^2)^{-1}+O(N_{*}^{-1}J_{*}^{-1})$}  as $N\to \infty.$ Note that the {$O(N_{*}^{-1}J_{*}^{-1})$ and $o(N^{*-2})$ are} negligible comparing with the terms $(\sigma_{i+}^2)^{-1}$ and $(\sigma_{+j}^2)^{-1}$. 
\end{proof}

{\sc Lemma \ref{lemma-sequence-and-bound-A-beta}.}
{\it Assume {Conditions \ref{cond:bound}--\ref{cond:speed2}} hold. If $A_{\beta}=\{g_{j}: j=1,...,J\}$ such that $g_j \in \Omega_{N}^*$ and $g_j(x)=\beta_j$ for $x \in \Omega_N$. Let $C_N=|A_{\beta}|=J$ be the cardinality of $A_{\beta}$.  
For any positive sequence $f_N$ such that $f_N^2/\log J \to \infty$ and $f_N^2N_{*}^{-1/2}\to 0$ as $N \to \infty$, there exists a sequence $d_N \geq 0$ satisfying the followings. 

\noindent(a). If $y, v \in \Omega_N$ and $\|y-M_N^*\|_{\sigma}(A_{\beta})\leq f_N$, then there exists $n<\infty$ such that for all $ N > n$, $\|U_N(y, v)\|_{\sigma}(A_{\beta}) \leq d_N \|y-M_N^*\|_{\sigma}(A_{\beta})\|v\|_{\sigma}(A_{\beta}).$ 

\noindent(b). $d_Nf_N^2 \to 0$ as  $N \to \infty$.



}

\begin{proof}
{First we note we have $\log J \ll N_{*}^{1/2}$ by Condition \ref{cond:speed2}(b),} so the rate requirements for $f_N$ is valid.
To find a valid $d_N$, we seek to upper bound $\|U_N(y,v)\|_{\sigma}(A_{\beta})$ and then show that $d_Nf_N \to 0$ as $N\to \infty$ for all $f_N$ satisfying the rate requirements $f_N^2/\log J \to \infty$ and $f_N^2N_{*}^{-1/2}\to 0$ as $N \to \infty$.
For any $y, v \in \Omega_N,$ by the definition of $\|\cdot\|_{\sigma}(A_{\beta})$, we have $$\|U_N(y,v)\|_{\sigma}(A_{\beta}) = \max_{g_j\in A_{\beta}} \vert g_j\{U_N(y,v)\} \vert/\sigma(g_j).$$ 
First note that by Lemma \ref{lemma-itemvar}, $\sigma^2(g_j)=(\sigma_{+j}^2)^{-1}+O\big\{(N_{*}J_{*})^{-1}\big\}$ for any $g_j \in A_{\beta}$. Therefore, there exist positive $0<c_1, c_2<\infty$ such that for all $N>n,$
$
c_1^{-1}{N^{*-{1}/{2}}}<\sigma(g_j) < c_2^{-1}N_{*}^{-{1}/{2}},
$ for all $g_j \in A_{\beta}.$
So we just need to find an upper bound for $\vert g_j\{U_N(y,v)\} \vert$ that holds for all $g_j \in A_{\beta}.$  
Consider
\begin{align*}
\vert g_j\{U_N(y,v)\}\vert &= \Big\vert \sum_{i'=1}^N\sum_{j'\in S_J(i')}d_{i'j'}(g_j)\{\sigma^2(y_{i'j'})-\sigma_{i'j'}^2\}v_{i'j'}\Big\vert\\
&\leq \sum_{i'=1}^N\sum_{j'\in S_J(i')} \vert d_{i'j'}(g_j) \vert\cdot \vert\sigma^2(y_{i'j'})-\sigma_{i'j'}^2\vert\cdot\vert v_{i'j'}\vert.
\end{align*}
Note $0\leq \sigma^2(y_{ij}), \sigma_{ij}^2 \leq 1,$ so $\vert\sigma^2(y_{ij})-\sigma_{ij}^2\vert\leq 1$.  
It can be implied that there exists some positive $c_3<\infty$ such that $\vert\sigma^2(y_{ij})-\sigma_{ij}^2\vert \leq  c_3 \vert g_j(y-M_N^*) \vert$.
Again, by the definition of $\|\cdot\|_{\sigma}(A_{\beta})$, we have
$\vert g_j(y-M_N^*)\vert\leq \| y-M_N^*\|_{\sigma}(A_{\beta})\sigma(g_j).$ 
Therefore, 
for all $i=1,...,N, j=1,...,J, z_{ij}=1,$ \begin{align*}
\vert\sigma^2(y_{ij})-\sigma_{ij}^2\vert\leq c_2^{-1}c_3 N_{*}^{-1/2}\|y-M_N^*\|_{\sigma}(A_{\beta}).
\end{align*} 
On the other hand, 
using a similar strategy, we can show that there exists a positive $c_4<\infty$ such that for all $i=1,...,N, j=1,...,J, z_{ij}=1,$ 
\begin{align*}
\vert v_{ij}\vert \leq c_2^{-1}c_4N_{*}^{-1/2}\|v\|_{\sigma}(A_{\beta}).   
\end{align*} 
Further note that
\begin{align*}
\sum_{i'=1}^N\sum_{j'\in S_J(i')}\vert d_{i'j'}(g_j)\vert\leq \sum_{i'=1}^N\sum_{j'\in S_J(i')}\vert d_{i'j'}'(g_j)\vert +\sum_{i'=1}^N\sum_{j'\in S_J(i')}\vert d_{i'j'}''(g_j)\vert.
\end{align*}
By definition, we know $d_{i'j'}'=(\sigma_{i'+}^2)^{-1}w_{i'+}+(\sigma_{+j'}^2)^{-1}w_{+j'}-(\sigma_{++}^2)^{-1}w_{++}.$ 
For any $g_j \in A_{\beta},$ $w_{i'+}=-1/N,$ for $i'=1,...,N,$ $w_{+j'}=-1$ if $j'=j$ and $w_{+j'}=0$ if $j'\neq j,$ $w_{++}=-1$. 
Hence, 
\begin{align*}
d_{i'j'}'(g_j)&= \begin{cases}
      -\frac{1}{N}(\sigma_{i'+}^2)^{-1}-(\sigma_{+j'}^2)^{-1}+(\sigma_{++}^2)^{-1}& \text{if} \quad j'=j \\
       -\frac{1}{N}(\sigma_{i'+}^2)^{-1}+(\sigma_{++}^2)^{-1} &\text{if} \quad j'\neq j.
    \end{cases}
\end{align*}
It follows
\begin{align*}
\sum_{i'=1}^N\sum_{j'\in S_J(i')}\vert d_{i'j'}'(g_j)\vert &\leq  \frac{J^*}{N}\sum_{i'=1}^N(\sigma_{i'+}^2)^{-1}+N^*(\sigma_{+j}^2)^{-1}+\sum_{i'=1}^N\sum_{j'\in S_J(i')}(\sigma_{++}^2)^{-1} \leq c_5,
\end{align*}
for some positive $c_5< \infty.$
On the other hand,
\begin{align*}
\sum_{i'=1}^N\sum_{j'\in S_J(i')}\vert d_{i'j'}''(g_j)\vert& \leq (N^*J^*)^{\frac{1}{2}}\|d''(g_j)\|_{\sigma} \leq c_6,
\end{align*}
for some positive $c_6 < \infty$. {The last step follows from Lemma \ref{lemma-itemvar} which implies that $\|d''(g_j)\|_{\sigma}=o\left( N^{*-1}\right)$.} Overall, 
\begin{align*}
\|U_N(y,v)\|_{\sigma}(A_{\beta}) &= \max_{g_j\in A_{\beta}} \vert g_j\{U_N(y,v)\} \vert/\sigma(g_j)\\
&\leq \max_{g_j\in A_{\beta}} \vert g_j\{U_N(y,v)\} \vert \cdot \max_{g_j\in A_{\beta}}\{\sigma^{-1}(g_j)\}.\\
&\leq c_1c_2^{-2}c_3c_4(c_5+c_6) N_{*}^{-\frac{1}{2}}\|y-M_N^*\|_{\sigma}(A_{\beta})\|v\|_{\sigma}(A_{\beta}).
\end{align*}
Note that by taking $d_N= c_1c_2^{-2}c_3c_4(c_5+c_6) N_{*}^{-1/2},$ part (a) of the lemma follows. Furthermore, by the rate requirement of $f_N$, for any positive sequence $f_N$ such that $\log J\ll f_N^2\ll N_{*}^{1/2}$, it can be seen easily that $d_Nf_N^2 \to 0$ as $N\to \infty$. Therefore, part (b) of the lemma follows.
\end{proof}

{\sc Lemma \ref{lemma-itemvar}.}
{\it Assume {Conditions \ref{cond:bound}--\ref{cond:speed2}} hold and $\sum_{i=1}^N\theta_i=0$. The asymptotic variance of the maximum likelihood estimator of an individual column parameter, var$(\hat{\beta}_j)$, asymptotically attains the oracle variance $(\sigma_{+j}^{2})^{-1}$ in the sense that}
\begin{align*}
var(\hat{\beta}_j)=(\sigma_{+j}^{2})^{-1}+O(N_{*}^{-1}J_{*}^{-1})\quad \quad \text{as} \quad N \to \infty.\numberthis\label{eq:itemvar}
\end{align*}
\begin{proof} 
We seek to construct a linear function $g_j \in \Omega_{N}^*$ such that $g_j(x)=\beta_j$ so that we can use $\|d'(g_j)\|_{\sigma}^2$ defined in Lemma \ref{lemma-d'n} to approximate var$(\hat{\beta}_j).$
To construct such a $g_j$, we may want to include all $x_{ij}$, $i=1,...,N$, in $g_j$ so that we can apply the constraint $\sum_{i=1}^{N}\theta_i=0$ to solve for $\beta_j$. For $i \in S_{N}(j)$, we use $x_{ij}=\theta_i-\beta_j$ directly. 
For each $i \in S_{N_{\phi}}(j)$, by Condition \ref{cond:connect},
there must exist $1\leq i_{i1}, i_{i2},...,i_{ik}\leq N$ and $1\leq j_{i1}, j_{i2},...,j_{ik}\leq J$ such that 
$
    z_{i,j_{i1}}=z_{i_{i1},j_{i1}}=z_{i_{i1},j_{i2}}=z_{i_{i2},j_{i2}}=...=z_{i_{ik},j_{ik}}=z_{i_{ik},j}=1,
$ 
with
\begin{align*}
    &x_{i,j_{i1}}-x_{i_{i1},j_{i1}}+x_{i_{i1},j_{i2}}-x_{i_{i2},j_{i2}}+...-x_{i_{ik},j_{ik}}+x_{i_{ik},j}\\
    =&(\theta_i-\beta_{j_{i1}})-(\theta_{i_{i1}}-\beta_{j_{i1}})+(\theta_{i_{i1}}-\beta_{j_{i2}})-(\theta_{i_{i2}}-\beta_{j_{i2}})+...-(\theta_{i_{ik}}-\beta_{j_{ik}})+(\theta_{i_{ik}}-\beta_j)\\
    =&\theta_i-\beta_j.
\end{align*}
Therefore, we can construct $g$ to be
\begin{align*}
g_j(x)=&-\frac{1}{N}\Big\{\sum_{i\in S_N(j)}x_{ij}\\
&+\sum_{i\in S_{N_{\phi}}(j)}\Big(x_{i,j_{i1}}-x_{i_{i1},j_{i1}}+x_{i_{i1},j_{i2}}-x_{i_{i2},j_{i2}}+...-x_{i_{ik},j_{ik}}+x_{i_{ik},j}\Big)\Big\}\\
=&-\frac{1}{N}\Big\{\Big(\sum_{i=1}^{N}\theta_i\Big)-N\beta_j\Big\}\\
=&\beta_j.
\end{align*}
Use $\|d'(g_j)\|_{\sigma}^2$ from Lemma \ref{lemma-d'n} to approximate $\sigma^2(g_j)$, with $w_{i+}=-1/N,$  for all $i=1,...,N$, $w_{+j}=-1, w_{+l}=0$ for all $l=1,...j-1,j+1,...,J$ and $w_{++}=-1$. It follows 
\begin{align*}
    \|d'(g_j)\|_{\sigma}^2&=(\sigma_{+j}^{2})^{-1}+\frac{1}{N^2}\sum_{i=1}^{N}(\sigma_{i+}^2)^{-1}+\frac{2}{N}\sum_{i \in S_{N}(j)}\sigma_{ij}^2(\sigma_{i+}^2)^{-1}(\sigma_{+j}^2)^{-1}-3(\sigma_{++}^2)^{-1}\\
    &=\left(\sigma_{+j}^{2}\right)^{-1}+O\left(N_{*}^{-1}J_{*}^{-1}\right)\quad \quad\text{as} \quad N\to \infty.
\end{align*}
To see whether $\|d'(g_j)\|_{\sigma}^2$ is a good approximation for $\sigma^2(g_j)$, we need to evaluate the order of $\|d''(g_j)\|_{\sigma}^2$ from Lemma \ref{lemma-d''n}. Note
\begin{align*}
l_i&=
    \begin{cases}
      \sigma_{ij}^2(\sigma_{+j}^2)^{-1}-\sigma_{i+}^2(\sigma_{++}^2)^{-1} &  \text{if} \quad z_{ij}=1 \\
       -\sigma_{i+}^2(\sigma_{++}^2)^{-1} & \text{if} \quad z_{ij}=0
    \end{cases}\\
    &=O\left(N_{*}^{-1}\right)\quad \quad\text{as} \quad N\to \infty,\quad\quad i=1,...,N.
\end{align*}
\begin{align*}
v_q&=\frac{1}{N}\sum_{i\in S_N(q)}\sigma_{iq}^2(\sigma_{i+}^2)^{-1}-\sigma_{+q}^2(\sigma_{++}^2)^{-1}\\
    &=O\left(J_{*}^{-1}\right)\quad \quad\text{as} \quad N\to \infty,\quad\quad q=1,...,J.
\end{align*}
Applying Lemma \ref{lemma-d''n}, we have
{\begin{eqnarray}
    \|d''(g_{j})\|_{\sigma}^2&\leq & \gamma_N^{-2}\Big\{\sum_{i=1}^Nl_i^2(\sigma_{i+}^2)^{-1}+\sum_{q=1}^Jv_q^2(\sigma_{+q}^2)^{-1}\Big\}^2 \nonumber \\
    &\leq & \gamma^{-2} \Big\{\sum_{i=1}^Nl_i^2(\sigma_{i+}^2)^{-1}+\sum_{q=1}^Jv_q^2(\sigma_{+q}^2)^{-1}\Big\}^2 \nonumber \\
    &=&o\left( N^{*-2}\right)\quad \quad\text{as} \quad N\to \infty, \nonumber
\end{eqnarray}}
{where the last equation follows from Condition \ref{cond:speed}(b)--(c).} 
Since $$\big(\|d'(g_j)\|_{\sigma}-\|d''(g_j)\|_{\sigma}\big)^2 \leq \sigma^2(g_j) \leq \big(\|d'(g_j)\|_{\sigma}+\|d''(g_j)\|_{\sigma}\big)^2,$$
It follows that var$(\hat{\beta}_j)=  (\sigma_{+j}^2)^{-1}+O\left(N_{*}^{-1}J_{*}^{-1}\right) \text{ as }  N\to \infty.
$
\end{proof}

%



{\sc Lemma \ref{lemma-sequence-and-bound-A-theta}.}
{\it Assume {Conditions \ref{cond:bound}--\ref{cond:speed2}} hold. If $A_{\theta}=\{g_{i}: i=1,...,N\}$ such that $g_i \in \Omega_{N}^*$ and $g_i(x)=\theta_i$ for $x \in \Omega_N$. Let $C_N=|A_{\theta}|=N$ be the cardinality of $A_{\theta}$.  Then for any positive sequence $f_N$ such that $f_N^2/\log N \to \infty$ and $J_{*}^{-1}f_N^2 \to 0$ as $N\to \infty,$ there exists a sequence $d_N \geq 0$ satisfying the followings. 

\noindent(a) If $y, v \in \Omega_N$ and $\|y-M_N^*\|_{\sigma}(A_{\theta})\leq f_N$, then there exists $n<\infty$ such that for all $ N > n$, $\|U_N(y, v)\|_{\sigma}(A_{\theta}) \leq d_N \|y-M_N^*\|_{\sigma}(A_{\theta})\|v\|_{\sigma}(A_{\theta}).$ 

\noindent(b). $d_Nf_N \to 0$ as  $N \to \infty$.
}

\begin{proof}
We first note that from Condition \ref{cond:speed}(a), $\log N \ll J_{*}$ as $N \to \infty.$ Therefore, the rate requirements for the sequence $f_N$, $f_N^2/\log N \to \infty$ and $J_{*}^{-1}f_N^2 \to 0$ as $N\to \infty,$ are valid. 
Now we seek to upper bound $\|U_N(y,v)\|_{\sigma}(A_{\theta})$ to find a sequence $d_N$ and then show that $d_Nf_N \to 0$ for any $f_N$ satisfying $f_N^2/\log N \to \infty$ and $J_{*}^{-1}f_N^2 \to 0$ as $N\to \infty.$
For any $y, v \in \Omega_N,$ by the definition of $\|\cdot\|_{\sigma}(A_{\theta})$, $$\|U_N(y,v)\|_{\sigma}(A_{\theta}) = \max_{g_i\in A_{\theta}} \vert g_i\{U_N(y,v)\} \vert/\sigma(g_i).$$ 
Note that by Lemma \ref{lemma-peoplevar}, we know that $\sigma^2(g_i)=(\sigma_{i+}^2)^{-1}+O\big\{N_{*}^{-1}J_{*}^{-1}\big\}$ for any $g_i \in A_{\theta}$. Hence, there exist positive $0<\gamma_1, \gamma_2<\infty$ such that for any $i=1,...,N,$
$$\gamma_1^{-1} {J^{*-1/2}}<\sigma(g_i) < \gamma_2^{-1} J_{*}^{-1/2}.$$
So we just need to find an upper bound for $\vert g_i\{U_N(y,v)\} \vert$ that holds for all $g_i \in A_{\theta}.$  
For any $g_i \in A_{\theta}$,  we have
\begin{align*}
\vert g_i\{U_N(y,v)\}\vert &= \Big\vert \sum_{i'=1}^N\sum_{j'\in S_J(i')}d_{i'j'}(g_i)\{\sigma^2(y_{i'j'})-\sigma_{i'j'}^2\}v_{i'j'}\Big\vert\\
&\leq \sum_{i'=1}^N\sum_{j'\in S_J(i')} \vert d_{i'j'}(g_i) \vert\cdot \vert\sigma^2(y_{i'j'})-\sigma_{i'j'}^2\vert\cdot\vert v_{i'j'}\vert.
\end{align*}
Since $\sigma^2(y_{ij}), \sigma_{ij}^2<1,$ so $\vert\sigma^2(y_{ij})-\sigma_{ij}^2\vert\leq 1$. It can be implied that there exists a positive $\gamma_3<\infty$ such that $\vert\sigma^2(y_{ij})-\sigma_{ij}^2\vert\leq \gamma_3\vert g_i(y-M_N^*)\vert.$ From the definition of $\|\cdot\|_{\sigma}(A_{\theta}), $ $\vert g_i(y-M_N^*)\vert \leq \|y-M_N^*\|_{\sigma}(A_{\theta})\sigma(g_i)$ for any $g_i\in A_{\theta}.$ Then it follows that
for any $i=1,...,N, j=1,...,J, z_{ij}=1,$ $$\vert\sigma^2(y_{ij})-\sigma_{ij}^2\vert\leq \gamma_2^{-1}\gamma_3 J_{*}^{-1/2}\|y-M_N^*\|_{\sigma}(A_{\theta}).$$
Using a similar strategy, we can also show that there exists a positive $\gamma_4<\infty$ such that for any $i=1,...,N, j=1,...,J, z_{ij}=1,$ $$\vert v_{ij}\vert \leq \gamma_2^{-1}\gamma_4 J_{*}^{-1/2}\|v\|_{\sigma}(A_{\theta}).$$
Similarly, we have 
\begin{align*}
\sum_{i'=1}^N\sum_{j'\in S_J(i')}\vert d_{i'j'}(g_i)\vert\leq \sum_{i'=1}^N\sum_{j'\in S_J(i')}\vert d_{i'j'}'(g_i)\vert +\sum_{i'=1}^N\sum_{j'\in S_J(i')}\vert d_{i'j'}''(g_i)\vert.
\end{align*}
By definition, we know $d_{i'j'}'=(\sigma_{i'+}^2)^{-1}w_{i'+}+(\sigma_{+j'}^2)^{-1}w_{+j'}-(\sigma_{++}^2)^{-1}w_{++}.$ 
For any $g_i \in A_{\theta},$ $w_{i'+}=1-1/N,$ if $i'=i$, and $w_{i'+}=-1/N$ for $i'\neq i$, $w_{+j'}=0$ for all $j'=1,...,J$ and  $w_{++}=0$. 
Hence, 
\begin{align*}
d_{i'j'}'(g_i)&= \begin{cases}
      (1-\frac{1}{N})(\sigma_{i'+}^2)^{-1}& \text{if} \quad i'=i \\
       -\frac{1}{N}(\sigma_{i'+}^2)^{-1}&\text{if} \quad i'\neq i.
    \end{cases}
\end{align*}
It follows
\begin{align*}
\sum_{i'=1}^N\sum_{j'\in S_J(i')}\vert d_{i'j'}'(g_i)\vert &=\sum_{j'\in S_J(i)}\Big(1-\frac{1}{N}\Big)(\sigma_{i+}^2)^{-1}+\sum_{i'=1, i'\neq i}^N\sum_{j'\in S_J(i')}\frac{1}{N}(\sigma_{i'+}^2)^{-1} \leq \gamma_5,
\end{align*}
for some positive $\gamma_5< \infty.$
On the other hand,
\begin{align*}
\sum_{i'=1}^N\sum_{j'\in S_J(i')}\vert d_{i'j'}''(g_j)\vert& \leq (N^*J^*)^{\frac{1}{2}}\|d''(g_i)\|_{\sigma} \leq \gamma_6,
\end{align*}
for some positive $\gamma_6 < \infty$. {The last step follows from Lemma \ref{lemma-peoplevar} which implies that $\|d''(g_j)\|_{\sigma}=o\left( N^{*-1}\right)$.} 
Overall,
\begin{align*}
\|U_N(y,v)\|_{\sigma}(A_{\theta}) &= \max_{g_i\in A_{\theta}} \vert g_i\{U_N(y,v)\} \vert/\sigma(g_i)\\
&\leq \max_{g_i\in A_{\theta}} \vert g_i\{U_N(y,v)\} \vert \cdot \max_{g_i\in A_{\theta}}\{\sigma^{-1}(g_i)\}\\
&\leq \gamma_1\gamma_2^{-2}\gamma_3\gamma_4(\gamma_5+\gamma_6)J_{*}^{-\frac{1}{2}}\|y-M_N^*\|_{\sigma}(A_{\theta})\|v\|_{\sigma}(A_{\theta}).
\end{align*}
So we can set $d_N=\gamma_1\gamma_2^{-2}\gamma_3\gamma_4(\gamma_5+\gamma_6)J_{*}^{-\frac{1}{2}}$. 
Furthermore, by the rate requirement of $f_N$, for any positive sequence $f_N$ such that $(\log N)^{1/2} \ll f_N\ll J_{*}^{1/2}$, we must have $d_Nf_N \to 0$ as $N\to \infty$. Therefore, both part (a) and part (b) of the lemma are satisfied.
\end{proof}

{\sc Lemma \ref{lemma-peoplevar}.}
{\it Assume {Conditions \ref{cond:bound}--\ref{cond:speed2}} hold and $\sum_{i=1}^N\theta_i=0$, the asymptotic variance of an individual row parameter, var$(\hat{\theta}_i)$, asymptotically attains oracle variance $(\sigma_{i+}^2)^{-1}$ in the sense that
\begin{align*}
 var(\hat{\theta}_i)=(\sigma_{i+}^2)^{-1}+O\left(N_{*}^{-1}J_{*}^{-1}\right)\quad\quad \text{as}\quad N\to \infty.\numberthis\label{eq:peoplevar}
\end{align*}
} 
\begin{proof}
We seek to construct a linear function $g_i \in \Omega_{N}^*$ such that $g_i(x)=\theta_i$ so that we can use $\|d'(g_i)\|_{\sigma}^2$ in Lemma \ref{lemma-d'n} to approximate var$(\hat{\theta}_i).$
Fix some $j \in S_J(i)$, i.e. $z_{ij}=1$, since Condition~\ref{cond:connect} holds, we can use the linear function $g_j$ constructed in the proof of Theorem~\ref{thm-existence} to represent $\beta_j$, i.e. $g_j(x)=\beta_j.$
Hence, $g_i$ can easily be constructed with
\begin{align*}
g_i(x)&=\frac{1}{\vert S_J(i)\vert }\sum_{j\in S_J(i)}\{x_{ij}+g_j(x)\}\\
&=\frac{1}{\vert S_J(i)\vert }\sum_{j\in S_J(i)}\Big[x_{ij}-\frac{1}{N}\Big\{\sum_{i'\in S_N(j)}x_{i'j}\\
&+\sum_{i'\in S_{N_{\phi}}(j)}\Big(x_{i',j_{i'1}}-x_{i'_{i'1},j_{i'1}}+x_{i'_{i'1},j_{i'2}}-x_{i'_{i'2},j_{i'2}}+...-x_{i'_{i'k},j_{i'k}}+x_{i'_{i'k},j}\Big)\Big\}\Big]\\
&=\theta_i.
\end{align*}
We use $\|d'(g_i)\|_{\sigma}^2$ from Lemma \ref{lemma-d'n} to approximate $\sigma^2(g_i)$ , with $w_{i+}=1-N^{-1}, w_{k+}=-N^{-1},$ for all $ k=1,...,i-1,i+1,...,N, w_{+j}=0,$ for all $j=1,...,J, w_{++}=0$, we obtain
\begin{align*}
\|d'(g_i)\|_{\sigma}^2
&= \big(1-\frac{1}{N}\big)^{2}(\sigma_{i+}^2)^{-1}+\frac{1}{N^2}\sum_{k=1, k\neq i}^{N}(\sigma_{k+}^2)^{-1}\\
&=(\sigma_{i+}^2)^{-1}+O\left(N_{*}^{-1}J_{*}^{-1}\right)\quad\quad \text{as}\quad N\to \infty.
\end{align*}
To see whether $\|d'(g_i)\|_{\sigma}^2$ is a good approximation for $\sigma^2(g_i)$, we evaluate the order of $\|d''(g_i)\|_{\sigma}^2$. Note that in this case
\begin{align*}
l_p&=0, \quad\quad p=1,...,N.\\
v_q&= \begin{cases}
      \frac{1}{N}\sum_{k \in S_N(q), k\neq i} \sigma_{kq}^2(\sigma_{k+}^2)^{-1}-(1-\frac{1}{N})\sigma_{iq}^2(\sigma_{i+})^{-1}& \text{if} \quad z_{iq}=1 \\
       \frac{1}{N}\sum_{k \in S_N(q)} \sigma_{kq}^2(\sigma_{k+}^2)^{-1} &\text{if} \quad z_{iq}=0
    \end{cases}\\
    &=O\left(J_{*}^{-1}\right)\quad\quad \text{as}\quad N\to \infty, \quad\quad q=1,...,J.
\end{align*}
It follows that
{\begin{eqnarray}
    \|d''(g_{i})\|_{\sigma}^2&\leq & \gamma_N^{-2}\Big\{\sum_{i=1}^Nl_i^2(\sigma_{i+}^2)^{-1}+\sum_{q=1}^Jv_q^2(\sigma_{+q}^2)^{-1}\Big\}^2 \nonumber \\
    &\leq & \gamma^{-2} \Big\{\sum_{i=1}^Nl_i^2(\sigma_{i+}^2)^{-1}+\sum_{q=1}^Jv_q^2(\sigma_{+q}^2)^{-1}\Big\}^2 \nonumber \\
    &=&o\left( N^{*-2}\right)\quad \quad\text{as} \quad N\to \infty, \nonumber
\end{eqnarray}}
{where the last equation follows from Condition \ref{cond:speed}(b)--(c).} Since $$\big(\|d'(g_i)\|_{\sigma}-\|d''(g_i)\|_{\sigma}\big)^2 \leq \sigma^2(g_i) \leq \big(\|d'(g_i)\|_{\sigma}+\|d''(g_i)\|_{\sigma}\big)^2,$$
it follows that var$(\hat{\theta}_i)= (\sigma_{i+}^2)^{-1}+O(N_{*}^{-1}J_{*}^{-1})$ {as} $N\to \infty.
$
\end{proof}

{\sc Lemma \ref{sufficient-condition-var-approx}.}
{\it Assume {Conditions \ref{cond:bound}--\ref{cond:speed2}} hold and $\sum_{i=1}^{N}\theta_i=0$. Consider a linear function $g: \Omega_N \mapsto \mathbb{R}$ with $g(M)=\sum_{i=1 }^Nh_i\theta_i + \sum_{j=1}^Jh_{j}'\beta_j.$ If there exists a positive $C < \infty$ such that $\sum_{i=1}^N\vert h_i\vert<C$ and $\sum_{j=1}^J \vert h_j'\vert <C$, then }
\begin{align*}
\sigma^2(g)=\sum_{i=1}^Nh_{i}^2(\sigma_{i+}^2)^{-1}+\sum_{j=1 }^Jh_{j}'^2(\sigma_{+j}^2)^{-1}+O\left(N_{*}^{-1}J_{*}^{-1}\right)\quad \text{as}\quad N\to \infty.
\end{align*}
\begin{proof}
By Proposition~\ref{prop:connect}, we can reexpress function $g$ in terms of $m_{ij}$ for $i=1,...,N, j=1,...,J, z_{ij}=1$ with $g(M_N)=\sum_{i=1}^N\sum_{j \in S_J(i)}w_{ij}(g)m_{ij}.$ In particular, we have,
\begin{align*}
&w_{i+}(g)=h_i\big(1-\frac{1}{N}\big) -\frac{1}{N}\sum_{i'=1, i'\neq i}^Nh_{i'}-\frac{1}{N} \sum_{j=1}^Jh_j'\quad\quad i=1,...,N,  \\
&w_{+j}(g)=-h_j',\quad\quad j=1,...,J,\\
&w_{++}(g)=-\sum_{j=1}^Jh_j'.
\end{align*}
We apply $\|d'(g)\|_{\sigma}^2$ from Lemma \ref{lemma-d'n} to approximate $\sigma^2(g)$. Note that 
\begin{align*}
\|d'(g)\|_{\sigma}^2=
&\sum_{i=1}^N w_{i+}^2(g)(\sigma_{i+}^2)^{-1}+\sum_{j=1}^J w_{+j}^2(g)(\sigma_{+j}^2)^{-1}\\
&+2\sum_{i=1}^N\sum_{j\in S_j(i)}\sigma_{ij}^2(\sigma_{i+}^2)^{-1} w_{i+}(g)(\sigma_{+j}^2)^{-1} w_{+j}(g)-3(\sigma_{++}^2)^{-1}w_{++}^2(g)\\
&=\sum_{i=1}^N h_{i}^2(\sigma_{i+}^2)^{-1}+\sum_{j= 1}^Jh_{j}'^2(\sigma_{+j}^2)^{-1}+O\left(N_{*}^{-1}J_{*}^{-1}\right) \quad \text{as}\quad N\to \infty,
\end{align*}
where the last step follows from the assumption that $\sum_{i=1}^N\vert h_i\vert<C$ and $\sum_{j=1}^J\vert h_j'\vert<C$.
To see whether $\|d'(g)\|_{\sigma}^2$ is a good approximation for $\sigma^2(g)$, we need to evaluate the order of $\|d''(g)\|_{\sigma}^2$.
Note that for $i=1,...,N,$
\begin{align*}
 l_i &= -\sum_{j \in S_J(i)}\sigma_{ij}^2(\sigma_{+j}^2)^{-1}w_{+j}(g)+\sigma_{i+}^2(\sigma_{++}^2)^{-1}w_{++}(g)\\
 &=\sum_{j \in S_J(i)}\sigma_{ij}^2(\sigma_{+j}^2)^{-1}h_j'-\sigma_{i+}^2(\sigma_{++}^2)^{-1}\sum_{j=1}^Jh_j'=O\left(N_{*}^{-1}\right) \quad \text{as}\quad N\to \infty,\numberthis\label{eq:order-li}
\end{align*}
where the last step follows from $\sum_{j=1}^J\vert h_j'\vert<C$.
Similarly for $j=1,...,J,$
\begin{align*}
 v_j 
 =&-\sum_{i \in S_N(j)}\sigma_{ij}^2(\sigma_{i+}^2)^{-1}w_{i+}(g)+\sigma_{+j}^2(\sigma_{++}^2)^{-1}w_{++}(g)\\
 =&-\sum_{i \in S_N(j)}\sigma_{ij}^2(\sigma_{i+}^2)^{-1}\Big\{h_i\big(1-\frac{1}{N}\big) -\frac{1}{N}\sum_{i'=1, i'\neq i}^Nh_{i'}-\frac{1}{N} \sum_{j=1}^Jh_j'\Big\}\\
 &-\sigma_{+j}^2(\sigma_{++}^2)^{-1}\sum_{j=1}^Jh_j'\\
 =&O\left(J_{*}^{-1}\right)\quad \text{as}\quad N\to \infty,\numberthis\label{eq:order-vj}
\end{align*}
where the last step follows from $\sum_{j=1}^J\vert h_j'\vert<C$ and $\sum_{i=1}^N\vert h_i\vert<C$.
Hence, we have 
\begin{align*}
    \|d''(g)\|_{\sigma}^2&\leq \gamma_N^{-2}\Big\{\sum_{i=1}^N l_i^2(\sigma_{i+}^2)^{-1}+\sum_{j=1}^Jv_j^2(\sigma_{+j}^2)^{-1}\Big\}^2\\
    &={o\left(N^{*-2}\right)}\quad \text{as}\quad N\to \infty,
\end{align*}
{where the last equation follows from \eqref{eq:order-li}, \eqref{eq:order-vj} and Condition \ref{cond:speed}(b)--(c).}
It follows that \begin{align*}
\sigma^2(g)=\sum_{i=1}^Nh_{i}^2(\sigma_{i+}^2)^{-1}+\sum_{j=1}^Jh_{j}'^2(\sigma_{+j}^2)^{-1}+O\left(N_{*}^{-1}J_{*}^{-1}\right)\quad \quad\text{as} \quad N\to \infty.
\end{align*}
Hence, the result of the lemma follows.
\end{proof}

{\sc Lemma \ref{lemma-sequence-and-bound-A-theta-beta}. }
{\it Assume {Conditions \ref{cond:bound}-- \ref{cond:speed2}} hold. If $A_{\theta, \beta}=\{g_i, g_j':i=1,...,N, j=1,...,J\}$ such that $g_i, g_j' \in \Omega_{N}^*$, and $g_i(x)=\theta_i$ and $g_j'(x)=\beta_j$ for $x \in \Omega_N$. Let $C_N=|A_{\theta, \beta}|$, the cardinality of $A_{\theta, \beta}$.  Then there exist sequences $f_N>0$ and $d_N\geq 0$ satisfying the followings.

\noindent(a). As $N \to \infty$, $f_N^2/\log C_N \to \infty.$ 

\noindent(b). If $y, v \in \Omega_N$ and $\|y-M_N^*\|_{\sigma}(A_{\theta, \beta})\leq f_N$, then there exists $n<\infty$ such that for all $ N > n$, $\|U_N(y, v)\|_{\sigma}(A_{\theta, \beta}) \leq d_N \|y-M_N^*\|_{\sigma}(A_{\theta, \beta})\|v\|_{\sigma}(A_{\theta, \beta}).$ Furthermore, $d_Nf_N^2 \to 0$ as  $N \to \infty$.
}

\begin{proof}
{From Condition \ref{cond:speed2}(a), we have $J_{*}^{-2}N_{*}(\log N)^2\to 0$ as $N \to \infty$}, there must exists a positive sequence $L_N$ such that $L_N \to \infty$ but $J_{*}^{-1}N_{*}^{1/2}(\log N) L_N\to 0$ as $N\to \infty.$ Furthermore, note that $$\log(C_N)=\log(N+J)\leq \log(2N)=\log(2)+\log(N) =O(\log(N))\quad \text{as} \quad N\to \infty.$$ 
Let $f_N^2=\{\log (N)\}L_N$. It is easy to see that the constructed $f_N$ satisfies part (a) of the lemma.

Now we consider part (b). We seek to find an upper bound for $\|U_N(y,z)\|_{\sigma}(A_{\theta,\beta})$ in order to find $d_N$ and then show that $d_Nf_N^2 \to 0$ as $N\to \infty.$
For any $y, v \in \Omega_N,$ by the definition of $\|\cdot\|_{\sigma}(A_{\theta,\beta})$, $$\|U_N(y,v)\|_{\sigma}(A_{\theta,\beta}) = \max_{f\in A_{\theta,\beta}} \vert f\{U_N(y,v)\} \vert/\sigma(f).$$ 
First note from \eqref{eq:itemvar} and \eqref{eq:peoplevar}, we know that for any $f \in A_{\theta, \beta}$, there exist $0<c_1, c_2<\infty$ such that for all $N>n$,
\begin{align*}
 c_1^{-1}N_{*}^{-\frac{1}{2}}<\sigma(f) < c_2^{-1}J_{*}^{-\frac{1}{2}}.  
\end{align*}
So we just need to find an upper bound for $\vert f\{U_N(y,v)\} \vert$ that holds for all $f \in A_{\theta,\beta}.$ Note that
\begin{align*}
\vert f\{U_N(y,v)\}\vert &= \Big\vert \sum_{i'=1}^N\sum_{j'\in S_J(i')}d_{i'j'}(f)\{\sigma^2(y_{i'j'})-\sigma_{i'j'}^2\}v_{i'j'}\Big\vert\\
&\leq \sum_{i'=1}^N\sum_{j'\in S_J(i')} \vert d_{i'j'}(f) \vert\cdot \vert\sigma^2(y_{i'j'})-\sigma_{i'j'}^2\vert\cdot\vert v_{i'j'}\vert.\numberthis\label{eq: beta-utility-bound} 
\end{align*}
Note $0\leq \sigma^2(y_{ij}), \sigma_{ij}^2 \leq 1,$ so $\vert\sigma^2(y_{ij})-\sigma_{ij}^2\vert\leq 1$.  
It can be implied that $\vert\sigma^2(y_{ij})-\sigma_{ij}^2\vert\leq c_3\vert f(y-M_N^*)\vert$ for some positive $c_3<\infty.$ By the definition of $\|\cdot\|_{\sigma}(A_{\theta, \beta})$, we have $\vert f(y-M_N^*)\vert\leq \|y-M_N^*\|_{\sigma}(A_{\theta, \beta})\sigma(f).$
Hence, it follows that for any $i=1,...,N, j=1,...,J, z_{ij}=1,$ $$\vert\sigma^2(y_{ij})-\sigma_{ij}^2\vert\leq c_2^{-1}c_3J_{*}^{-1/2}\|y-M_N^*\|_{\sigma}(A_{\theta, \beta}).$$ 
Using a similar strategy, we can show that there exists a positive $c_4<\infty$ such that for any $i=1,...,N, j=1,...,J, z_{ij}=1,$
$$\vert v_{ij}\vert \leq c_2^{-1}c_4J_{*}^{-1/2}\|v\|_{\sigma}(A_{\theta, \beta}).$$ Further, note also that 
\begin{align*}
\sum_{i'=1}^N\sum_{j'\in S_J(i')}\vert d_{i'j'}(f)\vert\leq \sum_{i'=1}^N\sum_{j'\in S_J(i')}\vert d_{i'j'}'(f)\vert +\sum_{i'=1}^N\sum_{j'\in S_J(i')}\vert d_{i'j'}''(f)\vert.
\end{align*}
By definition, $d_{i'j'}'=(\sigma_{i'+}^2)^{-1}w_{i'+}+(\sigma_{+j'}^2)^{-1}w_{+j'}-(\sigma_{++}^2)^{-1}w_{++}.$ 
For any $f \in A_{\theta, \beta},$ either $f=g_j'$ or $f=g_i$. When $f=g_j'$, $w_{i'+}=-1/N,$ for $i'=1,...,N,$ $w_{+j'}=-1$ if $j'=j$ and $w_{+j'}=0$ if $j'\neq j,$ $w_{++}=-1$. 
Hence, 
\begin{align*}
d_{i'j'}'(g_j')&= \begin{cases}
      -\frac{1}{N}(\sigma_{i'+}^2)^{-1}-(\sigma_{+j'}^2)^{-1}+(\sigma_{++}^2)^{-1}& \text{if} \quad j'=j \\
       -\frac{1}{N}(\sigma_{i'+}^2)^{-1}+(\sigma_{++}^2)^{-1} &\text{if} \quad j'\neq j
    \end{cases}
\end{align*}
It follows
\begin{align*}
\sum_{i'=1}^N\sum_{j'\in S_J(i')}\vert d_{i'j'}'(g_j')\vert &\leq  \frac{J^*}{N}\sum_{i'=1}^N(\sigma_{i'+}^2)^{-1}+N^*(\sigma_{+j}^2)^{-1}+\sum_{i'=1}^N\sum_{j'\in S_J(i')}(\sigma_{++}^2)^{-1} \leq c_5,
\end{align*}
for some positive $c_5< \infty.$ Furthermore,
\begin{align*}
\sum_{i'=1}^N\sum_{j'\in S_J(i')}\vert d_{i'j'}''(g_j')\vert& \leq (N^*J^*)^{\frac{1}{2}}\|d''(g_j')\|_{\sigma} \leq c_6,
\end{align*}
for some positive $c_6 < \infty$. {The last step follows from Lemma \ref{lemma-itemvar} which implies that $\|d''(g_j')\|_{\sigma}=o(N^{*-1})$.}
On the other hand, when $f=g_i,$ we have $w_{i'+}=1-1/N,$ if $i'=i$, and $w_{i'+}=-1/N$ for $i'\neq i$, $w_{+j'}=0$ for all $j'=1,...,J$ and  $w_{++}=0$. 
Hence, 
\begin{align*}
d_{i'j'}'(g_i)&= \begin{cases}
      (1-\frac{1}{N})(\sigma_{i'+}^2)^{-1}& \text{if} \quad i'=i \\
       -\frac{1}{N}(\sigma_{i'+}^2)^{-1}&\text{if} \quad i'\neq i.
    \end{cases}
\end{align*}
It follows
\begin{align*}
\sum_{i'=1}^N\sum_{j'\in S_J(i')}\vert d_{i'j'}'(g_i)\vert &=\sum_{j'\in S_J(i)}\Big(1-\frac{1}{N}\Big)(\sigma_{i+}^2)^{-1}-\sum_{i'=1, i'\neq i}^N\sum_{j'\in S_J(i')}\frac{1}{N}(\sigma_{i'+}^2)^{-1} \leq c_7,
\end{align*}
for some positive $c_7< \infty.$ Furthermore,
\begin{align*}
\sum_{i'=1}^N\sum_{j'\in S_J(i')}\vert d_{i'j'}''(g_i)\vert& \leq (N^*J^*)^{\frac{1}{2}}\|d''(g_i)\|_{\sigma} \leq c_8,
\end{align*}
for some positive $c_8 < \infty$. {The last step follows from Lemma \ref{lemma-peoplevar} which implies that $\|d''(g_i)\|_{\sigma}=o(N^{*-1})$.}
Overall, 
\begin{align*}
\|U_N(y,v)\|_{\sigma}(A_{\theta,\beta})& = \max_{f\in A_{\theta,\beta}} \vert f\{U_N(y,v)\} \vert/\sigma(f)\\
&\leq \max_{f\in A_{\theta,\beta}} \vert f\{U_N(y,v)\} \vert \max_{f\in A_{\theta,\beta}} \{\sigma(f)^{-1}\}\\
&\leq c_1c_2^{-2}c_3c_4\max\{c_5+c_6, c_7+c_8\} J_{*}^{-1}N_{*}^{\frac{1}{2}}\|y-M_N^*\|_{\sigma}(A_{\theta, \beta})\|v\|_{\sigma}(A_{\theta, \beta}).
\end{align*}
Note that in this case we can take $d_N=c_1c_2^{-2}c_3c_4\max\{c_5+c_6, c_7+c_8\} J_{*}^{-1}N_{*}^{1/2}.$ 
We have $$d_Nf_N^2=c_1c_2^{-2}c_3c_4\max\{c_5+c_6, c_7+c_8\} J_{*}^{-1}N_{*}^{1/2}\log(N)L_N \to 0\quad \text{as}\quad N\to \infty.$$  Hence both parts (a) and (b) of the lemma are satisfied.
\end{proof}

\section*{Appendix C: Full Senator Rankings}\label{appendix-senator-ranking}
Appendix C includes additional results for Section \ref{sec-real-data2} ``Application to Senate Voting'' of the main article.
In specific, with the same set-up as in Section \ref{sec-real-data2}, we give a full list of rankings for senators serving the 111th, the 112th and the 113th United States senate according to their conservativeness scores. The results are summarized in Tables \ref{table: ranking-62} and \ref{table: ranking-63-139} below. We observe from Table \ref{table: ranking-62} that all the top 62 most conservative senators predicted by the model are Republicans. While the Democrats and the independent politicians are predicted to have much lower conservativeness scores as presented in Table \ref{table: ranking-63-139}. This aligns well with the public perceptions about the Republican party and the Democratic party. Standard errors of the estimated row parameters (i.e. senator's conservativeness score) are also included to facilitate inferences.
\begin{table}[!ht]
\centering
\begin{tabular}{p{0.6cm}p{1.5cm}p{0.6cm}p{0.7cm}cp{0.8cm}|p{0.6cm}p{1.5cm}p{0.6cm}p{0.7cm}c p{0.8cm}}
  \hline
 Rank& Senator & State & Party &  $\hat{\theta}$ & s.e.$(\hat{\theta}$)&Rank&Senator & State & Party &  $\hat{\theta}$  & s.e.$(\hat{\theta}$)  \\ 
  \hline
1 & Demint      & SC & Rep & 5.87 & 0.157 & 2 & Lee         & UT & Rep & 5.73 & 0.138 \\ 
   3 & Cruz        & TX & Rep & 5.65 & 0.195 & 4 & Coburn      & OK & Rep & 5.25 & 0.114 \\ 
   5 & Paul        & KY & Rep & 5.24 & 0.129 & 6 & Scott       & SC & Rep & 5.17 & 0.176 \\ 
   7 & Bunning     & KY & Rep & 4.92 & 0.204 & 8 & Johnson     & WI & Rep & 4.84 & 0.119 \\ 
   9 & Risch       & ID & Rep & 4.81 & 0.102 & 10 & Inhofe      & OK & Rep & 4.69 & 0.103 \\ 
   11 & Crapo       & ID & Rep & 4.56 & 0.097 & 12 & Sessions    & AL & Rep & 4.48 & 0.096 \\ 
   13 & Enzi        & WY & Rep & 4.36 & 0.094 & 14 & Barasso     & WY & Rep & 4.35 & 0.094 \\ 
   15 & Cornyn      & TX & Rep & 4.33 & 0.095 & 16 & Rubio       & FL & Rep & 4.25 & 0.112 \\ 
   17 & Ensign      & NV & Rep & 4.24 & 0.166 & 18 & Vitter      & LA & Rep & 4.20 & 0.094 \\ 
   19 & Fischer     & NE & Rep & 4.14 & 0.145 & 20 & Toomey      & PA & Rep & 4.12 & 0.109 \\ 
   21 & Kyl         & AZ & Rep & 4.10 & 0.115 & 22 & Roberts     & KS & Rep & 4.06 & 0.091 \\ 
   23 & Mcconnell   & KY & Rep & 4.02 & 0.089 & 24 & Thune       & SD & Rep & 3.95 & 0.088 \\ 
   25 & Burr        & NC & Rep & 3.95 & 0.090 & 26 & Moran       & KS & Rep & 3.89 & 0.109 \\ 
   27 & Grassley    & IA & Rep & 3.80 & 0.086 & 28 & Shelby      & AL & Rep & 3.78 & 0.086 \\ 
   29 & Boozman     & AR & Rep & 3.68 & 0.105 & 30 & Chambliss   & GA & Rep & 3.65 & 0.087 \\ 
   31 & Mccain      & AZ & Rep & 3.65 & 0.086 & 32 & Brownback   & KS & Rep & 3.61 & 0.153 \\ 
   33 & Coats       & IN & Rep & 3.51 & 0.101 & 34 & Johanns     & NE & Rep & 3.39 & 0.082 \\ 
   35 & Isakson     & GA & Rep & 3.38 & 0.082 & 36 & Hatch       & UT & Rep & 3.38 & 0.083 \\ 
   37 & Lemieux     & FL & Rep & 3.34 & 0.188 & 38 & Blunt       & MO & Rep & 3.31 & 0.099 \\ 
   39 & Wicker      & MS & Rep & 3.29 & 0.080 & 40 & Portman     & OH & Rep & 3.28 & 0.098 \\ 
   41 & Corker      & TN & Rep & 3.27 & 0.080 & 42 & Heller      & NV & Rep & 3.26 & 0.100 \\ 
   43 & Hutchison   & TX & Rep & 3.25 & 0.105 & 44 & Graham      & SC & Rep & 3.18 & 0.080 \\ 
   45 & Flake       & AZ & Rep & 3.03 & 0.125 & 46 & Ayotte      & NH & Rep & 3.02 & 0.095 \\ 
   47 & Hoeven      & ND & Rep & 2.97 & 0.094 & 48 & Bennett     & UT & Rep & 2.74 & 0.127 \\ 
   49 & Alexander   & TN & Rep & 2.71 & 0.075 & 50 & Kirk        & IL & Rep & 2.67 & 0.105 \\ 
   51 & Cochran     & MS & Rep & 2.63 & 0.075 & 52 & Chiesa      & NJ & Rep & 2.61 & 0.343 \\ 
   53 & Gregg       & NH & Rep & 2.59 & 0.127 & 54 & Martinez    & FL & Rep & 2.47 & 0.186 \\ 
   55 & Lugar       & IN & Rep & 2.29 & 0.088 & 56 & Bond        & MO & Rep & 2.25 & 0.118 \\ 
   57 & Murkowski   & AK & Rep & 1.47 & 0.066 & 58 & Brown       & MA & Rep & 1.29 & 0.103 \\ 
   59 & Voinovich   & OH & Rep & 1.22 & 0.102 & 60 & Snowe       & ME & Rep & 1.06 & 0.080 \\ 
   61 & Specter     & PA & Rep & 1.03 & 0.192 & 62 & Collins     & ME & Rep & 0.82 & 0.064 \\ 
  \hline
\end{tabular}
\caption{Ranking of the top 62 most conservative senators predicted by the model. Rep represents the Republican party and the states are listed in their standard abbreviations. $\hat{\theta}$ represents the conservativeness score of senators and s.e.$(\hat{\theta})$ is the standard error of the estimated conservativeness score.}
\label{table: ranking-62}
\end{table}

\begin{table}[!ht]
\small
\centering
\begin{tabular}{p{0.6cm}p{1.5cm}p{0.6cm}p{0.7cm}cp{0.8cm}|p{0.6cm}p{1.5cm}p{0.6cm}p{0.7cm}cp{0.8cm}}
  \hline
 Rank& Senator & State & Party &  $\hat{\theta}$ & s.e.$(\hat{\theta}$)&Rank&Senator & State & Party &  $\hat{\theta}$  & s.e.$(\hat{\theta}$)  \\ 
  \hline
63 & Nelson  & NE & Dem & -0.05 & 0.084 & 64 & Bayh        & IN & Dem & -0.13 & 0.104 \\ 
   65 & Manchin     & WV & Dem & -0.66 & 0.099 & 66 & Feingold    & WI & Dem & -0.92 & 0.115 \\ 
   67 & Lincoln     & AR & Dem & -0.96 & 0.119 & 68 & Mccaskill   & MO & Dem & -1.15 & 0.083 \\ 
   69 & Webb        & VA & Dem & -1.49 & 0.108 & 70 & Pryor       & AR & Dem & -1.63 & 0.094 \\ 
   71 & Lieberman   & CT & Dem & -1.68 & 0.113 & 72 & Heitkamp    & ND & Dem & -1.87 & 0.183 \\ 
   73 & Donnelly    & IN & Dem & -1.87 & 0.182 & 74 & Hagan       & NC & Dem & -1.90 & 0.100 \\ 
   75 & Byrd  & WV & Dem & -2.00 & 0.217 & 76 & Warner      & VA & Dem & -2.06 & 0.105 \\ 
   77 & Landrieu    & LA & Dem & -2.07 & 0.106 & 78 & Tester      & MT & Dem & -2.11 & 0.105 \\ 
   79 & Baucus      & MT & Dem & -2.11 & 0.112 & 80 & Bennet      & CO & Dem & -2.16 & 0.107 \\ 
   81 & Klobuchar   & MN & Dem & -2.26 & 0.109 & 82 & Conrad      & ND & Dem & -2.29 & 0.131 \\ 
   83 & King        & ME & Ind & -2.30 & 0.208 & 84 & Nelson      & FL & Dem & -2.32 & 0.112 \\ 
   85 & Kohl        & WI & Dem & -2.34 & 0.131 & 86 & Carper      & DE & Dem & -2.36 & 0.112 \\ 
   87 & Udall       & CO & Dem & -2.39 & 0.113 & 88 & Begich      & AK & Dem & -2.43 & 0.116 \\ 
   89 & Dorgan      & ND & Dem & -2.44 & 0.167 & 90 & Reid        & NV & Dem & -2.68 & 0.122 \\ 
   91 & Shaheen     & NH & Dem & -2.76 & 0.125 & 92 & Kaine       & VA & Dem & -2.80 & 0.246 \\ 
   93 & Casey       & PA & Dem & -2.83 & 0.127 & 94 & Cantwell    & WA & Dem & -2.84 & 0.127 \\ 
   95 & Coons       & DE & Dem & -2.84 & 0.170 & 96 & Specter     & PA & Dem & -2.84 & 0.222 \\ 
   97 & Walsh       & MT & Dem & -2.85 & 0.395 & 98 & Wyden       & OR & Dem & -2.97 & 0.132 \\ 
   99 & Bingaman    & NM & Dem & -3.03 & 0.155 & 100 & Johnson     & SD & Dem & -3.09 & 0.137 \\ 
   101 & Stabenow    & MI & Dem & -3.11 & 0.137 & 102 & Cowan       & MA & Dem & -3.19 & 0.439 \\ 
   103 & Merkley     & OR & Dem & -3.19 & 0.140 & 104 & Sanders     & VT & Ind & -3.23 & 0.143 \\ 
   105 & Feinstein   & CA & Dem & -3.24 & 0.143 & 106 & Kerry & MA & Dem & -3.25 & 0.165 \\ 
   107 & Kaufman     & DE & Dem & -3.28 & 0.219 & 108 & Murray      & WA & Dem & -3.29 & 0.143 \\ 
   109 & Heinrich    & NM & Dem & -3.30 & 0.290 & 110 & Menendez    & NJ & Dem & -3.32 & 0.144 \\ 
   111 & Inouye      & HI & Dem & -3.33 & 0.169 & 112 & Boxer       & CA & Dem & -3.35 & 0.148 \\ 
   113 & Dodd        & CT & Dem & -3.38 & 0.218 & 114 & Warren      & MA & Dem & -3.45 & 0.307 \\ 
   115 & Levin  & MI & Dem & -3.52 & 0.152 & 116 & Blumenthal  & CT & Dem & -3.52 & 0.214 \\ 
   117 & Kirk        & MA & Dem & -3.54 & 0.716 & 118 & Akaka       & HI & Dem & -3.54 & 0.174 \\ 
   119 & Franken     & MN & Dem & -3.55 & 0.166 & 120 & Rockefeller & WV & Dem & -3.56 & 0.161 \\ 
   121 & Mikulski    & MD & Dem & -3.60 & 0.158 & 122 & Leahy       & VT & Dem & -3.63 & 0.158 \\ 
   123 & Harkin      & IA & Dem & -3.64 & 0.158 & 124 & Lautenberg  & NJ & Dem & -3.65 & 0.179 \\ 
   125 & Schumer     & NY & Dem & -3.65 & 0.159 & 126 & Reed        & RI & Dem & -3.67 & 0.157 \\ 
   127 & Gillibrand  & NY & Dem & -3.67 & 0.158 & 128 & Murphy      & CT & Dem & -3.68 & 0.327 \\ 
   129 & Markey      & MA & Dem & -3.73 & 0.465 & 130 & Whitehouse  & RI & Dem & -3.74 & 0.163 \\ 
   131 & Cardin      & MD & Dem & -3.82 & 0.163 & 132 & Durbin      & IL & Dem & -3.83 & 0.164 \\ 
   133 & Udall       & NM & Dem & -3.85 & 0.165 & 134 & Brown       & OH & Dem & -3.89 & 0.168 \\ 
   135 & Baldwin     & WI & Dem & -3.90 & 0.352 & 136 & Booker      & NJ & Dem & -4.14 & 0.572 \\ 
   137 & Hirono      & HI & Dem & -4.17 & 0.383 & 138 & Burris      & IL & Dem & -4.43 & 0.297 \\ 
   139 & Schatz      & HI & Dem & -4.74 & 0.468 &  &  &  &  & &  \\ 
  \hline
\end{tabular}
\caption{Ranking of the top 63-139 most conservative senators predicted by the model. Dem and Ind represent the Democratic party and independent politician, respectively. The states are presented in their standard abbreviations.  $\hat{\theta}$ represents the conservativeness score of senators and s.e.$(\hat{\theta})$ is the standard error of the estimated conservativeness score. }
\label{table: ranking-63-139}
\end{table}

\end{document}